\documentclass[a4paper]{article}

\usepackage{amsmath, amsthm, amsfonts, amssymb, amscd, bm}
\usepackage{color}
\usepackage{thmtools}
\usepackage{thm-restate}
\usepackage{hyperref}
\usepackage{algorithmic}
\usepackage{enumerate}
\usepackage{todonotes}
\usepackage{float}
\newtheorem{theorem}{Theorem}[section]
\newtheorem*{theorem*}{Theorem}
\newtheorem{corollary}[theorem]{Corollary}
\newtheorem{lemma}[theorem]{Lemma}
\newtheorem{proposition}[theorem]{Proposition}
\theoremstyle{definition}
\newtheorem{definition}[theorem]{Definition}
\newtheorem{assumption}[theorem]{Assumption}
\newtheorem{exmp}[theorem]{Example}
\numberwithin{equation}{section}
\DeclareMathOperator*{\esssup}{ess\,sup}
\DeclareMathOperator*{\essinf}{ess\,inf}

\DeclareMathOperator*{\argmin}{arg\,min}

\theoremstyle{remark}
\newtheorem{remark}{Remark}

\begin{document}
\title{Gittins' theorem under uncertainty}

\author{
	Samuel N. Cohen\footnote{Mathematical Institute, University of Oxford, Woodstock Rd, Oxford, OX2 6GG, UK, \texttt{samuel.cohen@maths.ox.ac.uk}}
	\and
	Tanut Treetanthiploet\footnote{Mathematical Institute, University of Oxford, Woodstock Rd, Oxford, OX2 6GG, UK, \texttt{tanut.treetanthiploet@maths.ox.ac.uk}}
}
\date{\today}


\maketitle
\begin{abstract}
We study dynamic allocation problems for discrete time multi-armed bandits under uncertainty, based on the the theory of nonlinear expectations. We show that, under independence assumption on the bandits and with some relaxation in the definition of optimality, a Gittins allocation index gives optimal choices. This involves studying the interaction of our uncertainty with controls which determine the filtration. We also run a simple numerical example which illustrates the interaction between the willingness to explore and uncertainty aversion of the agent when making decisions.

 

\end{abstract}

\section{Introduction}
\label{sec1}
When making decisions, people generally have a strict preference for options which they understand well. Since the classical work of Knight \cite{Knight} and Keynes \cite{Keynes}, there has been a stream of thinking within economics and statistics that focuses on the difference between the randomness of an outcome and lack of knowledge of its probability distribution (sometimes called `Knightian uncertainty'). This lack of knowledge is often related to estimation, as the probabilities used are often based on past observations.

This raises a natural question: how should we make decisions, given they will affect both our short-term outcomes, and the information available in the future? Shall we make a decision to explore and obtain new information, or shall we exploit the information available to optimize our profit? A simple setting in which this arises is a multi-armed bandit problem.

Modeling learning of the distribution of outcomes leads us to a paradox due to inconsistency in our decisions. As Keynes is said to have remarked\footnote{It appears this quote may be misattributed. One suggestion (discussed by John Kay \cite{Kay2015}) is that the correct attribution is to Paul Samuelson, and should read ``When my information changes, I alter my conclusions'', which fits even more easily with the thrust of this paper. }, ``When the facts change, I change my mind. What do you do, sir?''  The question a rational decision maker faces is, ``if I suspect that I will change my opinions or preferences tomorrow, how do I account for this today?''

In this paper, we use the theory of nonlinear expectations
(or equivalently risk measures) which are known to model Knightian uncertainty
(see  F\"ollmer and Schied \cite{stoc_fin}). This has been used to address statistical uncertainty, for example in  \cite{DR_original, Extreme_risk}. To achieve consistency in decision 
making, we usually have to consider a time-consistent nonlinear expectation. 
This is widely studied through backward stochastic 
differential equations (BSDEs) (see, for example, the work of Peng and others \cite{BSDE_and_application, 
	g-expectation, G_expectation, review_g-expectation}).
	
However, these approaches presume that the flow of information is not controlled. (Formally, the filtration of our agent is fixed and independent of their controls.) When we can control the observations which we will receive, this is not the case. In order to address this issue, while accounting for uncertainty, we discuss an alternative approach to deriving a time-consistent control problem, based on ideas from indifference pricing and the martingale optimality principle.  Using this approach, we show that when comparing different independent options, we can calculate an index separately for each alternative such that the `optimal' strategy is always to choose the option with the smallest index. This idea was initially proposed by Gittins and Jones \cite{Gittin_origin} (see also \cite{Gittins1979, Index_theory}) in a context where the probability measure is fixed but estimation (in a Bayesian perspective) is modeled by the evolution of a Markov process.

Given this result, we demonstrate a numerical solution in a simple setting. We 
shall see that our algorithm gives behaviour which is both optimistic and 
pessimistic in different regimes, and compares well with existing methods for multi-armed bandits.

\subsection{Multi-armed bandits}

Multi-armed bandits are a classical toy example with which to study decision making with randomness. They are commonly known to have applications in medical trials (Armitage \cite{medic_trial2} or Anscombe \cite{Medic_trial1}) and experimental design (Berry and Fristedt \cite{Experiment_design_1} or the classic paper of Robbins \cite{Experimental_design2}), along with other areas. A few recent works in finance for portfolio selection can also be found in Huo and Fu \cite{Finance_application} or Shen et al. \cite{Finance_application2}. The basic idea is that one has $M$ `bandits\footnote{`One-armed bandits' are an early variety of automated gambling system, the descendants of which are also known as slot machines, fruit machines or poker machines, depending on nationality.}', or equivalently, a bandit with $M$ arms, and one must choose which bandit should be played at each time. A key paper studying these systems, Gittins and Jones \cite{Gittin_origin}, argued that for a collection of independent bandits, each governed by a countable state Markov process, one could compute the ``Gittins index'' for each bandit separately, and the optimal strategy is to play the machine with the lowest index (or the highest, depending on the sign of gains/losses). The proof of this result has been obtained using a number of different perspectives, for example Weber's prevailing charge formulation \cite{Weber_proof} (which we consider in more detail below), Whittle's retirement option formulation \cite{Whittle_proof} and its extension without a Markov assumption by El Karoui and Karatzas \cite{Karatzas_El_Karoui_discrete} (and \cite{cts_Gittin} in continuous time). A review of the proofs in discrete time is given by Frostig and Weiss \cite{Survey_proof_Gittin}. However, in all these cases, the objective to be optimized is the discounted expected gain/loss -- in particular, we are assumed to have no risk-aversion or uncertainty-aversion. 

Gittins' index theory is commonly known as the first solution to an adaptive 
and sequential experimental design problem (from a Bayesian perspective) where the payoff of each bandit is assumed to be generated from a 
fixed unknown distribution\footnote{There are a few variations on these assumption e.g. adversarial bandits, contextual bandits or non-stationary bandits. Reviews of these 
	can be found in Burtini et al. 
	\cite{adaptive_learning_survey} and Zhou \cite{Survey_contextual_bandit}.} which must be inferred 
`on-the-fly', but where experimentation may be costly. As an alternative to Gittins' index, Agrawal \cite{Original_UCB} proposed the `Upper Confidence Bound (UCB) 
algorithm' which achieves a regret (deviation of average reward from the 
optimal reward) with the minimal asymptotic order of $\log(N)$, as proved by Lai 
and Robbins \cite{Asymp_efficient}. In the UCB algorithm, we compute a 
confidence interval for the expected reward at each step, and then play the 
bandit with the largest upper bound (where positive outcomes are preferred). 
Intuitively, using an upper bound encourages us to try bandits where we are 
less certain of the average reward, which encourages exploration. This is a 
form of `optimism' in decisions, which is counter-intuitive from the classical 
`pessimistic' utility theory (\emph{\`a la} von Neumann and Morgenstern 
\cite{VonNeumann1944}), where our preferences are for more certain outcomes. 

Typically, under appropriate assumptions, it is also the case that Gittins' index is a form of upper confidence bound for the estimated reward, an idea originally based on observations in Bather \cite{Bather_Gittins_asymptotic} and Kelly \cite{Asymp_Gittins_Frank_Kelly} and explored in more detail by Chang and Lai \cite{Asymptotic_summary}, followed up by Brezzi and Lai \cite{Brezzi_and_Lai_approx} (see Yao \cite{Asymp_Gittins_with_correction} for an error correction). Lattimore \cite{Regret_Gittins} also proves that Gittins' index achieves a minimal order bound on regret.

The apparent contradiction between the optimism of the UCB algorithm and  Gittins' index and the pessimism of classical utility theory is what led to this paper. We extend the notion of Gittins' index to a robust (nonlinear) operator, allowing for uncertainty aversion. We work in a generic discrete-time setting, allowing for the possibility of online learning, non-stationary and continuous outcomes, embedding all these effects in an abstract `nonlinear expectation'. (A concrete application to a simple setting with learning and uncertainty is given in Section 6.) In particular, we reformulate the proof of Gittins index theorem proposed by Weber \cite{Weber_proof}, as this proof relies the least on the linearity of the expectation, and gives a natural form of time-consistency. We also remove a Markov assumption in Weber's proof by adapting El Karoui and Karatzas' formulation \cite{Karatzas_El_Karoui_discrete}.  Our solution involves an optimal stopping problem under a nonlinear expectation, which can be converted to a low dimensional reflected BSDE (see for example El Karoui et al. \cite{cts_reflect_BSDE} and Cheng and Riedel \cite{Nonlinear_stopping} in continuous time or An, Cohen and Ji \cite{discrete_reflect_BSDE} in discrete time). This allows us to see a balance between the desire to explore and to exploit in our decision making.

The robust version of Gittins index has some correlation to the adversarial 
bandit problem  (see, for example, Auer et al. \cite{Adversarial_bandit_original, adversarial_bandit_improvement}) where we are playing the bandit against an adversary. Our theory proposes an `optimal' deterministic strategy
(no additional randomness is introduced at the decision time) against an adversary who tries to maximize 
our cost, which is slightly different from the known random algorithms for the adversarial bandit problem. The key difference is that, in the classical 
adversarial problem, an adversary is trying 
to maximize our `regret' whereas in our setting, we view the adversary as 
trying to maximize our cost. In our setting, the adversary is also permitted to respond to our current controls at every time, and we do not assume a minimax theorem holds.

The study of an adversary for the payoff in the bandit problem (via Gittins index) has been considered by 
Caro and Gupta 
\cite{finite_robust_gittins} and Kim and Lim \cite{Robust_gittin_alternnative} 
(with additional penalty in the reward) using Whittle's retirement option 
argument 
\cite{Whittle_proof}. In their works, they rely heavily on a Markov assumption, which allows them to postulate a robust dynamic programming principle (see also Iyengar 
\cite{Robust_DPP}, Nilim and El Ghaoui \cite{Robust_Markov_DPP}). Their 
formulation considers the robust Gittins' strategy as a promising solution due to its optimality for a single bandit, but they do not show optimality for multiple bandits. Furthermore, their Markov assumption restricts them to have a fixed uncertainty at all times and, therefore, it is not clear how to incorporate learning in their model. In contrast, our framework pays more attention to defining a good notion of dynamic optimality for our nonlinear expectation without any Markov assumption. 

By encoding learning through nonlinear expectations, a wide range of modeling options are included in our approach. For example, statistical concerns are treated in this framework in \cite{DR_original, 
uncertainty_filtering} or Bielecki, Cialenco and Chen \cite{
Bielecki2017}. We could also allow adversarial choices with a range of a fixed set (as in the classical adversarial bandit problem \cite{Adversarial_bandit_original} or as in \cite{finite_robust_gittins, Robust_gittin_alternnative}) or a random 
set which can be used to model learning as in the 
classical Gittins' theory. We also allow dynamic adversaries, which are not considered in the usual adversarial setting.

The paper proceeds as follows: In Section \ref{related}, we present some relevant existing approaches to multi-armed bandits, which we will adapt and combine to obtain our result. In Section \ref{uncertainty}, we give the required definitions for the nonlinear expectations that we use to evaluate our decisions. We also discuss the different notions of optimality which are available, and how they interact with the dynamic programming principle.

 In Section \ref{Overview section}, we give a summary of how we apply these expectations to a multi-armed bandit problem, state the key result, and give a sketch outline of the proof. The full details of this (rather technical) proof are given in two appendices: Appendix \ref{sec:partAproof} works through the first half of the proof, giving careful analysis of an optimal stopping problem under nonlinear expectation, and the corresponding `fair value' process, for a single bandit; Appendix \ref{Multiarm section} gives the second half of the proof, and demonstrates that the single bandit analysis yields an optimal strategy when deciding between multiple bandits. Further technical lemmas, which are used but do not contributed significantly to the main proof, are given in Appendix \ref{Append: additional result}.

Section \ref{sec:numerics} considers a simple example of a multiple bandit problem numerically, suggesting some connections with behavioural finance; the algorithm used to compute this example is given in Appendix \ref{app:numerical}.

\section{Problem formulation and related approaches}\label{related}
\subsection{General Problem Formulation}
Broadly speaking, Gittins \cite{Gittin_origin} argues that in order to dynamically allocate a single resource amongst several alternative projects, the optimal policy is to play at each point a bandit of lowest ``Gittins' index''. This index can be computed separately for each bandit, by solving an optimal stopping problem.

The subtlety in the proof of Gittins' theorem is to give a tractable representation of the class of control policies available to the decision maker. In the original formulation (see, for example, \cite{Gittin_origin, Weber_proof, Whittle_proof, Survey_proof_Gittin}), the class considered is feedback controls, as in a standard Markovian stochastic control problem; i.e. the system of bandits is modelled as a single Markov process, and the controls alter its transition probabilities. This formulation is restrictive, as it is not clear how it can be applied to a non-Markovian framework. Furthermore, as the control determines the filtration observed, it is also not clear how to introduce a general form of uncertainty aversion in this framework.

El Karoui and Karatzas \cite{Karatzas_El_Karoui_discrete} extend the argument of Gittins' theorem to the general case, without a Markovian assumption, by using Mandelbaum's \cite{Mandelbaum_discrete_allocation} ``allocation strategy'' formulation of the class of control policies. In particular, they view the cost of the bandits as a fixed process. The effect of allocation is to delay the realization of these fixed costs, which results in a benefit to the decision maker due to the time-value of money. 

In this paper, we will use a slight modification of Mandelbaum's allocation strategy to describe optimal strategies using a robust Gittins' theorem (without a Markovian assumption) via the classical argument given by Weber \cite{Weber_proof}, but using El Karoui and Karatzas' \cite{Karatzas_El_Karoui_discrete} formulation. 

\begin{remark}
	In most of the literature on Gittins' theorem, maximization of rewards is usually considered. For convenience, as is common in the theory of nonlinear expectations, we will consider the minimization of costs instead. Our presentation of others' results is done with the corresponding changes in sign.
\end{remark}

\begin{assumption}
	\label{model_assumption}
	Suppose that we have $M$ bandits. The $m$th bandit is associated with a filtered probability space $\big(\Omega^{(m)}, \mathbb{P}^{(m)}, (\mathcal{F}^{(m)}_t)_{t \geq 0}\big)$. Playing this bandit for the $t$th time realizes a non-negative bounded cost $h^{(m)}(t)$, where the process $\big(h^{(m)}(t)\big)_{t\geq1}$ is adapted to $(\mathcal{F}^{(m)}_{t})_{t\ge 0}$. We assume that  $\mathcal{F}^{(m)}_0 = \{\phi, \Omega^{(m)}\}$. 

The goal of the decision maker is to minimize the discounted total cost, for a given discount factor $\beta\in (0,1)$, when they can choose the order in which bandits are played.
\end{assumption}

Before considering a robust approach, we first outline the solution to this problem in a standard setting of classical expectation. 

\begin{definition}
	The \textbf{Gittins index} at time $s \geq 0$ of the $m$th bandit is given by
	$$\gamma^{(m)}(s) = \essinf_{\tau \in \mathcal{T}^{(m)}(s)} \frac{\mathbb{E}^{\mathbb{P}^{(m)}}\big( \sum_{t=1}^{\tau} \beta^{t} h^{(m)}(s+t) \big| \mathcal{F}^{(m)}_s\big)}{\mathbb{E}^{\mathbb{P}^{(m)}}\big( \sum_{t=1}^{\tau} \beta^{t} \big| \mathcal{F}^{(m)}_s\big)} $$
	where $\mathcal{T}^{(m)}(s)$ is the space of positive 
	$(\mathcal{F}^{(m)}_{s+t})_{t \geq 0}$-stopping times\footnote{Equivalently, for $\tau \in \mathcal{T}^{(m)}(s)$, $s+\tau$ is an $(\mathcal{F}^{(m)}_{t})_{t \geq 0}$-stopping time.} and the essential infimum is taken in $L^\infty(\mathcal{F}_s^{(m)})$. 
\end{definition}

\begin{definition}
	\label{orthant space}
	We define the \emph{orthant probability space} $\left(\bar{\Omega}, \bar{\mathbb{P}}, \left(\mathcal{F}(s)\right)_{s \in \mathcal{S}}\right)$ by 
	$$\bar{\Omega} := \prod_{m=1}^M \Omega^{(m)}, \quad \bar{\mathbb{P}} := \bigotimes_{m=1}^M \mathbb{P}^{(m)}, \quad \mathcal{F}(s) := \bigotimes_{m=1}^M \mathcal{F}^{(m)}_{s^{(m)}} : s = (s^{(1)},...,s^{(M)}) \in \mathcal{S}$$
	where $\mathcal{S}:= \mathbb{N}_0^M$.  We write $\mathcal{F}(\infty) := \bigotimes_{m=1}^\infty \mathcal{F}^{(m)}_\infty$.  
	
	We call $\left(\mathcal{F}(s)\right)_{s \in \mathcal{S}}$, the \textit{orthant filtration}.  
	
	To describe a useful set of stopping times in this (multi-indexed) filtration, let 
	\[\mathcal{T}^{(m) } := \left\{ (\mathcal{F}^{(m)}_t)_{t\ge0}\text{-stopping times} \right\}\]
	and
	\[\mathfrak{T}(\mathcal{S}) := \left\{S = (S^{(1)},...,S^{(M)}) \quad : \quad S^{(m)} \in \mathcal{T}^{(m)} \right\}.\]
	For $S \in \mathfrak{T}(\mathcal{S})$, we write
	$\mathcal{F}(S) := \bigotimes_{m=1}^M \mathcal{F}^{(m)}_{S^{(m)}}\quad \text{for } S \in \mathfrak{T}(\mathcal{S}).$
	
\end{definition}

\subsubsection{Classical Gittins Theorem}
Our policies will be described by a (random) path in the space $\mathcal{S}$, which indicates how many times each bandit has been played.

\begin{definition}[Mandelbaum \cite{Mandelbaum_discrete_allocation}]
	\label{Mandelbaum definition}
	The \textbf{Mandelbaum allocation strategy} is an $\mathcal{S}$-valued random sequence $\big(\tilde{\eta}(n)\big)_{n \geq 0}$ such that
	\begin{enumerate}[(i)] 
		\item $\tilde{\eta}(0) = 0$
		\item $\tilde{\eta}(n+1) = \tilde{\eta}(n) + e^{(m)}$ for some $m \in \mathcal{M} =: \{1,...,M\}$.
		\item $\{\tilde{\eta}(n+1) = \tilde{\eta}(n) + e^{(m)}, \; \tilde{\eta}(n) = r\} \in \mathcal{F}(r)$ for all $m \in \mathcal{M} := \{1,...,M\}$ and for all $r \in \mathcal{S}$.
	\end{enumerate}
Here, $e^{(m)}$ denotes the $m$th unit vector in $\mathcal{S}$. We denote by $\mathcal{A}$ the family of all Mandelbaum allocation strategy.
\end{definition}

\begin{theorem}[Gittins' theorem, as proved by El Karoui and Karatzas \cite{Karatzas_El_Karoui_discrete}]
\label{thm: El Karoui and Karatzas Gittins theorem}
	Let $\big(\tilde{\eta}^*(n)\big)_{n \geq 0}$ be a Mandelbaum allocation strategy $\big(\tilde{\eta}^*(n)\big)_{n \geq 0}$ such that 
	$$\gamma^{(m)}(\tilde{\eta}^{(m)}(n)) = \min_{k \in \mathcal{M}}\gamma^{(k)}(\tilde{\eta}^{(k)}(n)) \qquad \text{on each event} \quad \{\tilde{\eta}(n+1) = \tilde{\eta}(n) + e^{(m)}\}$$
	for all $m \in \mathcal{M}$ and $n \geq 0$. Then $\big(\tilde{\eta}^*(n)\big)_{n \geq 0}$ is an optimal solution to the optimisation problem
	\begin{equation}
	\label{El Karoui and Karatzas Objective}
	\inf_{\tilde{\eta} \in \mathcal{A}}\mathbb{E}^{\bar{\mathbb{P}}}\Big(\sum_{n=1}^{\infty}\sum_{m \in \mathcal{M}}\beta^n h^{(m)}(\tilde{\eta}^{(m)}(n)) \big[\tilde{\eta}^{(m)}(n) - \tilde{\eta}^{(m)}(n-1)\big]\Big).
	\end{equation}
\end{theorem}
In particular, Theorem \ref{thm: El Karoui and Karatzas Gittins theorem} says that the strategy which always plays the bandit with the minimum index minimizes the expectation of the total discounted cost.
\begin{remark}
	\label{Mandelbaum and simple seq}
	A Mandelbaum allocation strategy $\big(\tilde{\eta}(n)\big)_{n \geq 0}$ can also be represented by its increments, in particular, by a sequence of decision variables $(\rho_n)_{n \geq 0}$ taking values in $\mathcal{M}$. In other words,  we can define $(\rho_n)_{n \geq 0}$ such that $\{\rho_n = m \} = \{ \tilde{\eta}(n+1) = \tilde{\eta}(n) + e^{(m)} \}$. We may then replace the objective equation \eqref{El Karoui and Karatzas Objective} by
	\begin{equation}
	\label{Alternative El Karoui and Karatzas Objective}
	\inf_{\tilde{\eta} \in \mathcal{A}}\mathbb{E}^{\bar{\mathbb{P}}}\Big(\sum_{n=1}^{\infty}\beta^n h^{(\rho_{n-1})}(t^{\rho}_n)\Big) \;\;\; \text{where} \;\;\; t^{\rho}_n := \sum_{k=0}^{n-1} \mathbb{I}(\rho_k = \rho_{n-1}).
	\end{equation}
\end{remark}
\begin{remark}
Our paper considers the orthant filtration as the product of filtrations defined on different spaces. This is slightly different from
Mandelbaum \cite{Mandelbaum_discrete_allocation} (and thus El Karoui and Karatzas \cite{Karatzas_El_Karoui_discrete}) where the orthant filtration is considered as the join of filtrations defined on the same space. This technical difference will allow us to more easily define a `Nonlinear expectation' which still carries some form of independence and `time-consistency'. (See discussion in Section \ref{sec: join nonlinear}.)
\end{remark}
\begin{remark}
	In El Karoui and Karatzas \cite{Karatzas_El_Karoui_discrete}, it is assumed that the cost process is predictable, instead of adapted, with respect to the filtration of the bandit. When using a classical expectation, there is no modelling difference between predictable and adapted cost processes  (as one can just take the conditional expectation to reduce adapted costs to predictable costs). However, under a `nonlinear expectation', this is not the case, so we give the more general result with adapted costs. 
\end{remark}

\subsubsection{Robust Gittins Index}
\label{subsec: robust Gittins finite}
Under a Markovian assumption, Caro and Gupta \cite{finite_robust_gittins} consider a `robust' Gittins index based on the Robust Bellman equation studied in Iyergar \cite{Robust_DPP} and Nilim and El Ghaoui \cite{Robust_Markov_DPP}. (Similar work is considered by Kim and Lim \cite{Robust_gittin_alternnative} with an additional penalty in the formulation.)

The following assumptions are used in Caro and Gupta \cite{finite_robust_gittins} (translated into our notation):
\begin{enumerate}[(i)]
	\item The cost process is driven by some underlying finite-state process $(X^{(m)}_t)_{t \geq 0}$ taking values in $\mathcal{X}^{(m)}$. i.e. we have $h^{(m)}(t) = \tilde{h}^{(m)}(X^{(m)}_t)$ for some deterministic function $\tilde{h}^{(m)}$.
	\item Ambiguity is described by families of transition matrices, $(\mathcal{U}^{(m)})_{m \in \mathcal{M}}$ for the dynamics of $X^{(m)}$, which may vary in time.
\end{enumerate}
The construction is then based on Whittle indexibility \cite{Whittle_index}. In particular, they reduce the problem to considering two bandits, where one bandit always generates a constant cost $\gamma$ and the other bandit is identical to the $m$th bandit. The worst-case expected cost obtained when starting in state $i$ in the $m$th bandit, $V^{(m)}(i)$, allowing any combination of transition rates, will then satisfy the robust dynamic programming principle, that is, 
\begin{equation}\label{eq:robustDPP}V^{(m)}(i) = \min\bigg(\tilde{h}^{(m)}(i) + \beta \sup_{P \in \mathcal{U}^{(m)}}\sum_{j \in \mathcal{X}^{(m)}}P_{ij} V^{(m)}(j), \quad \frac{\gamma}{1-\beta}\bigg), \qquad i \in \mathcal{X}^{(m)}.\end{equation}

Let $D^{(m)}(\gamma) \subseteq \mathcal{X}^{(m)}$ be the set of states for which it is optimal to rest the $m$th bandit when the reward of the constant bandit is $\gamma$. Caro and Gupta show that the robust bandit is Whittle indexible in the sense that $D^{(m)}(\gamma)$ increases monotonically from $\phi$ to $\mathcal{X}^{(m)}$ as $\gamma$ increases from $-\infty$ to $+\infty$. The index of the $m$th bandit at state $i$ is the unique value $\gamma$ such that the player is indifferent between playing the $m$th bandit and the constant bandit.

This index can be characterized by
\begin{equation}
\label{Caro_Gupta_robust_Gittins}
\tilde{\gamma}^{(m)}(i) = \inf_{\tau \in \mathcal{T}^{(m)}, \tau \geq 1} \;\;\sup_{{\mathbb{Q}} \in \mathcal{Q}^{(m)}} \frac{\mathbb{E}^{\mathbb{Q}}\big( \sum_{t=1}^{\tau} \beta^{t} \tilde{h}^{(m)}(X^{(m)}_t) \big| X^{(m)}_0 = i\big)}{\mathbb{E}^{\mathbb{Q}}\big( \sum_{t=1}^{\tau} \beta^{t} \big| X^{(m)}_0 = i\big)}.
\end{equation}
where $\mathcal{Q}^{(m)}$ is the family of measures corresponding to the family of transition matrices $\mathcal{U}^{(m)}$.

Unfortunately, as discussed in Caro and Gupta \cite{finite_robust_gittins}, the robust Gittins index \eqref{Caro_Gupta_robust_Gittins} does not yield a strategy optimizing the robust Bellman equation
\begin{equation*}
V(i_1,...,i_K) = \min_{k \in [K]}\bigg(\tilde{h}^{(k)}(i_k) + \beta \sup_{P \in \mathcal{U}^{(k)}}\sum_{j \in \mathcal{X}^{(k)}}P_{i_k j} V(i_1,...,i_{k-1}, j, i_{k+1},...,i_K)\bigg)
\end{equation*}
where $(i_1,...,i_K) \in \prod_{m=1}^M \mathcal{X}^{(m)}$.

In short, this non-optimality arises due to the fact that the robust Bellman equation introduces dependency between bandits. In particular, at equilibrium, the adversary (who determines the transition probabilities for each bandit) may choose differently depending on the state of \emph{all} bandits, rather than just the bandit of interest. 

The index \eqref{Caro_Gupta_robust_Gittins} also can be interpreted as a Lagrangian relaxation of the optimal control problem (see also Gocgun and Ghate \cite{Lagrangian_relaxation}). The natural question that arises is,  `Does this relaxation satisfy some adjusted notion of optimality?'

In this paper, we propose a new form of optimality in terms of compensators of the value function. This can be seen as a relaxation of the dynamic programming principle through the martingale optimality principle,  in order to address a control problem under an inconsistent nonlinear operator. We will show that the strategy given by robust Gittins index satisfies this optimality criteria. We also allow the cost to be continuous valued and non-Markovian as in El Karoui and Karatzas \cite{Karatzas_El_Karoui_discrete}. This allows the study of various numerical methods to estimate our probabilistic state in the learning problem, whereas the numerical method in Caro and Gupta \cite{finite_robust_gittins} is limited to finite state Markov process. A simple numerical example then allows us to observe some qualitative peculiarities given the interaction between uncertainty aversion and learning.
\begin{remark}
In a non-Markovian framework, Whittle indexibility is not well-defined. Hence, the interpretation of optimality is required to understand a solution to the multi-armed bandit problem under uncertainty aversion.
\end{remark}
\begin{remark}
	Li \cite{Bandit_multiple_prior} considers a Bayesian formulation for the index but allowing for multiple priors. The focus  is on describing how the set of uncertainty affects the index, but without proving any form of Whittle indexibility. Our models also verify and generalize these results.
\end{remark}

\section{Uncertainty, Nonlinear Expectation and Optimality}
\label{uncertainty}
In this section, we will outline how `nonlinear expectation' operators can be used to model Knightian uncertainty. We will also discuss how we can use these tools to study a control problem, under uncertainty, while retaining some form of time consistency. We will build on the modelling framework of El Karoui and Karatzas \cite{Karatzas_El_Karoui_discrete} as proposed in Assumption \ref{model_assumption}.

We will first outline our setup and the additional assumptions we use in our study of the robust bandit problem. We will use a `nonlinear expectation' $\mathcal{E}^{(m)}$ (Assumption \ref{assump:existence of nonlinear}) to model uncertainty on the space $\big(\Omega^{(m)}, \mathbb{P}^{(m)}, (\mathcal{F}^{(m)}_t)_{t \geq 0}\big)$ of a single bandit, and then extend our uncertainty to the orthant joint space (Definition \ref{orthant space}) via the combined nonlinear expectation $\mathfrak{E}$ (Definition \ref{orthant expectation}). We will omit the superscript $(m)$ when it is clear from context.

In order to avoid technical difficulties, we will make the following assumption on the cost processes.
\begin{assumption}
	\label{assumption for cost process}
	For each $m \in \mathcal{M}$, there exists $C^{(m)}<\infty$ such that  
	$$0 \leq h^{(m)}(t) \leq C^{(m)} \qquad \text{and} \qquad  h^{(m)}(t) \to C^{(m)} \;\; \text{as} \;\; t \to \infty \;\; \mathbb{P} \text{-a.s.}$$
\end{assumption}
Assumption \ref{assumption for cost process} is purely technical. We may replace boundedness of $h^{(m)}$ by an integrability assumption on the total discounted cost (as in \cite{Karatzas_El_Karoui_discrete}); we then need to generalize the domain of the nonlinear expectation. We can also remove the assumption on the convergence of $h^{(m)}$ to its bound, but we then need to take more care to ensure that the stopping times we considered in \eqref{Caro_Gupta_robust_Gittins} and elsewhere can be assumed to be a.s. finite. Given the discount factor, this assumption does not have large impact on our modelling.

\subsection{Nonlinear Expectations and Time Consistency}
We now focus on the filtered probability space $\left(\Omega, \mathbb{P},\left(\mathcal{F}_t\right)_{t \geq 0}\right)$ modelling the returns from playing a single bandit. As in Peng \cite{g-expectation}, we define a nonlinear expectation as follows:

\begin{definition}
	\label{nonlinear expectation}
	A system of operators \[\mathcal{E}\left(\; \cdot \; | \mathcal{F}_t \right) : 
	L^\infty(\mathbb{P}, \mathcal{F}_{\infty}) \to L^\infty(\mathbb{P}, 
	\mathcal{F}_t)\] for $t \in \mathbb{T} := \{0,1, 2, ...\}$ is 
	said to be an 
	$\left(\mathcal{F}_t\right)_{t \geq 0}$-\emph{consistent coherent nonlinear expectation} if it satisfies the following properties: for $ X_n, X, Y \in L^\infty(\mathcal{F}_{\infty})$ and $c \in L^\infty(\mathcal{F}_t)$,  with all (in)equalities holding  $\mathbb{P}$-a.s, we have
	\begin{enumerate}[(i)]
	\item \emph{Strict Monotonicity}: If $X \geq Y$ then $\mathcal{E}(X | \mathcal{F}_t) \geq \mathcal{E}(Y | \mathcal{F}_t)$ 
	If, in addition, $\mathcal{E}(X| \mathcal{F}_t) =  \mathcal{E}(Y | \mathcal{F}_t)$,  then $X = Y$.
	\item $(\mathcal{F}_t)_{t\ge 0}$-\emph{Translation Equivariance}:  $\mathcal{E}( X + c | \mathcal{F}_t) = \mathcal{E}( X | \mathcal{F}_t) + c$.
	\item\emph{Subadditivity}: $\mathcal{E}(X + Y | \mathcal{F}_t) \leq \mathcal{E}(X | \mathcal{F}_t) + \mathcal{E}( Y | \mathcal{F}_t)$.
	\item $\left(\mathcal{F}_t\right)_{t\ge 0}$-\emph{Positive Homogeneity}: $\mathcal{E}(c X | \mathcal{F}_t) = c \;	\mathcal{E}(X | \mathcal{F}_t)$  if $c \geq 0$.
	\item \emph{Lebesgue property}: If $\{X_n\}_{n\in\mathbb{N}}$ is uniformly $\mathbb{P}$-a.s. bounded and $X_n \to X$ $\mathbb{P}$-a.s.~then $\mathcal{E}(X_n | \mathcal{F}_t) \to	\mathcal{E}(X | \mathcal{F}_t)$ $\mathbb{P}$-a.s.
\item \emph{$\left(\mathcal{F}_t\right)_{t \geq 0}$-consistency}: for $0 \leq s \leq t$,
		$\mathcal{E}(X | \mathcal{F}_s ) = 
		\mathcal{E}\big( 
		\mathcal{E}(X | \mathcal{F}_t ) \big| 
		\mathcal{F}_s
		\big)$.
	\end{enumerate}
We write $\mathcal{E}(\; \cdot \;)$ for $\mathcal{E}\big(\cdot \big| \mathcal{F}_0 
\big)$.
\end{definition}
\begin{remark}
For simplicity, we assume the Lebesgue property throughout this paper. In the static case, upper semi-continuity can be shown to be equivalent to the Lebesgue property over $L^\infty$ (see \cite[Corollary 4.38]{stoc_fin}). Moreover, if the operator $\mathcal{E}$ is induced by a BSDE (as in \cite{general_discrete_BSDE, discrete_BSDE_and_consistent_nonlinear, g-expectation, review_BSDE_and_finance} and many other papers), then the Lebesgue property typically follows from the $L^2$-continuous dependence of the BSDE on its terminal value.
\end{remark}
\begin{remark}
	\label{remark: regularity}
	It is also known (see e.g. Detlefsen and Scandolo \cite{Dynamic_risk}) that 
	any coherent nonlinear expectation satisfies the 
	\textit{$(\mathcal{F}_t)$-regularity} 
	property. That is, for any $X,Y \in 
	L^\infty(\mathcal{F}_{\infty})$ and $A \in 
	\mathcal{F}_t$,
	$$\mathcal{E}\big(X\mathbb{I}_A + Y\mathbb{I}_{A^c} \big| \mathcal{F}_t\big) = 
	\mathbb{I}_A\mathcal{E}\big(X \big| \mathcal{F}_t \big) + 
	\mathbb{I}_{A^c}\mathcal{E}\big(Y\big| \mathcal{F}_t \big).$$
	In particular, $\mathcal{E}\big(X\mathbb{I}_A \big| \mathcal{F}_t\big) = 
	\mathbb{I}_A\mathcal{E}\big(X \big| \mathcal{F}_t\big)$.
\end{remark}

In order to study decision making, we often require a conditional expectation defined at a stopping time. As we are working in discrete time, this is an easy construction.
\begin{definition}
	Given a consistent coherent nonlinear expectation $\mathcal{E}$ and a 
	stopping time $\tau \leq T$, we define the conditional expectation at $\tau$ by
	\begin{eqnarray*}
		\mathcal{E}\big( \cdot \big| \mathcal{F}_\tau \big) \; : 
		L^\infty(\mathcal{F}_\infty)  \longrightarrow  L^\infty(\mathcal{F}_\tau), \qquad		X  \longmapsto  \sum_{t=0}^{\infty}\mathbb{I}(\tau = t) \;	\mathcal{E}\big(X 
		\big| \mathcal{F}_t \big).
	\end{eqnarray*}	
\end{definition}
With this definition, the following easy observations can be made.
\begin{proposition}
	\label{Prop:consistence on stopping time}
	The operator $	\mathcal{E}\big( \cdot  \big| \mathcal{F}_\tau 
	\big)$ satisfies the conditions of Definition \ref{nonlinear expectation} with $s$ and $t$ are replaced by stopping times.
\end{proposition}
	
Nonlinear expectations are well suited to the study of Knightian uncertainty, that is, uncertainty over the probability measure. This is most easily seen through the robust representation theorem (over a finite horizon) given by Artzner et al. \cite{Artzner}, see also F\"ollmer and Schied \cite{stoc_fin} and Frittelli and Rosazza-Gianin \cite{Frittelli_and_Rosazza}. Extensions to a dynamic setting are also considered by Detlefsen and Scandolo \cite{Dynamic_risk}, F\"ollmer and Schied \cite{stoc_fin} and Riedel \cite{Independent_of_tau_for_risk_measure}. We state a version of this result which is dynamic over stopping times.

\begin{theorem}
	\label{Thm:dynamic robust rep}
		Let $\mathcal{E}$ be a consistent coherent nonlinear expectation.  If there exists $T < \infty$ such that $\mathcal{F}_\infty = \mathcal{F}_T$,  then $\mathcal{E}$ admits the representation 
	$$\mathcal{E}\big( \cdot \big| \mathcal{F}_\tau\big)  =
	\esssup_{\mathbb{Q} \in {\mathcal{Q}}} 
	\mathbb{E}^\mathbb{Q}\big( \cdot \big| \mathcal{F}_\tau\big)$$
	where $\tau$ is a stopping time and ${\mathcal{Q}} \subseteq 
	\left\{\mathbb{Q}:\mathbb{Q} \approx
	\mathbb{P} \right\}$, and the essential supremum is taken in $L^\infty(\mathcal{F}_\tau, \mathbb{P})$.
	\begin{proof}
		See F\"ollmer and Schied \cite[Theorem 11.22]{stoc_fin} (with further discussion in F\"ollmer and Penner \cite{sensitivity_ref}).  The sensitivity assumption assumed in these references (i.e. for every nonnegative nonconstant $X \in L^\infty(\mathcal{F}_T)$, there exists $\lambda > 0$ such that $\mathcal{E}\big(\lambda X \big) > 0$) follows from strict monotonicity in our definition.
	\end{proof}
\end{theorem}
\begin{remark}
	Theorem \ref{Thm:dynamic robust rep} can be obtained by construction via considering the stability of the pasting in the family $\mathcal{Q}$. (See e.g. Bion-Nadal \cite{cocycle_condition} and Artzner et al. \cite{Artzner_risk_measure_independent_tau}).
\end{remark}
\subsection{Uncertainty on multiple bandits}
\label{sec: join nonlinear}
In the classical Gittins theorem, independence is crucial 
to separate the behaviour of different bandits. In the robust representation (Theorem \ref{Thm:dynamic robust rep}) we have seen that a nonlinear expectation can be viewed as the supremum of classical expectations over a family of probability measures. Therefore, the notion of independence between bandits becomes ambiguous, as statistical independence is based on the probability measure. Thanks to our explicit construction of the space (Definition \ref{orthant space}), we can explicitly construct a nonlinear expectation space where each bandit remains independent. 

\begin{remark}
	In \cite{G_expectation}, Peng proposed a definition of independence for a nonlinear expectation. In his approach, independence is not a symmetric relation, but typically describes independence based on the order of events: often  `$Y$ is independent of $X$' when $Y$ occurs after $X$. In the setting of multiple bandits, the order of events cannot be pre-identified, as it depends on the control chosen. Hence, it is not clear how to exploit the independence notion of \cite{G_expectation} in this setting.
\end{remark}
Let us make the last universal assumption in our paper, which describes model uncertainty for each individual bandit in our problem inspired by the robust representation (Theorem \ref{Thm:dynamic robust rep}). 
\begin{assumption}
	\label{assump:existence of nonlinear}
	For each $m \in \mathcal{M}$, we have a  $(\mathcal{F}^{(m)}_t)_{t\ge 0}$-consistent coherent nonlinear expectation, $(\mathcal{E}^{(m)}(\; \cdot \; | \mathcal{F}_\tau))_{\tau \in \mathcal{T}^{(m) }}$ defined on the space $L^\infty(\Omega^{(m)}, \mathbb{P}^{(m)})$ which admits the representation
$$\mathcal{E}^{(m)}\Big( \; \cdot \;  \Big| 
\mathcal{F}^{(m)}_{S^{(m)}}\Big) = \esssup_{\mathbb{Q} \in 
	\mathcal{Q}^{(m)}} \mathbb{E}^\mathbb{Q}\Big( \; \cdot \;
\Big| \mathcal{F}^{(m)}_{S^{(m)}}\Big)$$
whenever ${S}^{(m)}$ is an $(\mathcal{F}^{(m)}_t)$-stopping time.
\end{assumption}

\begin{definition}
	\label{orthant expectation}
	We define the \emph{partially consistent orthant nonlinear expectation} $\left(\mathfrak{E}_S\right)_{S \in \mathfrak{T}(\mathcal{S})}$,  to be the family of operators
	\begin{eqnarray*}
		\mathfrak{E}_S : L^\infty\left(\bar{\Omega}, \bar{\mathbb{P}}, \mathcal{F}(\infty) \right) & \longrightarrow & L^\infty\left(\bar{\Omega}, \bar{\mathbb{P}}, \mathcal{F}(S) \right)  \\
		X & \longmapsto & \esssup_{\mathbb{Q} \in 
			{\mathcal{Q}}}\mathbb{E}^{\mathbb{Q}}\big(X \big| 
		\mathcal{F}(S)\big)
	\end{eqnarray*}
	where, with $\mathcal{Q}^{(m)}$ as in Assumption \ref{assump:existence of nonlinear},
	$${\mathcal{Q}} := \bigg\{\bigotimes_{m=1}^M \mathbb{Q}^{(m)} \;\; \text{for} \;\; \mathbb{Q}^{(m)} \in \mathcal{Q}^{(m)}\bigg\}.$$
	We also write $\mathfrak{E}$ for $\mathfrak{E}_0$.
\end{definition}
\begin{remark}
	As $\bar{\mathbb{P}} = \bigotimes_{m=1}^M \mathbb{P}^{(m)}$ is a dominating measure for $\mathcal{Q}$, we easily observe that if $X = Y$ $\bar{\mathbb{P}}$-a.s., then $\mathfrak{E}_S(X) = \mathfrak{E}_S(Y)$ $\bar{\mathbb{P}}$-a.s. for all $S\in \mathfrak{T}(\mathcal{S})$.
\end{remark}

\begin{proposition}
	\label{property of joining expectation}
	The system of operators $(\mathfrak{E}_S)$ satisfies the following properties.
	\begin{enumerate}[(i)]
		\item The properties (i)-(v) in Definition \ref{nonlinear expectation}
 (with appropriate replacements on the operator and $\sigma$-algebra) hold for the operator $\mathfrak{E}_S$ $\; : S \in \mathfrak{T}(\mathcal{S})$ (i.e. strict monotonicity, translation equivariance, subadditivity, positive homogeneity and the Lebesgue property hold for $\mathfrak{E}$).		
		\item Sub-consistency: For $S, S' \in \mathfrak{T}(\mathcal{S})$ with $S \leq S'$, we have
		$$\mathfrak{E}_S \left( \; \cdot \; \right) \leq \mathfrak{E}_S 
		\left(\mathfrak{E}_{S'} \left( \; \cdot \; \right) \right) \;\;\; \bar{\mathbb{P}} \text{-a.s.}$$
		In particular, for any measurable $X$, if $\mathfrak{E}_{S'}(X)\leq 0$ $\bar{\mathbb{P}}$-a.s., then for any $A\in \mathcal{F}_{S'}$ we have $\mathfrak{E}_{S}(\mathbb{I}_A X)\leq 0$ $\bar{\mathbb{P}}$-a.s.
		\item Independence: Let $Y$ be a random variable on $(\bar{\Omega}, \mathcal{F}(\infty))$ given by
		$$Y(\omega^{(1)},...,\omega^{(M)}) = 
		X^{(1)}(\omega^{(1)}) \times \cdots \times X^{(M)}(\omega^{(M)}) \qquad \bar{\mathbb{P}}\text{-a.s.},$$
		where, for each $m\in \mathcal{M}$, we have a non-negative random variable $X^{(m)}$  defined on $(\Omega^{(m)}, \mathcal{F}^{(m)}_\infty)$. Then 
		$$\mathfrak{E}_S(Y) = \mathcal{E}^{(1)}\left( X^{(1)} \Big| 
		\mathcal{F}^{(1)}_{S^{(1)}}\right)\times \cdots \times\mathcal{E}^{(M)}\left( X^{(M)} 
		\Big| \mathcal{F}^{(M)}_{S^{(M)}}\right) \;\;\; \bar{\mathbb{P}} \text{-a.s.}$$
		\item Marginal projection: For a given $m \in \mathcal{M}$, let $X$ be a random variable defined on $(\Omega^{(m)}, 
		\mathcal{F}^{(m)}_{\infty})$. Define $\tilde{X} : \bar{\Omega} \to \mathbb{R}$ by
		$\tilde{X}(\omega^{(1)},...,\omega^{(M)}) = 
		X(\omega^{(m)})$.
		We then have
		$$\mathfrak{E}_S(\tilde{X}) = \mathcal{E}^{(m)}\big( X 
		\big| \mathcal{F}^{(m)}_{S^{(m)}}\big) \;\;\; \bar{\mathbb{P}} \text{-a.s.}$$	
	\end{enumerate}
	\begin{proof}
		See Proposition \ref{property of joining expectation: proof} in the appendix .
	\end{proof}
\end{proposition}

\begin{remark}
\label{rem: nonlinear expectation construction}
We deliberately choose our nonlinear expectation $\mathfrak{E}$ to be defined on a product space to simplify our discussion on the existence of the operator.  In fact, one can simply weaken our assumption by having a nonlinear expectation $\mathfrak{E}$ on a joint filtration (as in El Karoui and Karatzas \cite{Karatzas_El_Karoui_discrete}) such that the Proposition \ref{property of joining expectation} holds. All proofs are identical except the proof of Theorem \ref{positivity} (in the Appendix). We just need an extra step to show that the product of the marginal probability measure is also a probability measure considered under the robust representation of $\mathfrak{E}$. 
\end{remark}
In the proposition above, we have seen that $\mathfrak{E}$ is sub-consistent  on the orthant filtration. However,  $\mathfrak{E}$ is \emph{not} consistent in the sense of Definition \ref{nonlinear expectation}, i.e. if $S \leq S'$ (componentwise), it is not necessarily the case that $\mathfrak{E}_S(\; \cdot \;) = \mathfrak{E}_S \left( \mathfrak{E}_{S'}(\; \cdot \;) \right)$. A counterexample can be easily constructed based on the following:

\begin{exmp}
	\label{independence exmp}
	Let $X$ and $\tilde{X}$ be random variables taking values in $\{0,1\}$ and defined on different spaces $\Omega$ and $\tilde{\Omega}$. Let $\mathcal{Q}$ and $\tilde{\mathcal{Q}}$ be families of probability measures defined on these spaces. Suppose that for all $p \in [0,1]$ there exists $\mathbb{Q} \in \mathcal{Q}$ such that $\mathbb{Q}(X=0) = p$ and that for all  $\tilde{\mathbb{Q}} \in \tilde{\mathcal{Q}}$, $\tilde{\mathbb{Q}}(\tilde{X} = 0) = 1/2.$ Let $f: \mathbb{R}^2 \to \mathbb{R}$ be a given function. Then it is easy to show that
	\begin{align*}
	\sup_{\mathbb{Q},\tilde{\mathbb{Q}} } \mathbb{E}^{\mathbb{Q} \otimes \tilde{\mathbb{Q}}}\big(f(X,\tilde{X})\big) 
	& =  \sup_{{\mathbb{Q}}} \mathbb{E}^{{\mathbb{Q}}}\Big(\sup_{\tilde{\mathbb{Q}}} \mathbb{E}^{\tilde{\mathbb{Q}}}\big(f(x,\tilde{X})\big) \Big|_{{x} = {X}}\Big)  \\
	& =  \max \Big(\frac{f(0,0) + f(0,1)}{2}, \frac{f(1,0) + f(1,1)}{2}\Big)
	\end{align*}
	but
	\begin{align*}
	\sup_{\tilde{\mathbb{Q}}} \mathbb{E}^{\tilde{\mathbb{Q}}}\Big(\sup_{{\mathbb{Q}}} \mathbb{E}^{{\mathbb{Q}}}\big(f(X,\tilde{x})\big) \Big|_{\tilde{x} = \tilde{X}}\Big) = \frac{\max \{f(0,0),f(1,0)\}}{2} + \frac{\max\{f(0,1),f(1,1)\}}{2}.
	\end{align*}
	By considering $\mathcal{F}_1^{(1)} = \sigma(X)$ and $\mathcal{F}_1^{(2)} = \sigma(\tilde{X})$, and defining nonlinear expectation using supremum over the family $\mathcal{Q}$ and $\tilde{\mathcal{Q}}$, the above result shows that the joint operator $\mathfrak{E}$ is not consistent. In particular, we can find a function $f$ such that
	$$\mathfrak{E}\Big(\mathfrak{E}_{(0,1)}\Big(f(X, \tilde{X})\Big)\Big) \neq \mathfrak{E}\Big(\mathfrak{E}_{(1,0)}\Big(f(X, \tilde{X})\Big)\Big).$$
\end{exmp}
\subsection{Optimality}

We have discussed in the previous section that the robust Gittins index \eqref{Caro_Gupta_robust_Gittins} in the sense of Caro and Gupta \cite{finite_robust_gittins} is not optimal, as it does not lead to a solution of the robust Bellman equation  (discussed in \cite{Robust_DPP, Robust_Markov_DPP}). In order to understand what sense of optimality the robust index strategy \emph{does} satisfy, we will first consider a form of optimality criteria used by El Karoui and Karatzas \cite{Karatzas_El_Karoui_discrete}. 

Let us consider an abstract stochastic control problem on a space $(\bar{\Omega}, \bar{\mathcal{F}}, \bar{\mathbb{P}})$ in which a choice of control $\rho$ results in an instantaneous cost process $\big(g^\rho(n)\big)_{n \geq 1}$. We may view $g^\rho(n)$ as a cost occured at time $n$. For example, we have $g^\rho(n) := \beta^n h^{(\rho_{n-1})}(t^{\rho}_n)$ in \eqref{Alternative El Karoui and Karatzas Objective}. We can also define the filtration of information obtained up to time $n$ when following $\rho$ by 
\begin{equation}
\label{introduce decision filtration}
\mathcal{G}^\rho_n := \Big\{A \in \mathcal{F}(T) \; : \; A \cap \{\tilde{\eta}(n) = 
r\} \in \mathcal{F}(r) \;\;\; \forall r \in \mathcal{S} \Big\},
\end{equation}
where $\tilde{\eta}$ is the corresponding Mandelbaum allocation sequence (Definition \ref{Mandelbaum definition}, Remark \ref{Mandelbaum and simple seq}).
We will discuss this filtration in detail in Remarks \ref{simple form to Mandelbaum} and \ref{observed vs decision filtration}.

\begin{remark}
	It is clear from the definition that the strategy process $(\rho_n)_{n \geq 0}$ is $(\mathcal{G}^\rho_n)_{n \geq 0}$-adapted. We will show later that the cost process  $g^\rho(n) := \beta^n h^{(\rho_{n-1})}(t^{\rho}_n)$ (as in \eqref{Alternative El Karoui and Karatzas Objective}) is also adapted with respect to $\mathcal{G}^\rho_n$.
\end{remark}

Suppose that we are given a nonlinear expectation operator $\mathfrak{E}$, as in Definition \ref{orthant expectation}, and consider a minimization problem over the space of Mandelbaum allocation strategies, as represented by their equivalent form $\rho$ (Remark \ref{Mandelbaum and simple seq}). The process $\rho$ not only describes our strategy and the corresponding cost, but also determines the observed filtration. Therefore, at any point in time, it does not make sense to compare strategies unless those strategies yield the same information at the considered time. 

\begin{definition}
	We say strategies $\rho$ and $\rho'$ are \textit{historically equivalent} at time $N$, denoted by $\rho \sim_N \rho'$, if $\rho_n = \rho'_n$ for all $n \leq N$. 
\end{definition}

\begin{remark}
	For every strategy $\rho$, we have $\rho \sim_0 \rho^*$. 
\end{remark}
We can now give a standard form of optimality which is often considered when we have a consistent nonlinear expectation operator.
 
\begin{definition}
	\label{dynamic strong optimal}
	We say a strategy $\rho^*$ is a \emph{strong optimum} if for every strategy $\rho $ such that $\rho \sim_N \rho^*$, we have
	$$\mathfrak{E}\bigg(\mathbb{I}_A \bigg(\sum_{n=N+1}^\infty g^{\rho^*}(n)\bigg)\bigg) \leq \mathfrak{E}\bigg(\mathbb{I}_A \bigg(\sum_{n=N+1}^\infty g^{\rho}(n)\bigg)\bigg) \qquad \text{for all } A \in \mathcal{G}^\rho_N (= \mathcal{G}^{\rho^*}_N).$$
\end{definition}
\begin{remark}
When $\mathfrak{E}$ is replaced by an $(\mathcal{F}_n)$-consistent nonlinear expectation and $(\mathcal{G}^\rho_N)$ is replaced by $(\mathcal{F}_N)$,  strong optimality simplifies to
$$\mathcal{E}\bigg(\sum_{n=N+1}^\infty g^{\rho^*}(n)\bigg| \mathcal{F}_N\bigg) = \essinf_{\rho \sim_N \rho^*} \mathcal{E}\bigg(\sum_{n=N+1}^\infty g^{\rho}(n)\bigg| \mathcal{F}_N\bigg) .$$
\end{remark}

A standard approach to tackle the decision making under time-inconsistency (nonlinear expectation) operator is to define `the optimal strategy' through the solution of the robust Bellman equation \cite{Robust_DPP, Robust_Markov_DPP} as considered in Caro and Gupta \cite{finite_robust_gittins}.  Using the tower property,  we can show that the strong optimum under $(\mathcal{F}_n)$-consistent nonlinear expectation is equivalent to the solution to the robust Bellman equation. 

\subsection{C-Optimality}
In the bandit setting,  our nonlinear expectation is not necessary (time-)consistent. In order to understand the Gittins index strategy under an inconsistent operator, we propose an alternative notion of optimality, which is inspired by martingale optimality.

For motivation,  consider an $(\mathcal{F}_n)$-consistent nonlinear expectation $\mathcal{E}$. Suppose that we wish to solve the minimization problem
$$V_N = \essinf_{\rho} \mathcal{E}\bigg(\sum_{n=N+1}^\infty g^{\rho}(n)\bigg| \mathcal{F}_N\bigg).$$
For a given strategy $\rho$, we define a process 
$X^\rho_N := \sum_{n=1}^N g^{\rho}(n) + V_N.$
Under mild conditions, we know from the martingale optimality principle that $(X^\rho_N)$ is an $\mathcal{E}$-submartingale for every strategy $\rho$ and it is a martingale for an optimal strategy $\rho^*$.

By using the Doob--Meyer decomposition for nonlinear expectation (see e.g.  \cite[Theorem 8]{general_g-expectation}), we can write
 $$X^\rho_N := M^\rho_N + \sum_{n=1}^N C^\rho(n)$$
where $ (M^\rho_N)$ is an $\mathcal{E}$-martingale and $(C^\rho(n))$ is a non-negative predictable process with $C^\rho(n) \equiv 0$ for the optimal strategy $\rho^*$. 

By rearranging the equation above, for every $\rho$,
$$\mathcal{E}\bigg(\sum_{n=N+1}^\infty \Big(g^{\rho}(n) - C^\rho(n)\Big) \bigg| \mathcal{F}_N \bigg) = -V_N.$$
Moreover, for an optimal strategy $\rho^*$, we have
$$\sum_{n=N+1}^L  C^{\rho^*}(n) \leq \sum_{n=N+1}^L  C^{\rho}(n) \qquad \text{for all } N, L.$$

Inspired by the analysis above, we propose an alternative notion of optimality in an inconsistent setting. 
\begin{definition}
	\label{C-optimal} 
	We say a strategy $\rho^*$ is \textit{C-optimal} if there exists a $(\mathcal{G}^{\rho^*}_n)$-adapted process $(V_n)$ (called a \textit{value process}) and a collection of random variables
	$(C_N^\rho(n))_{N, n \geq N+1, \rho \sim_N \rho^*}$ (called a (sub-)compensator) such that
	\begin{enumerate}[(i)]
		\item $n \mapsto C_N^\rho(n)$ is a $(\mathcal{G}^\rho_n)$-predictable process,
		\item $N \mapsto C_N^\rho(n)$ is non-increasing,
		\item For every strategy $\rho \sim_N \rho^*$,
		\begin{equation}
		\label{eq: compensation}
		\mathfrak{E}\bigg(\mathbb{I}_A\sum_{n=N+1}^\infty \Big(g^{\rho}(n) - C_N^{\rho}(n)\Big)\bigg) \geq \mathfrak{E}\Big(-\mathbb{I}_A V_N\Big) \qquad \text{for all} \quad A \in \mathcal{G}^\rho_N
		\end{equation}
		with equality for $\rho =\rho^*$, 
		\item For every strategy $\rho \sim_N \rho^*$,
		\begin{equation}
	\label{C-optimal equation}
	\sum_{n=N+1}^L  C_N^{\rho^*}(n) \leq \sum_{n=N+1}^L  C_N^{\rho}(n) \qquad \text{for all } L \geq N +1.
		\end{equation}
	\end{enumerate}
\end{definition}

We can see $\big(C_N^\rho(n)\big)$ acts as `(sub-)compensator' to the cost, and $V_N$ acts as the value function.  This approach is loosely related to the capital requirement approach discussed by Frittelli and Scandolo \cite{capital_requirement}.  We can interpret Definition \ref{C-optimal} as requiring that the (sub-)compensators $(C_N^\rho(n))$
\begin{enumerate}[(i)]
\item is known one-step in advanced before observing the cost (i).
	\item consistently (sub-)compensate the cost. In particular, as time elapses, we obtain more information and thus require the same amount, or possibly less to (sub-)compensate (ii).
	\item complement the extra cost occurred for a sub-optimal strategy (iii).

	\item are bounded below by a compensator of a particular strategy $\rho^*$, which we call `optimal' (iv).
\end{enumerate}
\begin{remark}
	We have mentioned the robust Bellman equation \cite{Robust_DPP, Robust_Markov_DPP} as an approach to force time-consistency in our decision making. The fundamental idea of this approach is to freeze our value function and propagate its value backward in time. In particular, suppose we have $V^*_{n+1}$ as our expected remaining cost at time $n$. We then define an optimal strategy at time $n$ to be a strategy $\rho_n$ such that $g^\rho(n) +   V^*_{n+1}$ is optimized. 
	
A closely related approach to ensure time-consistency was proposed by Strotz \cite{Strotz_consistent} and Pollak \cite{Pollak_consistency} and developed further in Peleg and Yaari \cite{Peleg_yaari_consistent} and Koopmans \cite{Koopmans_consistency}. Recent extensions include Bj\"ork and Murgoci \cite{Bjoerk2014}, Bj\"ork, Khapko and Murgoci \cite{Bjoerk2017}, Yong \cite{Yong2012} and Hu, Jin and Zhou \cite{Hu2012}. For a problem with horizon $L$,  suppose that the optimal control is determined after time $n$,  in other words, $(\rho^*_{n+1},...,\rho^*_{L})$ is known. We then find a control $\rho^*_n$ at time $n$ to optimize over the space of possible strategies $\left\{\rho \;\; : \;\; (\rho_{n+1},...,\rho_{L}) = (\rho^*_{n+1},...,\rho^*_{L})\right\}$.  In this way, the (optimal) control rather than the value function, is constructed recursively. This idea is then extended by searching for (sub-game perfect) Nash equilibria, to allow for non-uniqueness of the optimal controls.
	
	As discussed in Section \ref{subsec: robust Gittins finite}, the robust Bellman approach may introduce some dependency between bandits in our system.  Hence,  Gittins index strategy is not optimal under that approach.  On the other hand, when considering a system of bandits, the measurability of our future states are determined by our current action. Therefore, the $\sigma$-algebra that is used to define the future control $(\rho^*_k)_{k \geq n+1}$ cannot be chosen independently of our current control. This means that we cannot directly consider the Strotz--Pollak approach for the bandit setting as we cannot freeze our future control without freezing our current control.
	
	The notion of C-optimality can be loosely interpreted as a third variation on these time-consistency approaches.  In particular, we can interpret the compensator process  as propagating a value backward in time, as in the robust Bellman approach. Optimality can then be defined forward in time, which relaxes the dependence on
the filtration.
\end{remark}

\subsection{Endowment Effect}
One natural question to ask is whether we can give an interpretation of C-optimality (Definition \ref{C-optimal}) in terms of classical strong optimality (Definition \ref{dynamic strong optimal}). To see this, we will consider an endowment effect through the strong optimality.
\begin{exmp}
	\label{unstability strong optimum}
	Let $H$ and $G$ be random variables representing the cost of two strategies and $\mathcal{Q}$ be a family of probability measures such that $H$ and $G$ are independent under each $\mathbb{Q} \in \mathcal{Q}$. Suppose $\{\mathbb{E}^\mathbb{Q}(H)\}_{\mathbb{Q}\in \mathcal{Q}} = [\underline{h}, \bar{h}]$ and similarly for $G$. Suppose further that $\bar{h} < \bar{g}$ but $\bar{h} - \underline{g} > \bar{g} - \underline{h}$. Then for	$\mathfrak{E}\left(\; \cdot \;\right) := \sup_{\mathbb{Q} \in \mathcal{Q}}\mathbb{E}^\mathbb{Q}\left(\; \cdot \;\right)$, we have
	\begin{equation}
	\label{endowment inequality}
	\mathfrak{E}(H) < \mathfrak{E}(G) \;\;\;\; \text{but} \;\;\;\; \mathfrak{E}\Big(H - \frac{H+G}{2}\Big) > \mathfrak{E}\Big(G - \frac{H+G}{2}\Big).
	\end{equation}
\end{exmp} 
From these inequalities, we see that, without any endowment, we strictly prefer $H$ to $G$ whereas our preference reverses with an endowment $(H+G)/2$. We know that in the classical linear expectation theory (where the classical Gittins theorem holds), an endowment does not affect our preference in the strategy. 

In this section, we will show that C-optimality is nearly equivalent to a strong optimality `up to an endowment' when our nonlinear expectation is time-consistent. 

The following proposition follows from the definition of C-optimality and monotonicity of nonlinear expectation (in particular, Definition \ref{dynamic strong optimal}(iii)-(iv)).
\begin{proposition}
	\label{Prop: C-optimal implies Endowment}
	Let $\rho^*$ be a C-optimal strategy with a predictable compensator $(C^{\rho^*}_N(n))$. Then for every $\rho \sim_N \rho^*$ and $A \in \mathcal{G}^{\rho^*}_N$,
	\begin{equation}
	\label{eq: C imply endow}
	\mathfrak{E}\bigg(\mathbb{I}_A\sum_{n=N+1}^\infty \big(g^{\rho^*}(n) - C^{\rho^*}_N(n) \big) \bigg) \leq \mathfrak{E}\bigg(\mathbb{I}_A \sum_{n=N+1}^\infty \big(g^{\rho}(n)  - C^{\rho^*}_N(n) \big) \bigg).
	\end{equation}
\end{proposition}

Let consider the case when $\mathcal{G}^{\rho}_N = \mathcal{F}_N$ for every strategy $\rho$ and pretend that $\mathfrak{E}$ is an $(\mathcal{F}_n)$-consistent nonlinear expectation operator. Then \eqref{eq: C imply endow} says that C-optimality implies strong optimality, when our agent is given the predictable endowment $-\sum_{n=N+1}^\infty C^{\rho^*}_N(n)$ at time $N$.  We will now show that a converse result also holds, when our operator is consistent.

\begin{definition}
	Let $\mathcal{E}$ be an $(\mathcal{F}_n)$-consistent nonlinear expectation. We say a strategy $\rho^*$ is \textit{optimal up to a predictable endowment} if there exists a family of random variables $(D_N(n))$ such that
	\begin{enumerate}[(i)]
		\item $n \mapsto D_N(n)$ is an $(\mathcal{F}_n)$-predictable process,
	\item $N \mapsto D_N(n)$ is non-increasing, 
	\item For every strategy $\rho \sim_N \rho^*$, for all $A \in \mathcal{G}^\rho_N$,
	\begin{equation}
	\label{eq: compensation_endowment}
	\mathcal{E}\bigg(\mathbb{I}_A\sum_{n=N+1}^\infty \Big(g^{\rho^*}(n) - D_N(n)\Big)\bigg) \leq 	\mathcal{E}\bigg(\mathbb{I}_A\sum_{n=N+1}^\infty \Big(g^{\rho^*}(n) - D_N(n)\Big)\bigg)
	\end{equation}
	or equivalently, 
	\begin{equation*}
	\mathcal{E}\bigg(\sum_{n=N+1}^\infty \Big(g^{\rho^*}(n) - D_N(n)\Big)\bigg|\mathcal{F}_N \bigg) \leq 	\mathcal{E}\bigg(\sum_{n=N+1}^\infty \Big(g^{\rho^*}(n) - D_N(n)\Big) \bigg|\mathcal{F}_N  \bigg)
	\end{equation*}
	\end{enumerate}
\end{definition} 

\begin{proposition}
\label{Prop: Endowment implies C-optimal}
Suppose that $\rho^*$ is an optimal strategy up to a predictable endowment, then $\rho^*$ is C-optimal.
\begin{proof}
	Take $C^{\rho}_N(n) = D_N(n)$ and $V_N = \mathcal{E}\big(\sum_{n=N+1}^\infty \big(g^{\rho^*}(n) - D_N(n)\big)\big|\mathcal{F}_N \big)$.
\end{proof}
\end{proposition}
In the coming section,  we will show that Gittins theorem holds in the sense of guaranteeing C-optimality under an operator $\mathfrak{E}$. This means that we prove that Gittins theorem is a (strong) optimum up to some predictable endowment.
\begin{remark}
	It is an open question under which conditions the C-optimum is unique. In the most trivial case when our operator $\mathfrak{E}$ is simply a classical expectation, the endowment never affects our evaluation; thus it is reduced to the uniqueness of the value function in the classical setting. 
\end{remark}
\section{Overview of Bandits under uncertainty}
\label{Overview section}
Let us recall that the objective of our problem is to dynamically allocate a single resource amongst $M$ bandits to minimize the total discounted cost. We have made a few assumptions to model uncertainty in the cost process which can be founded in Assumptions \ref{model_assumption}, \ref{assumption for cost process} and \ref{assump:existence of nonlinear}.

We also introduce a Mandelbaum allocation strategy (Definition \ref{Mandelbaum definition}) and the equivalent notion $\rho$ (Remark \ref{Mandelbaum and simple seq}) representing the choice of our control. We are now ready to establish a robust Gittins theorem with optimality in the sense of Definition \ref{C-optimal}. Our robust Gittins theorem generalize the result of El Karoui and Karatzas \cite{Karatzas_El_Karoui_discrete} to the uncertain case. One may also see this result as providing a sense of optimality for the index strategy considered by Caro and Gupta \cite{finite_robust_gittins} and Li \cite{Bandit_multiple_prior}.

\subsection{Robust Gittins theorem}
We will first give an alternative definition to the robust Gittins' index inspired by Weber \cite{Weber_proof}, which is more convenient to use in our analysis.
\begin{definition}
	\label{Gittins index}
	For each $s \geq 0$, we define the 
	\emph{robust Gittins index} of the $m$th bandit by
	\begin{align}
	\label{eq: Weber def for Gittins}
	\gamma^{(m)}(s) := \essinf\bigg\{\gamma:\essinf_{\tau \in 
		\mathcal{T}^{(m)}(s)} \mathcal{E}^{(m)}\bigg( \sum_{t=1}^{\tau} 
	\beta^{t} 
	\big(h^{(m)}(s+t) - \gamma\big) \; \bigg| \; \mathcal{F}^{(m)}_s 
	\bigg) \leq 
	0 
	\bigg\}
	\end{align}
	where $\mathcal{T}^{(m)}(s)$ is the space of positive 
	$(\mathcal{F}^{(m)}_{s+t})_{t \geq 0}$-stopping times\footnote{Equivalently, for $\tau \in \mathcal{T}(s)$, $s+\tau$ is an $(\mathcal{F}_{t})_{t \geq 0}$-stopping time.} and the outer essential infimum is taken in $L^\infty(\mathcal{F}_s^{(m)})$. 
\end{definition}
By using the results proved in the later sections, we can write the robust Gittins index explicitly. We present this result here for clarity, but make no use of it in subsequent arguments.
\begin{theorem}
	\label{Thm:Explicit_index}
	Let $\gamma(s)$ be the robust Gittins index (Definition \ref{Gittins index}) (with superscript $(m)$ omitted). Then
$$\gamma(s) = \essinf_{\tau \in \mathcal{T}(s)} \; \esssup_{\mathbb{Q} \in 
\mathcal{Q}} \frac{\mathbb{E}^{\mathbb{Q}}\big( \sum_{t=1}^{\tau} \beta^{t} 
h(s+t) \big| \mathcal{F}_s\big)}{\mathbb{E}^{\mathbb{Q}}\big( \sum_{t=1}^{\tau} 
\beta^{t} \big| \mathcal{F}_s\big)} $$
where $\mathcal{Q}$ is the family of probability measures defined in Theorem \ref{Thm:dynamic robust rep}.
\begin{proof}
	See Theorem \ref{Thm:Explicit_index: proof} in the appendix.
\end{proof}
\end{theorem}
Recall that $\mathfrak{E}$ is the partially consistent orthant nonlinear expectation induced by the family $\left(\mathcal{E}^{(m)}\right)_{m \in \mathcal{M}}$ as given in Definition \ref{orthant expectation}. We can obtain an optimal allocation strategy by considering the following theorem.
\begin{theorem}[Robust Gittins theorem]
	\label{robust Gittins}
	Suppose that for each $m \in \mathcal{M}$, $(\mathcal{F}^{(m)}_t)_{t \geq 0}$ is generated by some underlying process $(\xi^{(m)}_t)_{t \geq 1}$. Let $\psi^{(m)}_n$ be the total number of trials of the $m$th bandit before the $n$th play of the system. i.e.  $\psi^{(m)}_n := 
	\sum_{k=0}^{n-1}\mathbb{I}(\rho^*_k = m)$ (given an allocation strategy $\rho^*$ up to time $n-1$). 
	
	Then the allocation strategy 
	$\rho^*$ given (recursively) by $$\rho^*_n := \min \Big\{m \in \mathcal{M} \;\; : \;\; m \in \argmin_{k} 
	\gamma^{(k)}(\psi^{(k)}_n) \Big\}$$
	is C-optimal (Definition \ref{C-optimal}) under $\mathfrak{E}$ for the cost $$g^\rho(n) = \beta^n h^{(\rho_{n-1})}(t^{\rho}_n) \quad \text{where} \quad t^{\rho}_n = \sum_{k=0}^{n-1}\mathbb{I}(\rho_k = \rho_{n-1}).$$	
\end{theorem}
\begin{remark}
	\label{rem: symmetry breaking}
	We choose $\rho^*$ to be the minimum value in the (random) set of minimum Gittins index machines $\{\argmin_{k} 
	\gamma^{(k)}(\psi^{(k)}_n)\}$ as a simple method of symmetry breaking, in order to avoid complexities due to measurable selection. In fact, any choice of $\rho^*_n \in \{\argmin_{k} 
	\gamma^{(k)}(\psi^{(k)}_n)\}$ also yields C-optimality.
\end{remark}
\begin{remark}
	\label{time-consistency remark}
	The robust Gittins theorem states that an optimal choice is given by always playing a bandit with the lowest robust Gittins index. At each time, the indices of unplayed bandits do not change. This leads to a form of consistency in the values associated with different bandits, even though $\mathfrak{E}$ is not consistent.
\end{remark}
\subsection{Sketch of the Proof}
\label{sketch proof}
We will separate the proof into two parts: In Part A, we analyze a one-armed bandit in a robust setting. In Part B, we combine $M$ bandits together. The main body of the rigorous proof can be found in Appendices A and B (respectively) as self-explained sections. We summarize the structure and approach of the proof here.
\subsubsection{One-armed bandit optimality}
We begin by considering play of the $m$th machine (with the superscript $(m)$ omitted).
\paragraph{Step A.1} Observe that the robust Gittins index is the \emph{minimum} compensation for which we are willing to continue to play the bandit (with compensation). 

By minimality, the net expected cost under optimal play must be zero (Theorem \ref{no benefit lemma}), i.e.
$$\esssup_{\tau \in \mathcal{T}(s)} \mathcal{E}\bigg( \sum_{t=1}^{\tau} \beta^{t} 
\left(h(s+t) - \gamma(s)\right) \; \bigg| \; \mathcal{F}_s\bigg) = 0.$$

In particular,   for any subsequent stopping time $\tau \in \mathcal{T}(s)$, we have
\begin{equation}
\label{nonnegative stop}
\mathcal{E}\bigg( \sum_{t=1}^{\tau} \beta^{t} 
\left(h(s+t) - \gamma(s)\right) \; \bigg| \; \mathcal{F}_s\bigg) \geq 0.
\end{equation}
\paragraph{Step A.2} We view the process $\gamma$ as the `average' cost of playing the bandit. Once the process $(\gamma(t))_{t \geq s}$ exceeds $\gamma(s)$, the reward $\gamma(s)$ will no longer be sufficient to encourage continued play; so it will be optimal to stop.  In particular, the stopping time
\begin{equation}
\label{A2 key}
\sigma(s, \gamma(s)) := \inf\{\theta \geq 1 : 
\gamma(s+\theta) > \gamma(s) \}
\end{equation}
 yields equality in \eqref{nonnegative stop} (Theorem \ref{optimal stopping}). 
\paragraph{Step A.3}
Imagine that, whenever the bandit (with compensating reward) is no longer attractive to play, we were to increase the compensation sufficiently to make ourselves indifferent to continuing. The expected value of future loss, with this increased compensation, must again be zero (Proposition \ref{fair game}). The offered compensation can be written as a running maximum of the robust Gittins index process and we can express the expected return
\begin{align}
\label{no total cost running max}
\mathcal{E}\bigg(\sum_{t=1}^{\infty} \beta^{t} 
\big(h(t) - \Gamma(t)\big) \bigg) = 0 \qquad \text{where } \;\; \Gamma(t):= \max_{0 \leq 
	\theta \leq t-1} \gamma(\theta).
\end{align}
With the compensation reward $\left(\Gamma(t)\right)$, we are always willing to continue to play. In particular, at any point in time, we have a non-positive expected future cost (Theorem \ref{expected recovering reward}), i.e.
\begin{align}
\label{expected future reward}
\mathcal{E}\bigg( \sum_{t=N + 1}^{\infty} \beta^{t} 
\left(h(t) - \Gamma(t)\right) \bigg| \mathcal{F}_N \bigg) \leq 0 \;\;\; \text{for all} \;\;\; N = 0,1,...
\end{align}
\paragraph{Step A.4} 
Now suppose we were to take a break from playing for some period, and then resume our earlier strategy. In this case, we may lose some expected profit (Equation \eqref{expected future reward}) due to the discount effect of the delay. .

By \eqref{no total cost running max}, the total reward of this game is zero. Therefore, the delay of getting the reward must result in a possibly worse outcome. In Theorem \ref{delay and prevailing}, we use this observation, together with the robust representation (Assumption \ref{assump:existence of nonlinear}) to show that for any fixed 
$\epsilon > 0$ there is a probability measure $\mathbb{Q} \in \mathcal{Q}$ such that, for 
every decreasing predictable process $(\alpha(t))$ taking values in $[0,1]$,
\begin{equation}
\label{A.4 key}
\mathbb{E}^{\mathbb{Q}}\bigg( \sum_{t=1}^{\infty} \alpha(t) \beta^{t} 
\big(h(t) - \Gamma(t)\big)\bigg) \geq 
-\epsilon.
\end{equation}
\begin{remark}
	Step A.4 is the key point in which positive homogeneity of $\mathcal{E}$ is used.
	A predictable process $(\alpha(t))$ 
	represents the delay due to taking a break to play another bandit. 
	In step A.3, we choose the compensator such that  the total expected return 
	is zero but the bandit is always attractive to be played. (i.e. we always 
	have a reward for the future.) We therefore cannot expect a better 
	outcome 
	than zero if we delay our play. Mathematically, one can replace positive homogeneity and subadditivity by convexity and the property that: if $\mathcal{E}(X|\mathcal{F}_t) \leq 0$, then for all $\mathcal{F}_t$-measurable random variables $\alpha$ taking values in $[0,1]$, we have $\mathcal{E}(\alpha X|\mathcal{F}_t) \geq \mathcal{E}(X|\mathcal{F}_t)$.
\end{remark}

\subsubsection{Information structures for Multi-armed bandits}
We now consider combining play over multiple machines. 

To retain consistency for a single bandit, the nonlinear expectation needs to be defined together with the filtration.  It follows that we need to define an `independent' nonlinear expectation on the joint space of the bandits, which we do via an orthogonal product space. This restriction does not allow us to directly implement Mandelbaum's \cite{Mandelbaum_discrete_allocation} original approach for a dynamic allocation strategy (Definition \ref{Mandelbaum definition}). This is because the multi-parameter process $(\tilde{\eta}(n))$ is only defined to be measurable with respect to the orthant filtration. In particular, it is not clear how one could directly extract the component of $(\tilde{\eta}(n))$ to the marginal space $\Omega^{(m)}$ where our single-bandit nonlinear expectation is defined. 

The importance of decomposing a strategy on the multi-armed bandit to strategies for one-armed bandits can be seen in the proof of El Karoui and Karatzas \cite[Equation 5.1]{Karatzas_El_Karoui_discrete} (via  Whittle's approach \cite{Whittle_proof}), and is described more explicitly in their continuous time paper \cite[Equation 6.9]{cts_Gittin}.

In order to overcome this difficulty, we introduce a class of allocation strategies where there is a component associated to the stopping times of the marginal filtrations. This component allows us to connect and separate the space of multiple bandits to the marginal space of each single bandit. 

Our class of allocation strategies consists of two components $(\tau, p)$. The collection of random times
$\tau = (\tau^{(m)}_k)_{k \geq 0, m\in \{1,...,M\}}$ will identify the 
duration for which will play the 
$m$th bandit, the $k$th time we start to play. This sequence is chosen based 
on historical observations of the $m$th bandit only, that is, the random times 
$\sum_{k=0}^{K}\tau^{(m)}_k$ are $(\mathcal{F}^{(m)}_t)_{t \geq 
	0}$-stopping times for all $K \geq 0$. Once we play a bandit for $\tau^{(m)}_k$ 
trials, we will then reconsider which bandit to play. Our choice of new bandit (which may be the same as before) will be described by the sequence $(p_n)$ taking values in $\{1,..., M\}$, and may depend on information from all bandits. The allocation strategy can be defined formally as follows:

\begin{definition}
	\label{time allocation sequences}
	We say $\tau := \big(\tau^{(m)}_k\big)_{k \geq 0, m \in \mathcal{M}}$ 
	is a \emph{family of time allocation sequences} if 
	\begin{enumerate}[(i)]
		\item For each $m$, $(\tau^{(m)}_k)_{k \geq 0}$  is a sequence of 
		non-negative random times defined on the space $(\Omega^{(m)}, 
		\mathcal{F}^{(m)}_{\infty})$. 
		\item $\sum_{i=0}^{k}\tau^{(m)}_{i}$ is an 
		$(\mathcal{F}^{(m)}_{t})_{t \geq 0}$-stopping time for all 
		$k \geq 0$.
	\end{enumerate}
\end{definition}
Intuitively, the random sequence $(p_n)$ is allowed to depend on all prior 
observations from all bandits. For the sake of precise bookkeeping we need to record, at each moment, how many times we have already played each bandit. This leads to the following definition.
\begin{definition}
	\label{recording sequences}
	Given a family of time allocation sequences $\tau$, we say a sequence of 
	random variables $(\eta_n)_{n\geq0}$ taking values in $\mathcal{S}= \mathbb{N}_0^M$ is a 
	\emph{recording sequence} associated to $\tau$, with corresponding \emph{choice sequence} $(p_n)_{n\in \mathbb{N}_0}$ taking values in $\mathcal{M}$, if 
	\begin{enumerate}[(i)]
		\item $\eta_0 = (0,...,0).$
		\item $\eta_{n+1} = \eta_n + e^{(p_n)}$.
	\end{enumerate}
	
	The choice process $p_n$ satisfies 
	\begin{enumerate}[(i)]
		\setcounter{enumi}{2}
		\item  for all $k \in \mathcal{M}$ and $r\in \mathcal{S}$, 
		\[\{p_{n} = k\} \cap \{\eta_n = r\} \in \mathcal{F}(\Psi_r)= \bigotimes_{m=1}^M 
		\mathcal{F}^{(m)}_{\Psi_r^{(m)} }\] where $\Psi_r^{(m)} := \sum_{i=0}^{r^{(m)}-1} \tau^{(m)}_i$. In particular, $\{\eta_n = r\} \in \mathcal{F}(\Psi_r)$.
	\end{enumerate}		
\end{definition}
For a given time allocation sequence $\tau$, the recording sequence $\eta_n$ determines the \emph{decision filtration}, given by
\begin{equation}
\label{decision filtration}
\mathcal{G}^{(\tau,p)}_n := \Big\{A \in \mathcal{F}(T) \; : \; A \cap \{\eta_{n} = 
r\} \in \mathcal{F}(\Psi_r) \;\;\; \forall r \in \mathcal{S} \Big\} 
\end{equation}
where $\Psi_r^{(m)} = \sum_{i=0}^{r^{(m)}-1} \tau^{(m)}_i$.

\begin{remark}		
	We can see in Definition \ref{recording sequences}(iii) that $(p_n)$ is 
	adapted to the filtration $(\mathcal{G}^{(\tau,p)}_n)_{n \geq 0}$, i.e. we have made our decision what to do next based on our previous observations.
\end{remark}

\begin{definition}
	\label{admissible control}
	An \emph{(admissible) allocation strategy} $(\tau, p)$  
	consists of a family of time allocation sequences $\tau$ and a $(\mathcal{G}^{(\tau,p)}_n)_{n \geq 0}$-adapted choice sequence $p$ (defined under $\tau$).
\end{definition}
\begin{exmp}
	\label{example for allocation strategy}
	Suppose there are two bandits. The first bandit gives 
	only 2 outcomes: $\{w,l\}$. Consider the strategy of playing the first 
	bandit until we see the first $l$. Then we swap to the second bandit for two trials and swap back to the first bandit and repeat the same procedure.
	
	In this case, we define $(X_t)_{t \geq 1}$ to be the outcome of 
	the first bandit and define $\theta_{k+1} := \inf\{t \geq 1 \; : \; X_{t+\sum_{i=0}^k\theta_i} = l 
	\}$. We then have the representation of this strategy
	\begin{align*}
	\tau^{(1)}  &= (\theta_0, \theta_1, ....),  \quad
	\tau^{(2)}  = (2,2,2,...),\quad \text{and} \quad p = (1,2,1,2,...).
	\end{align*}
	The corresponding recording sequence is 
	$$\eta =  
	\big((0,0), (1,0), (1,1), (2,1), (2,2), (3,2), ...\big) .$$
	
 The same strategy can be represented in multiple ways. Here, for example, we can also write
	\begin{align*}
	\tau^{(1)}  &= (\theta_0, \theta_1, ....),  \quad
	\tau^{(2)}  = (1,1,1,...),\quad \text{and} \quad p = (1, 2, 2, 1, 2, 2, 1, ...).
	\end{align*}
	The corresponding recording sequence becomes
	$$\eta =  
	\big((0,0), (1,0), (1,1), (1,2), (2,2), (2,3), (2,4), (3,4), (3,5), (3,6), (4,7),...\big) .$$
	
	As discussed in Remark \ref{Mandelbaum and simple seq}, we can express a Mandelbaum allocation strategy (Definition \ref{Mandelbaum definition}) in terms of a sequence $\rho$ of decisions made at each time. For the strategy described above, this gives the unique sequence 
	$$\rho = \big(\underbrace{1, 1, .., 1}_{\theta_0}, 2, 2,
	\; 
	\underbrace{1, 1, .., 1}_{\theta_1}, 2, 2, \underbrace{1, 1, .., 1,}_{\theta_2} ...\big).$$
\end{exmp}
Extending this example, we can generally write our strategy $(\tau, p)$ in terms of $\rho$ and vice versa. This unique representation provides a simple (if inefficient) description of our strategy, which we now make precise.
\begin{definition}
	\label{simple form}
	Define the random variable $\rho_n$ to be the bandit which will be observed in the $n$th play under an admissible allocation strategy $(\tau, p)$. We call the process  $(\rho_n)_{n \geq 0}$, a \textit{simple form} allocation sequence. The construction of the sequence $(\rho_n)$ is given explicitly in Lemma \ref{construction simple form} in the appendix.
	
	For admissible allocation strategies $(\tau, p)$ and $(\hat{\tau}, \hat{p})$, we write $(\tau, p) \sim (\hat{\tau}, \hat{p})$ if they lead to the same simple form. (Clearly, $\sim$ defines equivalence classes.)
\end{definition}

\begin{remark}
	\label{simple form to Mandelbaum}
	Observe that if $\rho$ is the simple form of $(\tau,p)$ and we denote the time allocation sequence $\mathfrak{1} = \big(1, 1, 1, ...\big)$,	then $(\mathfrak{1}, \rho)$ is an allocation strategy which yields the same decisions as $(\tau, p)$. In particular, we have $(\mathfrak{1}, \rho) \sim (\tau, p)$. 
	
	Furthermore, one can check that the recording sequence corresponding to $(\mathfrak{1}, \rho)$ is exactly the Mandelbaum allocation strategy (Definition \ref{Mandelbaum definition}). In particular, we can explicitly construct a one-to-one correspondence between our equivalence class of admissible strategies (Definition \ref{admissible control}) and Mandelbaum allocation strategies, and we have $\mathcal{G}^{(\mathfrak{1},\rho)}_n = \mathcal{G}^{\rho}_n$ in \eqref{introduce decision filtration}.
\end{remark}
\begin{remark}
	\label{observed vs decision filtration}
	Assume that, for $m \in \mathcal{M}$, the filtration 
	$(\mathcal{F}^{(m)}_t)$  is generated by an underlying real process 
	$(\xi^{(m)}_t)_{t \geq 1}$ defined on the space $(\Omega^{(m)}, \mathcal{F}_\infty^{(m)})$. i.e. $$\mathcal{F}^{(m)}_0 = \big\{\phi, \Omega^{(m)} \big\} \quad \text{and} \quad \mathcal{F}^{(m)}_t = \sigma\big(\xi^{(m)}_1, \xi^{(m)}_2, ..., \xi^{(m)}_t\big). $$ 
	
	If we parameterize our actions by a simple form strategy $(\mathfrak{1},\rho)$ with associated recording sequence $\eta$, then $\rho_{n-1}$ is the decision made at time $n-1$ to generate the outcome observed at time $n$. The observation at the $n$th play is given by
	$$\xi^\rho_n := \xi^{(\rho_{n-1})}_{\eta^{(\rho_{n-1})}_{n}} = \sum_{m=1}^M \sum_{t=1}^{\infty}\xi^{(m)}_t \mathbb{I}(\rho_{n-1} = m,  \; \eta^{(m)}_n = t).$$
	We define the \textit{observed filtration} by $\mathcal{H}^{\rho}_0 = \{\phi, \bar{\Omega} \}$ and $
	\mathcal{H}^{\rho}_n := \sigma\left(\xi^\rho_1, ..., \xi^\rho_n\right) $. 
	We prove, in the appendix, that the observed filtration agrees with that used in Definition \ref{recording sequences} when considering measurability of $\rho$. That is
	\begin{equation}
	\label{decision filtration generated by rho}
	\mathcal{H}^{\rho}_n = \mathcal{G}^{(\mathfrak{1},\rho)}_n =\left\{A \in \mathcal{F}(T) \;\; : \;\; A \cap \{\eta_{n} = r\} \in \mathcal{F}(r) \;\right\}
	\end{equation}
	where $\eta$ is the recording sequence corresponding to $(\mathfrak{1}, \rho)$.
\end{remark}

\subsubsection{Multi-armed bandit optimality}
We can now give the second half of the proof for the robust Gittins index 
theorem where we will consider $0$ as our referencing value function.

In order to prove the optimality of the robust Gittins' strategy, we define the target function for an allocation strategy by 
\begin{equation}
\label{target for C-optimality}
V(\tau,p) := 
\mathfrak{E}\bigg(\sum_{n=1}^{\infty} 
\beta^{n} \left(h^{(\rho_{n-1})}(t^{\rho}_n) - 
\Gamma^{(\rho_{n-1})}(t^{\rho}_n)\right) 
\bigg) \; : \;\; t^{\rho}_n := \sum_{k=0}^{n-1} \mathbb{I}(\rho_k = \rho_{n-1})
\end{equation}
where $\rho$ is a simple form derived from $(\tau,p)$ and $\big(\Gamma^{(m)}(t)\big)$ is the running max of the robust Gittins index of the $m$th bandit, as considered in \eqref{no total cost running max}.

\paragraph{Step B.1} 
 Suppose that we have $M$ bandits, with associated indifference rewards $(\Gamma^{(m)})_{m \in \mathcal{M}}$ as in 
step A.3. If we mix the play of these bandits, this is equivalent to 
taking a break in a single bandit to play the others. This delay will result in a possibly worse outcome (Equation \eqref{A.4 key} in step A.4). 

In Theorem \ref{positivity}, we use the definition of $\mathfrak{E}$ and apply Fubini's theorem to show that, for all allocation strategies $(\tau,p)$, this implies that for any $\epsilon > 0$, there exists a probability measure $\bigotimes_{m=1}^M \mathbb{Q}^{(m)} \in \mathcal{Q}$ such that
\begin{align*}
	V(\tau,p) & \geq
	\mathbb{E}^{\bigotimes_{m=1}^M \mathbb{Q}^{(m)}}\bigg( 
	\sum_{n=1}^{\infty} 
	\beta^{n} \left(h^{({\rho}_{n-1})}(t^{\rho}_n) - 
	\Gamma^{({\rho}_{n-1})}(t^{\rho}_n)\right)\bigg)\\
	& =  \sum_{m=1}^{M} \mathbb{E}^{\mathbb{Q}^{(m)}}\bigg(
	\sum_{t=1}^{\infty} 
	\tilde{\alpha}^{(m)}(t) \beta^t\left(h^{(m)}(t) - 
	\Gamma^{(m)}(t)\right)\bigg) \;\;\; \geq \;\;\; -M\epsilon
\end{align*}
where $\tilde{\alpha}^{(m)}(t)$ is the delay effect on the $m$th bandit due to playing other bandits.

As $\epsilon$ is arbitrary, it follows that for all 
allocation strategies $(\tau,p)$,
\begin{equation}
\label{suboptimal value}
V(\tau,p) \geq 0.
\end{equation}

\paragraph{Step B.2} In step A.2, we noticed that the total expected loss of a single 
bandit between $S$ and $S'$ is zero, for $S$ and $S'$ the consecutive stopping times when 
the robust Gittins index hits a new maximum (Equation \eqref{A2 key}). We use this fact to construct a family of time allocation sequences as a candidate optimal strategy.

Define (inductively) $S^{(m)}_{k} := \sum_{l=0}^{k-1} 
\sigma^{(m)}_{l}$ and
\begin{equation}
\label{sigma allocation}
\sigma^{(m)}_{k} := 
\inf\left\{\theta \geq 1 : \gamma^{(m)}(S^{(m)}_{k} + \theta) >  \gamma^{(m)}(S^{(m)}_{k}) \right\}  
\end{equation}
Using our construction on the class of allocation strategies, we can project the joint nonlinear valuation to its marginal space which is equipped with a consistent nonlinear expectation. We can then use the result from step A.2, that  $\sigma^{(m)}_{k}$ yields equality in \eqref{nonnegative stop}, to show that, for any $p$, with the choice of time allocation sequences $\sigma = (\sigma^{(m)}_k)$, the allocation strategy $(\sigma,p)$ has value
$$V(\sigma,p) \leq 0.$$
This result is shown in Theorem \ref{negativity}.
\paragraph{Step B.3} By combining Step B.1 and Step B.2, for any $p$, with the choice of time allocation sequences $\sigma = (\sigma^{(m)}_k)$ considered above, we have 
\begin{equation}
\label{no cost}
V(\sigma,p) = 0.
\end{equation}

We consider $C^\rho(n) = \beta^n\Gamma^{(\rho_{n-1})}(t^{\rho}_n)$ as a (sub-)compensator in Definition \ref{C-optimal}. The strategy $\rho^*$ given in Theorem \ref{robust Gittins}  is the strategy of always playing the bandit with the minimal index. Therefore, it lies in the same equivalence class as a strategy with the time allocation sequences $(\sigma^{(m)})$ (and with $p$ indicating the minimum index amongst all bandits at each time). Hence, by \eqref{no cost},
$$\mathfrak{E}\left(\sum_{n=1}^{\infty} 
\beta^{n} \left(h^{(\rho^*_{n-1})}(t^{\rho^*}_n) - 
\Gamma^{(\rho^*_{n-1})}(t^{\rho^*}_n)\right) 
\right) = 0.$$

Furthermore, by \eqref{suboptimal value}, 
$$\mathfrak{E}\left(\sum_{n=1}^{\infty} 
\beta^{n} \left(h^{(\rho_{n-1})}(t^{\rho}_n) - 
\Gamma^{(\rho_{n-1})}(t^{\rho}_n)\right) 
\right) \geq 0.$$

  By monotonicity of the process $\Gamma$, we prefer lower value earlier, due to the discount effect.  Thus,  we prove the optimality condition when $N = 0$.  We can now restart our analysis at the considered (orthant) time to obtain the optimal condition for $N > 0$.  We now thus show that $\rho^*$ satisfies the condition for C-optimal. The formal proof of this result can be found in  Theorem \ref{proof: robust Gittins}.

\section{Numerical Results}\label{sec:numerics}

In this section, we study the behaviour of the robust Gittins index using a numerical example. Again we omit the superscript $(m)$ for notational simplicity. 

We suppose the bandit under consideration generates independent identically distributed costs $(h(t))_{1 \leq t \leq T}$ of either \$1 or \$0, given (unknown) 
probability $\mathbb{P}(h(t) = 1) = \theta$ and $h(t) = 2$ for all $t > T$. The filtration $\left(\mathcal{F}_t\right)_{t\ge0}$ is 
generated by the observed cost process $(\xi_t)_{1 \leq t \leq T} = (h(t))_{1 \leq t \leq T}$ (with $\mathcal{F}_0$ trivial). The horizon $T$ can be thought of as the maximum number of times that each bandit can be played.

\begin{remark}
An imaginary horizon $T$ is introduced in order to allow us to easily construct a data-driven recursive nonlinear expectation \eqref{test model} by backward induction

The future cost $h(t) = 2$ is introduced to simplify our numerical method. By considering \eqref{eq: Weber def for Gittins}, we can see that the robust Gittins index $(\gamma(t))_{t \geq 1}$ takes values between $0$ and $1$ when $t \leq T$ and $\gamma(t) = 2$ for $t > T$. Moreover, the optimal stopping time $\sigma(t, \gamma(t)) \leq T-t$. Hence, one can calculate the robust index $\gamma(t)$ by considering a finite horizon optimal stopping problem.
\end{remark}
We model uncertainty in this setting by constructing a one-step coherent nonlinear expectation $\mathcal{E}_{(t)}(\cdot) : L^\infty(\mathcal{F}_{t+1}) \to L^\infty(\mathcal{F}_t)$. Once we have a one-step coherent nonlinear expectation, we can construct an $(\mathcal{F}_t)$-consistent coherent nonlinear expectation by
 $$\mathcal{E}\Big(\; \cdot \; \Big| \mathcal{F}_t \Big) = \mathcal{E}_{(t)}\Big( \mathcal{E}_{(t+1)}\big( \cdots \mathcal{E}_{(T-1)}(\; \cdot \; ) \cdots \big) \Big).$$
 \begin{remark}
We will consider one-step coherent nonlinear expectation which is inspired by the DR-Expectation \cite{DR_original}, see also Bielecki, Chen and Cialenco \cite{Bielecki2017}):
\begin{equation}
\label{test model}
\mathcal{E}_{(t)}\Big(f(\xi_1,...,\xi_t, \xi_{t+1}) \Big) := \sup_{\theta \in \Theta_t} \Big(\theta f(\xi_1,...,\xi_t, 1) + (1-\theta) f(\xi_1,...,\xi_t, 0) \Big)
\end{equation}
 where $\Theta_t = \Big[p^-(p_t, n_t), p^+(p_t, n_t)\Big]$ corresponds to a 
 credible interval for $\theta$ given our observations at time $t$, using a (possibly improper) Beta prior distribution.  The processes $n_t$ and $p_t$ 
 correspond to the number of observations and the (posterior mean) estimate of $\theta$ at time $t$.
 
 In particular, we may choose a credible level $k \in [0,1]$ and obtain 
 $p^\pm(p_t, n_t)$ by
 $$p^\pm(p_t, n_t) = I_{(p_tn_t, (1-p_t)n_t)}^{-1}(0.5 \pm k/2)$$
 where $q \mapsto I_{(a,b)}^{-1}(q)$ is the quantile function of the $\text{Beta}(a,b)$ distribution.

 	One could also use the central limit theorem to obtain an asymptotic confidence interval. 
 	However, due to the fact that $\Theta_t \subseteq [0,1]$, we restrict 
 	ourselves to the credible set above to avoid end-effects, and allow for asymmetry in the plausible values around the `best' estimate.
 \end{remark}

As our credible set is constructed from $p_t$ and $n_t$, and the pair $(p_t, n_t)$ can be computed recursively, it follows that for every 
$f: \{0,1\}^{T-t} \to \mathbb{R}$, there exists a function $g_{k, T-t}: 
\mathbb{R}^2 \to \mathbb{R}$ such that
$\mathcal{E}\big( f(\xi_{t+1}, ..., \xi_T)\big| \mathcal{F}_t\big) = 
g_{k,T-t}\Big(p_t, \frac{1}{\sqrt{n_t}}\Big).$
\footnote{Here, we write the nonlinear expectation as a function of 
$n_t^{-1/2}$ instead of $n_t$ as we wish to approximate our function 
on a compact domain. The choice of $n_t^{-1/2}$ comes from the natural scaling of the credible set.}

By recalling the definition of $\gamma(s)$ (Definition \ref{Gittins index}),  one can show (using a general robust dynamic programming argument, as in Ruszczy\'nski \cite{General_robust_Markov_DPP}, or the nonlinear Snell's envelope, as in Riedel \cite{nonlinear_snell_envelope}) that we can write 
$$\gamma(t) = \gamma_{k, \beta, T-t}\Big(p_t, \frac{1}{\sqrt{n_t}}\Big) \qquad \text{where} \qquad n_t = n_0 + t.$$
for some function $\gamma_{k, \beta, T-t}$. 

We then use a simple finite-difference algorithm (see Appendix \ref{Surface Algorithm}) to estimate the function $$\Big(p, \frac{1}{\sqrt{n}} \Big) \mapsto \gamma_{k,\beta,T-(n-n_0)}\Big(p, \frac{1}{\sqrt{n}}\Big) - p$$
where, in our simulations, we fix $n_0 = 1$.

 Plots of this estimate, for various values of $k$, $\beta$ and $T$, can be found in Figure \ref{fig:surface}.  
\begin{figure}
	\centering	
	\includegraphics[clip, trim=2cm 3cm 1cm 4cm, width=1.00\textwidth]{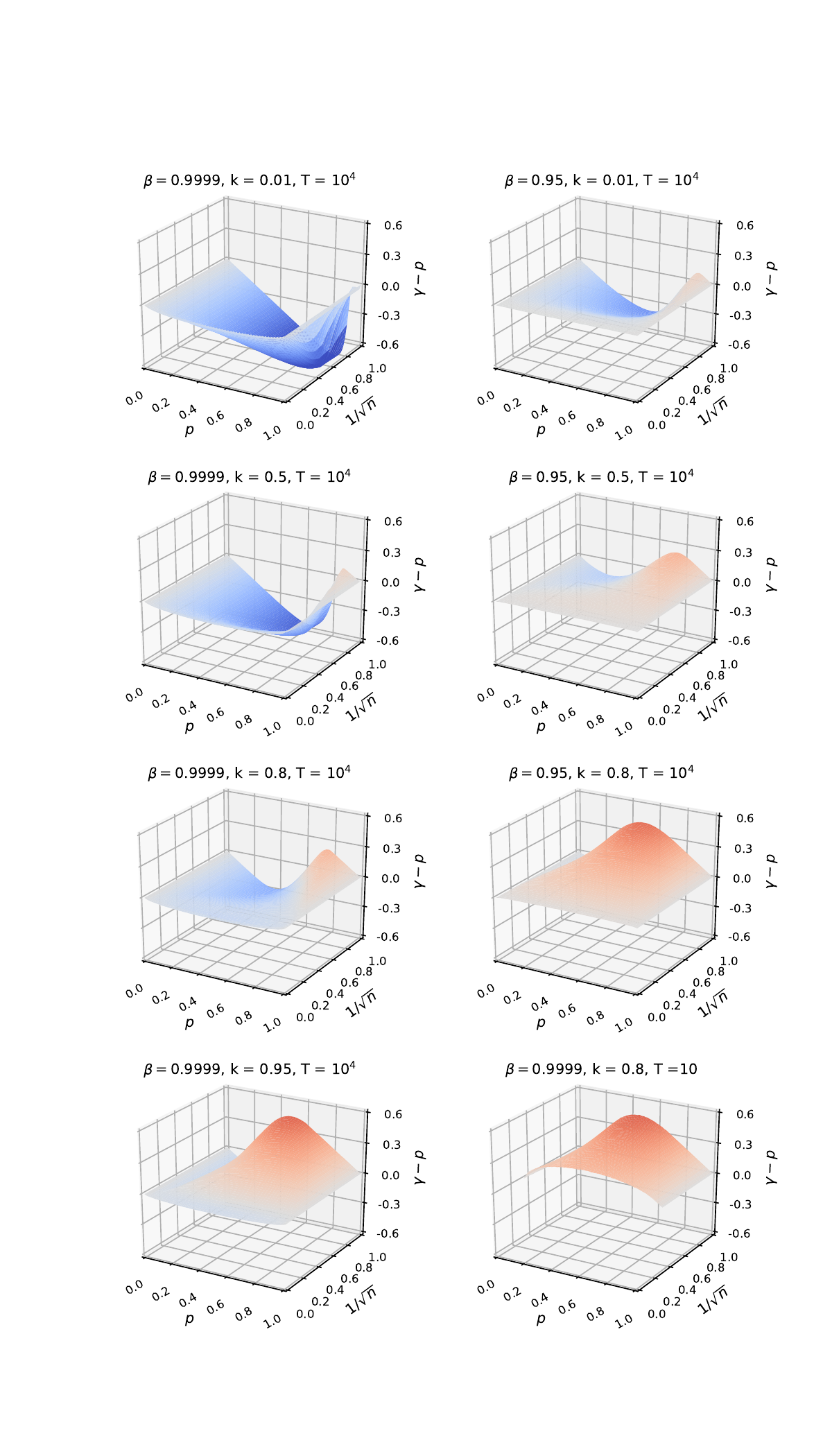}
	\caption{Estimated value of $\gamma - p$ for different values of $k$, $\beta$ and $T$. The case $T=10$ is truncated as $n$ cannot exceed $T$ (by definition)}
	\label{fig:surface}
\end{figure}

For $h(t) = \xi_t$, with uncertainty modeled by \eqref{test model}, at each time step we wish to play the bandit with the lowest $\theta$. Classically, this is estimated by $p$, so a na{\"\i}ve (greedy) strategy would suggest playing the bandit with the lowest estimated average loss $p$. By using C-optimality, at each point, we choose a bandit with the lowest $\gamma$. Therefore, we may think of $\gamma$ as an implied probability $p$, distorted to account for exploration and exploitation of the system of bandits.

In Figure \ref{fig:surface}, we see the following broad phenomena:
\begin{itemize}
 \item When $1/\sqrt{n}$ is small, the difference between $\gamma$ and $p$ is close to zero. In particular, this says that when we have high certainty in our estimates, $\gamma$ is equivalent to the estimated probability.  
 \item When we increase $\beta$, the difference typically $\gamma - p$ decreases. This corresponds to the fact that $\beta$ is a discount factor which determines how much we value future costs. Therefore, increasing $\beta$ increases the degree that we wish to explore the system, i.e. we become more optimistic in our evaluation. We also observe that decreasing $\beta$ also yields a similar result to shortening the horizon.
 \item When $k$ is increased,  the difference $\gamma - p$ increases. This is due to the fact that  $k$  corresponds to the `width' of the `credible interval'. Hence, large $k$ means that we become more conservative and favour exploiting over exploring.
 \end{itemize}
 \subsection{Prospect Theory}
One result suggested in  Figure \ref{fig:surface} when $\beta = 0.9999$ and $k=0.01$ is that, when we do not worry about uncertainty, we are more optimistic when $p$ is large (close to $1$), that is, $\gamma$ is clearly less than $p$. On the other hand, when uncertainty dominates, e.g. when  $\beta = 0.9999$ and $k=0.95$, or $\beta = 0.95$ and $k=0.8$, we become more pessimistic. 

Curiously, when $\beta = 0.9999$ and $k=0.8$, or $\beta = 0.95$ and $k=0.5$, both optimism and pessimism can be seen. For large $p$, (when the game seems bad), pessimism dominates, while for small $p$ (when the game seems good) we become optimistic in our optimal strategy. This gives a bias in the probabilities, related to that used in the probability weighting functions as considered in prospect theory by Kahneman and Tversky \cite{prospect_theory} or in rank-dependent expected utility by Quiggin \cite{rank-dependent_theory_paper, rank-dependent_theory_book}. In this literature, they propose models to explain irrationality in human decisions under risk. They argue that people generally reweigh the probabilities of different outcomes using a nonlinear increasing map $p\mapsto \pi(p)$, with various assumptions on its curvature.

Our result (for appropriate values of $\beta$ and $k$) reflects this behaviour without imposing a probability weighting function as in classical prospect theory. Instead, the combination of the effect of learning and uncertainty leads to distortions of the estimated probability. 
\subsection{Monte-Carlo Simulation}
In order to illustrate the performance of the robust Gittins index calculated above in the real decision making, we consider the Bernoulli bandit as described above over 50 exchangeable bandits and for a horizon $T = 10^4$. We run $10^3$ Monte-Carlo simulations and compare performance of various strategies for decision making. To provide a wide range of scenarios in which our strategies must perform, in each simulation we first generate $a,b$ independently from a $\Gamma(1, 1/100)$ distribution, then generate the `true' probabilities for each bandit independently from $\text{Beta}(a,b)$. We generate $10$ trials on each bandit to provide initial information.

N.B. Formally, we assume that each bandit can be played for at least $T = 10^4$ trials in constructing our Gittins index. We illustrate the performance of the first $10^4$ plays to compare with other algorithms.

\subsubsection{Measures of Regret}
There are a number of possible objectives to measure the loss of our decisions. We will consider the following examples
 (from Bertini et al.\:\cite{adaptive_learning_survey} and Lai and Robbins \cite{Asymp_efficient}).
\begin{itemize}
	\item \textbf{Expected--expected regret.} This is the difference in the true expectations under our strategy and an optimal strategy  with perfect information. In our setting, this can be given by
	$R(L) = \sum_{n=0}^{L}(\theta^{(\rho_{n})} - \theta^{*})$
	where $\theta^{(m)}$ is the true probability of the $m$th bandit and $\theta^{*} = \min_m \theta^{(m)}$.
	\item \textbf{Sub-optimal plays.} This measures the number of times where we play a sub-optimal bandit which is given by
	$N_\vee(L) = \sum_{n=0}^{L} \mathbb{I}(\theta^{(\rho_{n})} \neq \theta^{*}).$
\end{itemize}
\subsubsection{Policy for multi-armed-bandits}
In our simulation, we will label our algorithm the DR (Data-Robust) algorithm. We also consider the following classical policies which are commonly used to solve the Bernoulli bandit problem. These policies choose an arm by considering the minimal index $I$ evaluated on each bandit separately. Literature about these policies and further developments can be found in the reviews by Bertini et al. \cite{adaptive_learning_survey} or Russo et al. \cite{Tutorial_on_Thompson_Sampling}. For notational simplicity, we will denote by $p$ and $n$ the estimated probability and the number of observations of the considered bandit at the time before making a decision. 
\begin{itemize}
	\item \textbf{Greedy strategy.} In this policy, we choose the bandit with the minimal estimated probability given by $I^{Greedy} = p.$
	\item \textbf{Thompson strategy.} This is a Bayesian adaptive decision strategy for the bandit problem. It proceeds by first randomly generating a sample from the posterior distribution of the mean cost of each bandit, then chooses to play the bandit which gave the minimal sample. In our setting, these samples are given by
		$I^{Thompson} \sim \text{Beta}(a_0 + pn, b_0 + (1-p)n)$
	where $a_0,b_0 > 0$ are the parameters of a Beta prior distribution for the mean. To avoid biasing our estimation, we consider initial values $a_0=b_0\in \{0.0001, 1, 50\}$, where larger values correspond to a more informative prior.
	\item \textbf{UCB strategy.} This is an optimistic strategy to choose the bandit based on its lower bound.
	$I^{UCB} = p - \sqrt{\frac{\lambda \log{N}}{n}}$
	where $\lambda > 0$ is a chosen parameter, which is commonly chosen to be $2$ and $N := \sum_{m=1}^{M}n^{(m)}$ is the total number of observations across all bandits.
\end{itemize}
\begin{remark}
	To avoid bias in the algorithms, we choose a bandit uniformly at random if there is more than one bandit with minimal index.
\end{remark}
\begin{figure}[H]
	\centering
	\includegraphics[clip, trim=3.2cm 1cm 3.3cm 1.5cm, width=0.9\linewidth]{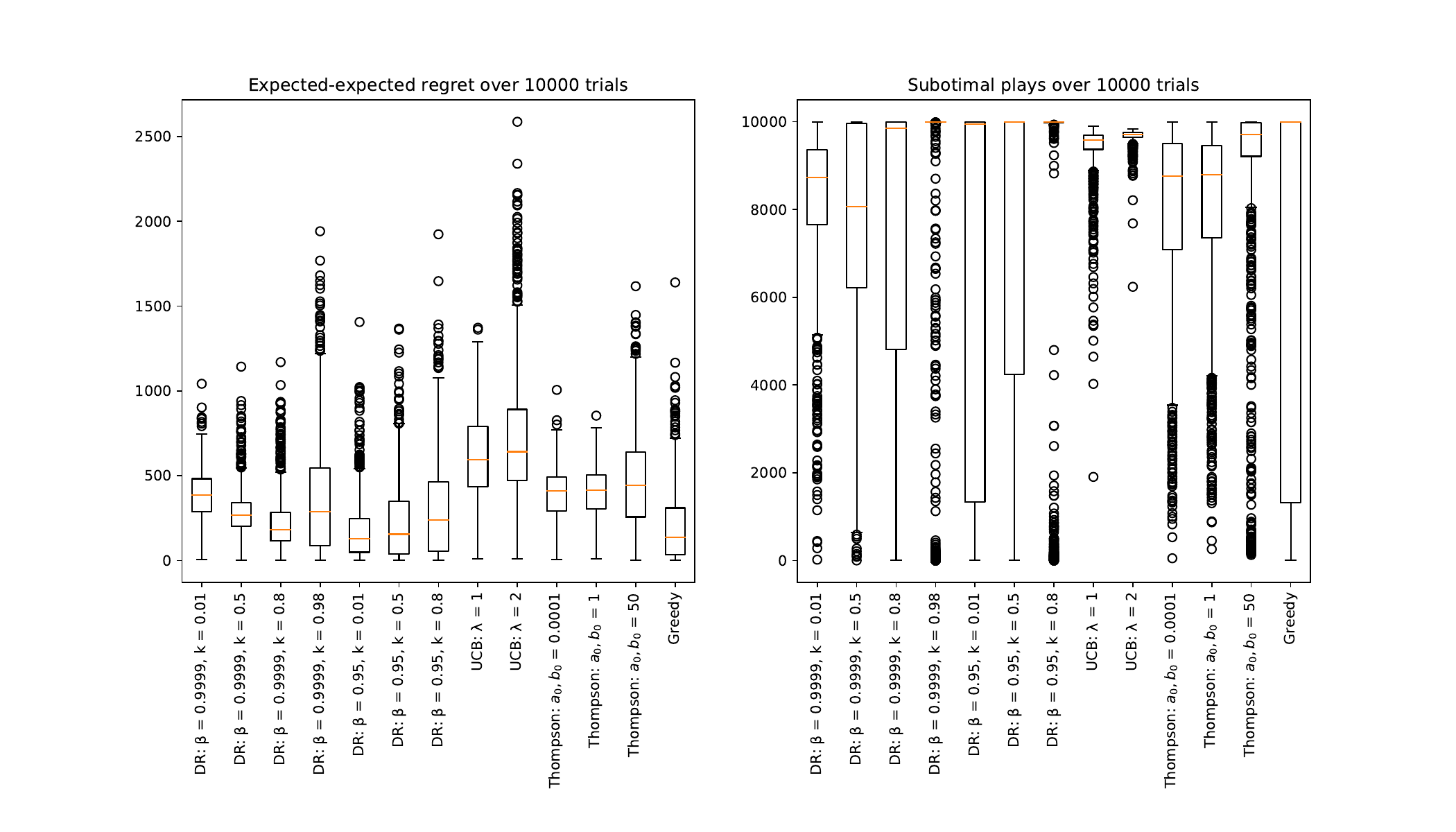}
	\caption{Regret under different policies}
	\label{fig:regretplot}
\end{figure}

In Figure \ref{fig:regretplot}, considering first the cases where $\beta = 0.9999$, we can see that an increase in the value of $k$ has a nonlinear effect on the distribution of regret. Initially, increasing $k$ appears to lead to a reduction in the typical regret, but a possible increase in the average and variability of the number of suboptimal plays. However, setting $k$ too large clearly leads to worse outcomes. This is because $k$ corresponds to the level of robustness; the more robust we are, the less willing we are to explore and the more willing we are to exploit. It follows that a large value of $k$ encourages us to exploit early, and we may not find the optimal bandit to play.

On the other hand, the discount rate $\beta$ determines how much we value our future costs. If we have a high level of robustness (large $k$) but do not value the future cost enough (small $\beta$), we may end up settling for a sub-optimal decision. This can be seen most clearly when $\beta = 0.9999$ and $k = 0.8$. In this case the average expected-expected regret is relatively small when compared to other strategies, but its average number of suboptimal plays is relatively high. Reducing $\beta$ to $0.95$ emphasizes these effects even further.

As discussed in the introduction, the UCB algorithm asymptotically achieves a minimal regret bound (see \cite{Asymp_efficient}). It does so by ensuring that, over short horizons, the algorithm explores a sufficient amount, in order to guarantee good asymptotic performance. We can see that in our simulation (with $10^4$ plays over $50$ bandits), the UCB algorithm is still in its high exploration regime which results in a high regret and very few optimal plays. In contrast, a Greedy algorithm always chooses an arm to play without taking into account its uncertainty (and so without considering the possibility for exploration) and therefore there is no learning in its procedure. This results in the greedy algorithm yielding a low average regret but a high average number of suboptimal plays.
\subsubsection{Robustness of the DR Algorithms}
\begin{figure}[H]
	\centering
	\includegraphics[clip, trim=1cm 0cm 1cm 0cm, width=1\linewidth]{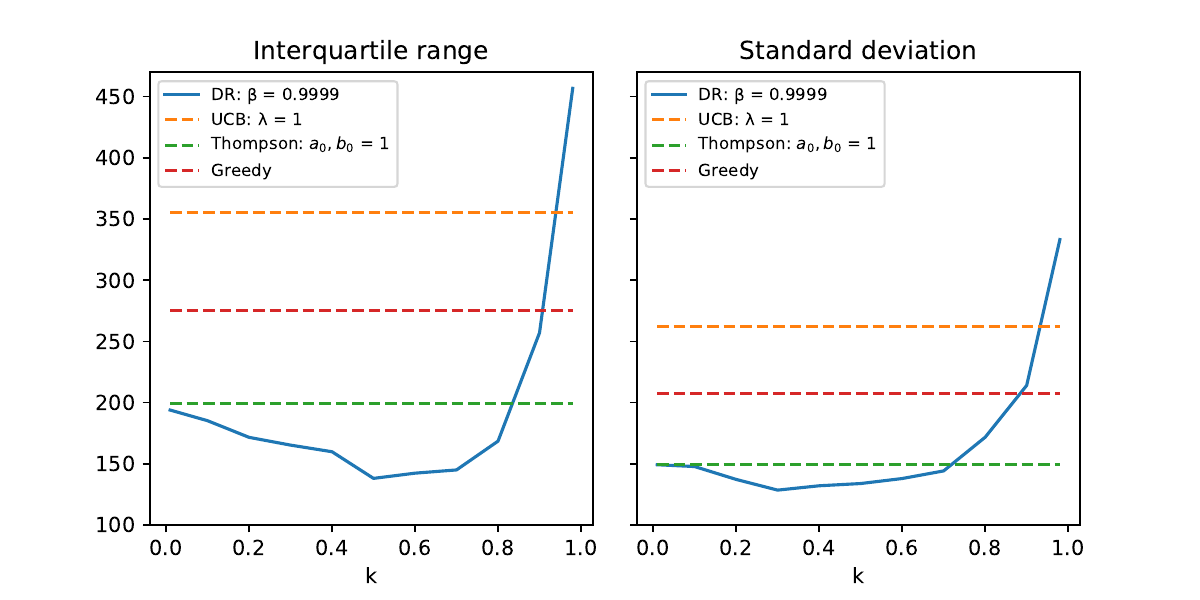}
	\caption{Deviation of the expected-expected regret when $\beta = 0.9999$}
	\label{fig:errorregretplot}
\end{figure}
In Figure \ref{fig:errorregretplot}, we illustrate the interquartile range and the standard deviation of the total expected-expected regret when $\beta = 0.9999$ with different values of $k$ over $1000$ simulations. We can see that by introducing an appropriate values of $k$, we can obtain a substantial reduction in the interquartile range and the standard deviation. In particular, the DR algorithm does not only give a low average regret but also does so consistently over different simulations.

%
%

\textbf{Acknowledgements:} Samuel Cohen thanks the Oxford-Man Institute for research support and acknowledges the support of The Alan Turing Institute under the Engineering and Physical Sciences Research Council grant EP/N510129/1. Tanut Treetanthiploet acknowledges support of the Development and Promotion of Science and Technology Talents Project (DPST) of the Government of Thailand.

\addcontentsline{toc}{chapter}{Bibliography}
\bibliography{bibliography} 
\bibliographystyle{plain}
\appendix
\section{Part A: Analysis of a single bandit}\label{sec:partAproof}
We will now flesh out the sketch given in Section \ref{sketch proof}.

In this section, we will focus the discussion on a single bandit.
\subsection{Step A.1: Indifference reward and Optimal Stopping problem}
We first recall the definition of the robust Gittins index (process).
\begin{align*}{\gamma(s) := \essinf\bigg\{\gamma \in L^\infty(\mathcal{F}_s) : 
\essinf_{\tau \in 
		\mathcal{T}(s)} \mathcal{E}\bigg( \sum_{t=s+1}^{s+\tau} 
	\beta^{t} 
	\big(h(t) - \gamma\big) \; \bigg| \; \mathcal{F}_s 
	\bigg) \leq 
	0 
	\bigg\}
}\end{align*}
where $\mathcal{T}(s)$ denotes the family of $\left(\mathcal{F}_{s+t}\right)_{t 
\geq 0}$-positive stopping times. 
\begin{remark}
	If we take $\tau =  1$, we observe by boundedness of $h$ (Assumption 
	\ref{assumption for cost process}) that $\gamma(t) < C$.
\end{remark}
To study the process $\gamma$, we introduce an auxiliary optimal stopping 
problem. At each time step, the player decides whether to continue or to stop 
play of the machine. If the player decides to continue to play, he will be 
offered a fixed reward $\lambda$ (known at the initial time $s$) in addition to 
the cost $h(t)$. 
\begin{definition}
	The target function $V_s: \mathcal{T}(s) \times L^{\infty} (\mathcal{F}_s) 
	\to L^{\infty} (\mathcal{F}_s)$ for a stopping time $\tau \in 
	\mathcal{T}(s)$ with a reward $\lambda$ is defined by
	$$V_s(\tau,\lambda) = \mathcal{E}\bigg( \sum_{t=s+1}^{s+\tau} \beta^{t} 
	\left(h(t) - \lambda\right) \; \bigg| \; \mathcal{F}_s\bigg).$$	
\end{definition}

We know that $\gamma(s)$ is defined to be the minimum reward $\lambda$ such 
that, with a choice of $\tau$ minimizing $V_s(\tau, \lambda)$, the expected 
loss 
is at most zero. By minimality of $\gamma(s)$ and monotonicity of 
$\mathcal{E}$, the reward $\gamma(s)$ will yield zero loss under optimal 
stopping and, therefore, cannot yield a positive expected reward under 
suboptimal stopping. In particular, the following holds.

\begin{theorem}
	\label{no benefit lemma}
	The function $V_s$ defined above satisfies.
	$$\essinf_{\tau \in \mathcal{T}(s)}V_s(\tau, \gamma(s)) = 0. $$	
	\begin{proof}
		This can be done by showing that $V_s$ satisfies the regularity 
		assumptions 
		of Lemma 
		\ref{double inf relation}. 
		
		By considering $\lambda = C+1$, where $C$ is an upper bound on $h$, we 
		see that $V_s(\tau, C+1) < 0$. As $h(t) \geq 0$, it also follows that 
		$V_s(\tau,0) \geq 0$. Hence, condition (i) is satisfied.
		
		For condition (ii), suppose that $\lambda' > \lambda$. Then
		\begin{align*}
		\bigg( \sum_{t=s+1}^{s+\tau} \beta^{t} \big(h(t) - 
		\lambda\big) \bigg) &= \bigg( \sum_{t=s+1}^{s+\tau} \beta^{t} \big(h(t) 
		- 
		\lambda'\big) \bigg) +  \sum_{t=s+1}^{s+\tau} 
		\beta^{t}(\lambda'-\lambda) \\
		&\leq \bigg( \sum_{t=s+1}^{s+\tau} \beta^{t} \big(h(t) - 
		\lambda'\big) \bigg) +  
		\Big(\frac{\beta^{s+1}}{1-\beta}\Big)(\lambda'-\lambda).
		\end{align*}
		
		By monotonicity and translation equivariance, we have
		\begin{align*}
		0 & \leq  V_s(\tau,\lambda) - V_s(\tau,\lambda') \\ 
		& =  \mathcal{E}\bigg( \sum_{t=s+1}^{s+\tau} \beta^{t} \big(h(t) - 
		\lambda\big) \; \bigg| \; \mathcal{F}_s\bigg) - \mathcal{E}\bigg( 
		\sum_{t=s+1}^{s+\tau} \beta^{t} \big(h(t) - \lambda'\big) \; \bigg| 
		\; \mathcal{F}_s\bigg)\\
		&\leq \Big(\frac{\beta^{s+1}}{1-\beta}\Big)(\lambda'-\lambda).
		\end{align*}
		So, $V_s$ is Lipschitz in $\lambda$.
		
		Condition (iii) follows from $(\mathcal{F}_t)_{t\ge 0}$-regularity 
		of $\mathcal{E}$ (Remark \ref{remark: regularity}). The result follows from Lemma \ref{double inf relation}.
	\end{proof}
\end{theorem}

\begin{corollary}
	\label{cost_vs_fair charge}
	For every $\tau \in \mathcal{T}(s)$, we have
	$$\mathcal{E}\bigg( \sum_{t=s+1}^{s+\tau} \beta^{t} \big(h(t) - 
	\gamma(s)\big) \bigg| \mathcal{F}_s\bigg) \geq 0.$$
\end{corollary}
\begin{remark}
	\label{gamma as average cost}
	Theorem \ref{no benefit lemma} shows that, under optimal stopping, with the 
	reward $\gamma(s)$, the expected total loss is zero. In particular, we may 
	view $\gamma(s)$ as an `average cost under optimal play' of the bandit. 
\end{remark}
\subsection{Step A.2: Optimal Stopping time}
By considering a Snell envelope argument, as in Riedel 
\cite{nonlinear_snell_envelope} with slight modification, we can establish 
that a stopping time $\tau^*$ achieving the minimum value $V_s(\tau^*, \lambda) 
= \essinf_{\tau \in 
	\mathcal{T}(s)} 
V_s(\tau, \lambda)$ exists (Theorem \ref{existence of optimal stopping time}). 
In this subsection, we will show that $\tau^*$ can be expressed as a hitting 
time of the Gittins index process $(\gamma(s))$.

\begin{definition}
	\label{optimal stop upto lambda} Let $\lambda$ be a non-negative 
	$\mathcal{F}_s$-measurable random variable. Define a stopping time 
	$\sigma(s,\lambda)$ by	
	$$\sigma(s,\lambda) := 	\inf\{\theta \geq 1 \; : \; \gamma(s + \theta) > 
	\lambda \}$$
\end{definition}

As mentioned in Remark \ref{gamma as average cost}, we may view $\gamma$ as 
a time-average cost under optimal stopping. The stopping time 
$\sigma(s,\lambda)$ can be interpreted as the first time when this average 
cost exceeds a fixed $\lambda$. Once $\gamma$ exceeds $\lambda$, the offered 
compensation $\lambda$ is insufficient to make the bandit attractive so, to 
minimize the total `expected' cost, we will stop.

In what follows, we formalize this intuition. We will show that 
$\sigma(s,\lambda)$ is an optimal stopping time when the reward $\lambda$ 
is offered. In particular, we will show that $\sigma(s,\gamma(s))$ attains 
the optimal value with the reward $\lambda = \gamma(s)$. Moreover, the 
value for this optimal stopping problem is zero (by Theorem
\ref{no benefit lemma}).

The optimality of $\sigma(s, \lambda)$ can be proved by showing that for any 
stopping time $\tau \in \mathcal{T}(s)$, if $\tau > \sigma(s, \lambda)$ on some 
event, our value can be improved by stopping at $\sigma(s, \lambda)$ (Lemma 
\ref{optimal stopping1}). On the other hand, if $\tau < \sigma(s, \lambda)$ on 
some event, the value can be improved by continuing to play  
(Lemma \ref{optimal stopping2}). The easy proofs of these Lemmata are in the 
appendix \ref{Append: additional result}.
\begin{lemma}
	\label{optimal stopping1}
	For every $\lambda \in L^\infty(\mathcal{F}_s)$  taking values in $[0,C)$
	and $\tau \in \mathcal{T}(s)$, 
	$$V_s(\tau, \lambda) \; \geq \; V_s(\tau \wedge \sigma(s,\lambda) , 
	\lambda).$$
	\begin{proof}
		We will prove this result by applying Corollary \ref{cost_vs_fair 
		charge} 
		together with time-consistency and monotonicity of our nonlinear 
		expectation. 
		
		Define $\nu = \tau \wedge \sigma(s,\lambda)$. By Corollary 
		\ref{cost_vs_fair charge} and regularity  (Remark \ref{remark: regularity}),	
		\begin{align*}
		0 & \leq  \mathcal{E}\bigg( \sum_{t=s+\nu+1}^{s+\tau} \beta^{t} 
		\left(h(t) - \gamma(s+\nu)\right) \bigg|  \mathcal{F}_{s + 
			\nu}\bigg)	\\
		& =  \mathbb{I}(\tau > \sigma(s,\lambda)) \mathcal{E}\bigg( 
		\sum_{t=s+\nu+1}^{s+\tau} \beta^{t} \left(h(t) - \gamma(	
		s+\sigma(s,\lambda))\right)  \bigg| \mathcal{F}_{s+\nu}\bigg) + 
		\mathbb{I}(\tau \leq \sigma(s,\lambda))(0) \\
		& \leq  \mathbb{I}(\tau > \sigma(s,\lambda)) \mathcal{E}\bigg( 
		\sum_{t=s+\nu+1}^{s+\tau} \beta^{t} \left(h(t) - \lambda 
		\right) \bigg| \mathcal{F}_{s+\nu}\bigg) \\
		& =   \mathcal{E}\bigg( \sum_{t=s+\nu+1}^{s+\tau} \beta^{t } 
		\left(h(t) - \lambda \right) \bigg| 
		\mathcal{F}_{s+\nu}\bigg).
		\end{align*}
		By translation equivariance,
		\begin{align*}
		\sum_{t=s+1}^{s+\nu } \beta^{t} \left(h(t) - \lambda\right)
		&\leq  \sum_{t=s+1}^{s+\nu} \beta^{t} \left(h(t) - \lambda\right) + 
		\mathcal{E}\bigg( \sum_{t=s+\nu+1}^{s+\tau} \beta^{t } 
		\left(h(t) - \lambda \right) \bigg| 
		\mathcal{F}_{s+\nu}\bigg)\\
		&  =  \mathcal{E}\bigg( \sum_{t=s+1}^{s+\tau} \beta^{t} \left(h(t) 
		- 
		\lambda \right)  \bigg| \mathcal{F}_{s+\nu}\bigg).
		\end{align*}
		By monotonicity and {time-consistency},
		\begin{align*}
		\mathcal{E}\bigg( \sum_{t=s+1}^{s + \nu} \beta^{t} 
		\left(h(t) - \lambda\right)  \bigg|  \mathcal{F}_s\bigg) 
		& \leq  	\mathcal{E} \bigg( \mathcal{E}\bigg( \sum_{t=s+1}^{s+\tau } 
		\beta^{t} \left(h(t) - \lambda\right) \bigg| 
		\mathcal{F}_{s+\nu}\bigg)  \bigg|  \mathcal{F}_s\bigg) \\& =  
		\mathcal{E}\bigg( \sum_{t=s+1}^{s+\tau} \beta^{t} \left(h(t) - 
		\lambda\right) \bigg| \mathcal{F}_s\bigg).
		\end{align*}	
		In particular, $V_s(\tau \wedge \sigma(s,\lambda) , \lambda) = V_s(\nu 
		, 
		\lambda) \leq V_s(\tau , \lambda)$.
	\end{proof}
\end{lemma}
\begin{lemma}
	\label{optimal stopping2}
	Let $\tau \in \mathcal{T}(s)$ and let $\lambda \in L^\infty(\mathcal{F}_s)$ 
	taking values in $[0,C)$. 
	Then there exists a stopping time $\tau_1 \in 
	\mathcal{T}(s)$ with  $\tau_1 \geq \tau$ such that 
	$$V_s(\tau, \lambda) \geq V_s(\tau_1, \lambda)$$
	and on the event $A := \{\gamma(s + \tau) \leq \lambda\}$, we have $\tau_1 
	> \tau$.
	\begin{proof} 
		For $A = \{\gamma(s+\tau) \leq \lambda\}$, define $\gamma^{\tau} := 
		\gamma(s + \tau)\mathbb{I}_{A^c} + \lambda \mathbb{I}_A \geq \gamma(s + 
		\tau)$. By Theorem \ref{no benefit lemma} and monotonicity of our 
		nonlinear 
		expectation, 
		\begin{align*}
		0 & = \essinf_{\tilde{\tau} \in \mathcal{T}(s + \tau)} \mathcal{E} 
		\bigg(\sum_{t=s+\tau + 1}^{s+\tau + \tilde{\tau}}\beta^{t} \left(h(t) - 
		\gamma(s + \tau) \right) \bigg| \mathcal{F}_{s+\tau}\bigg) \\
		& \geq  \essinf_{\tilde{\tau} \in \mathcal{T}(s + \tau)} \mathcal{E} 
		\bigg(\sum_{t=s + \tau+1}^{s + \tau + \tilde{\tau}}\beta^{t} \left(h(t) 
		- 
		\gamma^{\tau} \right)  \bigg| \mathcal{F}_{s+\tau}\bigg).
		\end{align*}
		
		Thus, by Theorem \ref{existence of optimal stopping time}, there exists 
		$\tilde{\tau}^* \in \mathcal{T}(s + \tau)$ such that,
		\begin{align*}
		0 & \geq  \mathcal{E} 
		\bigg(\sum_{t=s + \tau+1}^{s + \tau + \tilde{\tau}^*}\beta^{t} 
		\left(h(t) 
		- 
		\gamma^{\tau} \right)  \bigg| \mathcal{F}_{s+\tau}\bigg).
		\end{align*}
		
		Define a stopping time $\tau_1 := \tau + \tilde{\tau}^* \; 
		\mathbb{I}_A$. 
		As $\gamma^\tau \mathbb{I}_A = \lambda \mathbb{I}_A$, then
		\begin{align*}
		\sum_{t=s + 1}^{s+\tau_1}\beta^{t} \left(h(t) - 
		\lambda\right) &=  
		\sum_{t=s + 1}^{s+\tau}\beta^{t} \left(h(t) - \lambda\right)	+ 
		\mathbb{I}_A \sum_{t=s+\tau + 1}^{s + \tau + \tilde{\tau}^*}\beta^{t} 
		\left(h(t) - \lambda \right)\\
		&=  
		\sum_{t=s + 1}^{s+\tau}\beta^{t} \left(h(t) - \lambda\right)	+ 
		\mathbb{I}_A \sum_{t=s+\tau + 1}^{s + \tau + \tilde{\tau}^*}\beta^{t} 
		\left(h(t) - \gamma^\tau \right).
		\end{align*}
		By translation equivariance and regularity (Remark \ref{remark: 
		regularity}), it follows that
		\begin{align*}
		&\mathcal{E} \bigg(\sum_{t=s + 1}^{s+\tau_1}\beta^{t} \left(h(t) - 
		\lambda\right)\bigg| \mathcal{F}_{s+\tau}\bigg) 
		\\& \qquad = \sum_{t=s + 
			1}^{s+\tau}\beta^{t} \left(h(t) - \lambda\right) + 
			\mathbb{I}_A\mathcal{E} 
		\bigg(\sum_{t=s + 1}^{s+\tau + \tilde{\tau}^*}\beta^{t} \left(h(t) - 
		\gamma^\tau\right)\bigg| \mathcal{F}_{s+\tau}\bigg) 
		\\ & \qquad \geq \sum_{t=s + 
			1}^{s+\tau}\beta^{t} \left(h(t) - \lambda\right).
		\end{align*}
		Finally, by applying monotonicity and time-consistency as in the 
		previous 
		lemma, the result follows.
	\end{proof}
\end{lemma}

\begin{corollary}
	\label{optimal stopping3}
	Let $\tau \in \mathcal{T}(s)$. Then there exists an increasing sequence 
	$(\tau_n)_{n \geq 1}$ in $\mathcal{T}(s)$ with $\tau_{n+1}\geq \tau_n \geq 
	\tau $ for all $n \geq 1$ such that 
	$$V_s(\tau, \lambda) \geq V_s(\tau_1, \lambda) \geq ...  \geq V_s(\tau_n, 
	\lambda)$$
	and on the event
	$\bigcap_{k=1}^{n-1} \left\{\gamma(s+\tau_k) \leq \lambda \right\}$
	we have $\tau_n > \tau_{n-1} > ... > \tau_1 > \tau$. In particular, on this 
	event, $\tau_n \geq n$. 
\end{corollary}
By combining these observations with the Lebesgue property of $\mathcal{E}$, 
we have the following theorem.
\begin{theorem}	
	\label{optimal stopping}
	For every $\lambda \in L^\infty(\mathcal{F}_s)$ taking values in $[0,C)$
	and $\tau \in \mathcal{T}(s)$, we have
	
	$$V_s(\tau, \lambda) \; \geq \; V_s(\sigma(s,\lambda), \lambda)= 
	\essinf_{\tau \in \mathcal{T}(s)}V_s(\tau, \lambda).$$
	Therefore, 
	$$V_s(\sigma(s,\gamma(s)), \gamma(s)) = \essinf_{\tau \in 
		\mathcal{T}(s)}V_s(\tau, \gamma(s)) = 0.$$
	In particular, $\sigma(s,\gamma(s))$ yields equality in Corollary 
	\ref{cost_vs_fair charge}.
	\begin{proof}
		By Lemma \ref{optimal stopping1} and Corollary \ref{optimal stopping3},
		\begin{equation*}
		V_s(\tau, \lambda) \geq V_s(\tau_n, \lambda) \geq V_s(\tau_n \wedge 
		\sigma(s,\lambda), \lambda).
		\end{equation*}
		Observe that by Corollary \ref{optimal stopping3},
		$$\left\{\tau_n < \sigma(s,\lambda) \right\} = 
		\{\gamma(s+\theta) \leq \lambda \;\; \forall \theta \leq \tau_n\} 
		\subseteq 
		\bigcap_{k=1}^{n-1} \left\{\gamma(s+\tau_k) \leq \lambda \right\} 
		\subseteq \{\tau_n \geq n \}.$$
		Hence, it follows that $\tau_n \wedge 
		\sigma(s,\lambda) \to \sigma(s,\lambda)$ as $n \to \infty$ and thus 
		$$\sum_{t=s+1}^{s+\tau_n \wedge \sigma(s,\lambda) }\beta^{t} \left(h( 
		t) - \lambda\right) \longrightarrow \sum_{t=s+1}^{s+\sigma(s,\lambda) 
		}\beta^{t} \left(h(t) - \lambda\right) \;\;\; \text{as} \;\; n \to 
		\infty \;\; \text{for all} \; \omega \in \Omega.$$
		As $h$ is bounded, it follows from the Lebesgue property of our 
		nonlinear expectation that
		$$\mathcal{E} \bigg(\sum_{t=s+1}^{s+\tau_n \wedge 
			\sigma(s,\lambda)}\beta^{t} \big(h(t) - \lambda\big)  \; \bigg| \; 
		\mathcal{F}_{s}\bigg) \longrightarrow \mathcal{E} 
		\bigg(\sum_{t=s+1}^{s+\sigma(s,\lambda)}\beta^{t} \big(h(t) - 
		\lambda\big)  \; \bigg| \; \mathcal{F}_{s}\bigg). $$
		In particular,
		$V_s(\tau, \lambda) \geq V_s(\sigma(s,\lambda), \lambda).$
	\end{proof}
\end{theorem}
\begin{remark}
	Bank and El Karoui \cite{Bank_El_karoui_rep_thm} consider a similar 
	result to this theorem, but under a classical expectation with the 
	summation $\sum_{t=s+1}^{s+\tau} \beta^{t} \left(h(t) - \gamma(s)\right)$ 
	replaced by a more general function in continuous time. (See also 
	\cite{Running_max_optional_stopping} and \cite{Bank_proof_Gittins} for 
	further discussion).
\end{remark}
\subsection{Step A.3: Fair Game and Prevailing process}
Previously, we considered an optimal stopping problem when the Gittins index 
is offered as compensation for continued play. In this subsection, we consider 
a `fair game' when we 
offer a compensation which is (just) sufficient to encourage us to 
continue playing the bandit. In particular, the 
compensation increases at each optimal stopping time in order to encourage the 
agent to 
continue.

We will first define a sequence of optimal stopping times that we have to 
consider in order to analyze our (minimal) compensation process.
\begin{definition}
	\label{optimal sequence}
	We define $\hat{S}_n$ to be the stopping time where the Gittins index 
	process $\left(\gamma(s)\right)_{s \geq 0}$ exceeds its running maximum for 
	the $n$th time. We write $\sigma_n$  for the duration between $\hat{S}_n$ 
	and $\hat{S}_{n+1}$, that is, $\sigma_n$ is a random time identifying how 
	long after time $\hat{S}_n$ the process $\left(\gamma(s)\right)_{s \geq 0}$ 
	hits a new maximum. 
	
	More precisely, we define $\hat{S}_n$ and $\sigma_n$ inductively:
	\begin{enumerate}[(i)]
		\item  Let $\hat{S}_0 := 0$.
		\item  Given $\hat{S}_n$, define 
		$$\sigma_n := \inf\{\theta \geq 1 \; : \; \gamma(\hat{S}_n + \theta) 
		>
		\gamma(\hat{S}_n) \}$$
		and $\hat{S}_{n+1} := \hat{S}_n + \sigma_n$.
	\end{enumerate}
	Equivalently, we can define $\sigma_n := \sigma(\hat{S}_n, 
	\gamma(\hat{S}_n))$ as in Definition \ref{optimal stop upto lambda}.
\end{definition}
\begin{definition}
	\label{prevailing process}
	We define the \emph{prevailing reward} process $\Gamma$ by the running 
	maximum of $\gamma$, that is,
	$$\Gamma(t) := \max_{0 \leq \theta \leq t-1} \gamma(\theta).$$	
\end{definition}
We can then show that the process $\Gamma$ serves as an indifference reward 
(process) for our agent, when evaluated from the perspective of one of the 
stopping times $\hat{S}_n$.
\begin{proposition}
	\label{fair game}
	For all $n \in \mathbb{N}$,
	$$\mathcal{E}\bigg(\sum_{t=\hat{S}_n + 1}^\infty \beta^t \left(h(t) - 
	\Gamma(t)\right) \bigg| \mathcal{F}_{\hat{S}_n}\bigg) = 0.$$
	In particular, 
	$$\mathcal{E}\bigg(\sum_{t=1}^\infty \beta^t \big(h(t) - \Gamma(t)\big) 	
	\bigg) = 0.$$ 	
	\begin{proof}
		By Theorem \ref{optimal stopping}, we have, for all $k \in \mathbb{N}$,
		$$0  =   \mathcal{E}\bigg(\sum_{t = \hat{S}_k + 1}^{\hat{S}_{k+1}} 
		\beta^t \big(h(t) - \gamma(\hat{S}_n)\big)\bigg| 
		\mathcal{F}_{\hat{S}_k}\bigg).$$
		Fix $n,N \in \mathbb{N}$ with $N \geq n$, by time-consistency, 
		translation equivariance, 
		\begin{align*}
		&\mathcal{E}\bigg(\sum_{t = {\hat{S}_n} + 1}^{\hat{S}_{N}} \beta^t 
		\big(h(t) - \Gamma(t)\big) \bigg| \mathcal{F}_{\hat{S}_n}\bigg) \\ 
		& =  \mathcal{E}\Bigg(\sum_{t = {\hat{S}_n} + 1}^{\hat{S}_{N-1}} 
		\beta^t \big(h(t) - \Gamma(t)\big) + \mathcal{E}\bigg(\sum_{t = 
			\hat{S}_{N-1} + 1}^{\hat{S}_{N}} \beta^t \big(h(t) - 
		\Gamma(t)\big)\bigg| \mathcal{F}_{\hat{S}_{N-1}}\bigg) \Bigg| 
		\mathcal{F}_{\hat{S}_n}\Bigg) \\
		& =  \mathcal{E}\Bigg(\sum_{t = {\hat{S}_n} + 1}^{\hat{S}_{N-1}} 
		\beta^t \big(h(t) - \Gamma(t)\big) + \mathcal{E}\bigg(\sum_{t = 
			\hat{S}_{N-1} + 1}^{\hat{S}_{N}} \beta^t \big(h(t) - 
		\gamma(\hat{S}_{N-1})\big)\bigg| \mathcal{F}_{\hat{S}_{N-1}}\bigg) 
		\Bigg| \mathcal{F}_{\hat{S}_n} \Bigg)\\
		& =  \mathcal{E}\bigg(\sum_{t = {\hat{S}_n} + 1}^{\hat{S}_{N-1}} 
		\beta^t \big(h(t) - \Gamma(t)\big) \bigg| 
		\mathcal{F}_{\hat{S}_n}\bigg)  =  \mathcal{E}\bigg(\sum_{t = 
			{\hat{S}_n} + 1}^{\hat{S}_{N-2}} \beta^t \big(h(t) - \Gamma(t)\big) 
		\bigg| \mathcal{F}_{\hat{S}_n}\bigg)\\ 
		&=  \cdots  =  0.
		\end{align*}	
		By our definition of $\hat{S}_n$, we have $\hat{S}_N \geq N$. Hence, 
		$\hat{S}_N \to \infty$ as $N \to \infty$. Therefore, by applying 
		Lebesgue property, the result follows. 
	\end{proof}
\end{proposition}
Intuitively, as $\Gamma(t) \geq \gamma(t-1)$, the process $\Gamma$ should be 
sufficient to compensate for continuing to play. This means that the total 
`expected' loss, evaluated from any point in time, must be non-positive if a 
reward $\Gamma(t)$ is offered. This is stated formally in the following lemma 
and theorem.
\begin{lemma}
	\label{recovering reward}
	Let $\tau \in \mathcal{T}(s)$ with $1 \leq \tau \leq \sigma := 
	\sigma(s,\lambda)$. 
	Then 
	$$\mathcal{E}\bigg(\sum_{t=s+\tau+1}^{s+\sigma} \beta^t \left(h(t) - 
	\lambda 
	\right) \bigg| \mathcal{F}_{s+\tau}\bigg) \leq 0. $$
	\begin{proof}
		Write $H_\tau := \mathcal{E}\left(\sum_{t=s+\tau+1}^{s+\sigma} \beta^t 
		\Big(h(t) - \lambda 
		\Big) \Big| \mathcal{F}_{s+\tau}\right)$ and $A := 
		\left\{H_\tau > 0 \right\}$. 
		
		Define $\tilde{\sigma} := \tau \; \mathbb{I}_A + \sigma \; 
		\mathbb{I}_{A^c}$.
		\begin{align*}
		&\mathcal{E}\bigg(\sum_{t=s+1}^{s+\sigma} \beta^t \left(h(t) 
		- 
		\lambda 
		\right) \bigg| \mathcal{F}_{s+\tau}\bigg)\\
		& \qquad= \sum_{t=s+1}^{s+\tau} 
		\beta^t \left(h(t) - 
		\lambda \right) + 	\mathcal{E}\bigg(\sum_{t=s+\tau+1}^{s+\sigma} 
		\beta^t 
		\left(h(t) - 
		\lambda 
		\right) \bigg| \mathcal{F}_{s+\tau}\bigg) \\
		& \qquad \geq  \sum_{t=s+1}^{s+\tau} 
		\beta^t \left(h(t) - 
		\lambda \right) + 	\mathcal{E}\bigg(\sum_{t=s+\tau+1}^{s+\sigma} 
		\beta^t 
		\left(h(t) - 
		\lambda 
		\right) \bigg| \mathcal{F}_{s+\tau}\bigg) \mathbb{I}_{A^c} \\
		& \qquad = \mathcal{E}\bigg(\sum_{t=s+1}^{s+\tilde{\sigma}} \beta^t 
		\left(h(t) - 
		\lambda 
		\right) \bigg| \mathcal{F}_{s+\tau}\bigg).
		\end{align*}
		
		Moreover, the above inequality is strict on $A$. Hence, if $A$ is not a 
		$\mathbb{P}$-null set, it then follows from strict monotonicity that
		$$	\mathcal{E}\bigg(\sum_{t=s+1}^{s+\sigma} \beta^t \left(h(t) - 
		\lambda 
		\right) \bigg| \mathcal{F}_{s}\bigg)  > 	
		\mathcal{E}\bigg(\sum_{t=s+1}^{s+\tilde{\sigma}} 
		\beta^t 
		\left(h(t) 	- 	\lambda    \right) \bigg| \mathcal{F}_{s} \bigg). $$
		This contradicts the minimality of $\sigma(s,\lambda)$ established in 
		Theorem \ref{optimal stopping}.
	\end{proof}
\end{lemma}
\begin{theorem}
	\label{expected recovering reward}
	For all $N \in \mathbb{N}$,
	$$\mathcal{E}\bigg(\sum_{t=N+1}^\infty \beta^t \big(h(t) - \Gamma(t)\big) 
	\bigg| \mathcal{F}_N  \bigg) \leq 0.$$
	\begin{proof}
		Define $\tau_n := (\hat{S}_{n+1} \wedge N) \vee \hat{S}_n$. Since 
		$\hat{S}_n$ is a stopping time for all $n \in \mathbb{N}$, so is 
		$\tau_n$. 
		Hence, by Proposition \ref{fair game} and Lemma \ref{recovering reward},
		\begin{align*}
		&\mathcal{E}\bigg(\sum_{t=\tau_n + 1 }^\infty \beta^t \big(h(t) 
		- \Gamma(t)\big) \bigg| \mathcal{F}_{\tau_n}  \bigg)\\ 
		& = \mathcal{E}\Bigg(\sum_{t=\tau_n + 1 }^{\hat{S}_{n+1}} \beta^t 
		\big(h(t) - \Gamma(t)\big) + 
		\mathcal{E}\bigg(\sum_{t=\hat{S}_{n+1} 
			+1}^{\infty} \beta^t \big(h(t) - \Gamma(t)\big)  \bigg| 
		\mathcal{F}_{\hat{S}_{n+1}} \bigg) \Bigg| \mathcal{F}_{\tau_n}  
		\Bigg) 
		\\
		& =  \mathcal{E}\bigg(\sum_{t=\tau_n + 1 }^{\hat{S}_{n+1}} \beta^t 
		\big(h(t) - \Gamma(t)\big)\bigg| \mathcal{F}_{\tau_n}  \bigg) \; 
		= 
		\; \mathcal{E}\bigg(\sum_{t=\tau_n + 1 }^{\hat{S}_{n+1}} \beta^t 
		\big(h(t) - \gamma(\hat{S}_n)\big) \bigg| \mathcal{F}_{\tau_n}  
		\bigg)	\leq 0.	
		\end{align*}
		Therefore, as $\{\hat{S}_n \leq N < \hat{S}_{n+1}\}$ is 
		$\mathcal{F}_{N}$-measurable, by Lebesgue property and regularity 
		(Remark \ref{remark: regularity}),
		\begin{align*}
		&\mathcal{E}\bigg(\sum_{t=N + 1 }^\infty \beta^t \big(h(t) 
		- 				\Gamma(t)\big) \bigg| \mathcal{F}_N \bigg) \\ 
		&= \mathcal{E}\bigg(\lim_{L\to 
		\infty}\bigg(\sum_{n=0}^{L}\mathbb{I}_{\{\hat{S}_n \leq N < 
			\hat{S}_{n+1}\}}\bigg)\bigg(\sum_{t=\tau_n + 1 }^\infty \beta^t 
		\big(h(t) - \Gamma(t)\big)\bigg) \bigg| \mathcal{F}_N \bigg)\\ 
		&= \lim_{L\to \infty}\sum_{n=0}^{L}\mathbb{I}_{\{\hat{S}_n \leq N < 
			\hat{S}_{n+1}\}}\mathcal{E}\bigg(\sum_{t=\tau_n + 1 }^\infty 
		\beta^t 
		\big(h(t) - \Gamma(t)\big) \bigg| \mathcal{F}_N \bigg)\\ 
		&= \lim_{L\to \infty}\sum_{n=0}^{L}\mathbb{I}_{\{\hat{S}_n \leq N < 
			\hat{S}_{n+1}\}}\mathbb{I}_{\{\tau_n \geq 
			N\}}\mathcal{E}\bigg(\sum_{t=\tau_n + 1 }^\infty \beta^t 
		\big(h(t) - \Gamma(t)\big) \bigg| \mathcal{F}_N \bigg)\\ 
		& =  \lim_{L\to \infty}\sum_{n=0}^{L}\mathbb{I}_{\{\hat{S}_n \leq N < 
			\hat{S}_{n+1}\}}\mathbb{I}_{\{\tau_n \geq 
			N\}}\mathcal{E}\bigg(\mathcal{E}\bigg(\sum_{t=\tau_n + 1 }^\infty 
		\beta^t 
		\big(h(t) - \Gamma(t)\big) \bigg| \mathcal{F}_{\tau_n} \bigg) 
		\bigg| 
		\mathcal{F}_N \bigg) \\
		& \leq  0
		\end{align*}
	\end{proof}	
\end{theorem}
\begin{remark}
	The above theorem says that, with compensation $\Gamma(t)$, at any point in 
	time we expect to 
	obtain a net reward from continuing to play, i.e. we have a non-positive 
	expected 
	total loss. 
\end{remark}
\subsection{Step A.4: Reward Delay and Robust Representation Theorem}
In Step A.3, we have shown that our reward $\Gamma$ is defined to be 
(just) sufficient to encourage the player to 
continue playing (Theorem \ref{expected recovering reward}) until the 
horizon (i.e. the total expected loss is zero, as in Proposition 
\ref{fair game}). We now show that taking a 
break from play cannot improve a player's expected discounted costs. We now formulate this 
observation by establishing the existence of a 
probability measure in our representing set $\mathcal{Q}$ such that the expected discounted costs, accounting for the break in play, have a lower bound close to zero. This 
result will be useful when considering multiple bandits. 

\begin{theorem}
	\label{delay and prevailing} By Assumption \ref{assump:existence of nonlinear}, recall that $\mathcal{E}$ admits a robust representation of the form
	$$\mathcal{E}\big( \cdot \big) = \sup_{\mathbb{Q} \in \mathcal{Q}}\mathbb{E}^\mathbb{Q}( \cdot).$$
	For every fixed $\epsilon > 0$, there exists a probability measure 
	$\mathbb{Q} \in \mathcal{Q}$ such that for every predictable decreasing process 
	$(\alpha(t))_{t \geq 0}$ taking values in $[0,1]$, 
	we have 
	$$\mathbb{E}^{\mathbb{Q}}\bigg(\sum_{t=1}^{\infty} \alpha(t)\beta^t 
	\big(h(t) - \Gamma(t) 
	\big) \bigg)  \geq -\epsilon.$$
	\begin{proof}
		By Proposition \ref{fair game} and the robust representation theorem, 
		for a fixed $\epsilon > 0$, we can find a probability measure 
		$\mathbb{Q} \in \mathcal{Q}$ such that
		$$\mathbb{E}^{\mathbb{Q}}\bigg(\sum_{t=1}^{\infty} \beta^t 
		\big(h(t) -\Gamma(t) 
		\big) \bigg)  \geq -\epsilon.$$
		For each predictable decreasing process 
		$(\alpha(t))_{t \geq 0}$ taking values in $[0,1]$, we define  
		\begin{equation*}
		\alpha^N(t) := \begin{cases}
		\alpha(t) & \text{for }  t \leq N, \\
		\alpha(N) & \text{for }  t > N .
		\end{cases}
		\end{equation*}
		We claim that
		\begin{equation}
		\label{induction alpha}
		\mathbb{E}^{\mathbb{Q}}\bigg(\sum_{t=1}^{\infty} \alpha^N(t) \beta^t 
		\big(h(t) - \Gamma(t) \big) \bigg)  \geq -\epsilon.
		\end{equation}			
		Indeed, it is clear that the result holds when $N=0$. 
		
		For the sake of induction, assume that the result holds for a given 
		$N$. We 
		then have
		\begin{equation}
		\label{inequality for delay}
		-\epsilon \leq  
		\mathbb{E}^{\mathbb{Q}}\bigg(\sum_{t=1}^{N} 
		\alpha(t)  \beta^t 
		\big(h(t) -	\Gamma(t) 
		\big) \bigg)  + 
		\mathbb{E}^{\mathbb{Q}}\bigg(\sum_{t=N+1}^{T} 
		\alpha(N)  \beta^t 
		\big(h(t)-\Gamma(t) 	\big) \bigg).
		\end{equation}							
		By the robust representation theorem (Theorem \ref{Thm:dynamic robust rep}), 
		$$\mathcal{E}\big(\; \cdot \; \big| \mathcal{F}_N \big) = 
		\esssup_{\mathbb{Q} \in 
			\mathcal{Q}}\mathbb{E}^\mathbb{Q}\big(\; \cdot \; \big| 
		\mathcal{F}_N\big).$$
		By Theorem \ref{expected recovering reward}, we know that
		$$\mathbb{E}^{\mathbb{Q}}\bigg(\sum_{t=N+1}^{\infty} \beta^t	\big(h(t)-\Gamma(t)\big) \bigg| \mathcal{F}_N \bigg) \leq	0.$$
		Since $\alpha$ is decreasing,
		$$\big(\alpha(N) - \alpha(N+1) \big)\mathbb{E}^{\mathbb{Q}}\bigg(\sum_{t=N+1}^{\infty} 
		\beta^t 
		\big(h(t) 
		- 
		\Gamma(t) 	\big) \bigg| \mathcal{F}_N \bigg) \leq 
		0.$$
		As $\alpha$ is predictable, by rearranging the above inequality, we obtain
		$$\mathbb{E}^{\mathbb{Q}}\bigg(\sum_{t=N+1}^{\infty} 
		\alpha(N)\beta^t 
		\big(h(t) 
		- 
		\Gamma(t) 	\big) \bigg| \mathcal{F}_N \bigg) \leq 
		\mathbb{E}^{\mathbb{Q}}\bigg(\sum_{t=N+1}^{\infty} 
		\; \alpha(N+1)\beta^t 
		\big(h(t) 
		- 
		\Gamma(t) 	\big) \bigg| \mathcal{F}_N \bigg).$$
		Hence, by the tower property,
		$$\mathbb{E}^{\mathbb{Q}}\bigg(\sum_{t=N+1}^{\infty} 
		\alpha(N)\beta^t 
		\big(h(t) 
		- 
		\Gamma(t) 	\big) \bigg) \leq 
		\mathbb{E}^{\mathbb{Q}}\bigg(\sum_{t=N+1}^{\infty} 
		\alpha(N+1)\beta^t 
		\big(h(t) 
		- 
		\Gamma(t) 	\big)  \bigg).$$
		By substituting this into \eqref{inequality for delay}, 
		we prove \eqref{induction alpha} with $N$ replaced by $N+1$ and done the induction step. 
		
		By using bounded convergence theorem, we can take $N \to \infty$ and obtain the required result.
	\end{proof}
\end{theorem}

\section{Part B: Analysis of multiple bandits}
\label{Multiarm section}
We are now ready to consider the problem of choosing between multiple bandits.

In Definition \ref{admissible control}, we introduce our class of admissible control which can be considered in our dynamic allocation problems. This class of control introduces a few natural ways of parameterizing time. We 
therefore will use the following terminology to describe the evolution of time 
in different ways. This terminology will be useful in our discussion on the proof.
\begin{enumerate}
	\item \textbf{`Play'} refers to the \underline{total} number of (real) 
	times that we play the system of bandits. (This corresponds to the time 
	parameter of the simple form $\rho$ which is briefly discussed earlier on in 
	Remark \ref{Mandelbaum and simple seq} and later in Definition \ref{simple form}.)
	\item \textbf{`Trial'} refers to the number of times that we play a specific bandit. (This corresponds to the sum $\sum_{k=0}^{K}\tau^{(m)}_k$.)
	\item \textbf{`Decision'} refers to the number of times that we make a 
	decision between bandits. (This corresponds to the time parameter for the 
	choice process $\left(p_n\right)_{n \geq 0}$.)
	\item \textbf{`Run'} refers to the number of times that we have made the decision to select a specific bandit. (This corresponds to the time parameter of the allocation sequence $(\tau^{(m)}_k)_{k \geq 0}$ for each fixed $m \in \mathcal{M}$.)
\end{enumerate}
\begin{remark}
	The terms `play' and `trial' can be referred to without directly identifying 
	the time allocation sequence. On the other hand, the terms `decision' and 
	`run' need to be interpreted under a given time allocation sequence  $\tau$ (Definition \ref{time allocation sequences}).
\end{remark}
\begin{remark}
	The $m$th component of the recording sequence $\eta_n$ (Definition \ref{recording sequences}) represents the number 
	of runs in the $m$th
	bandit before the $n$th decision. 	We can see that the random variable $\eta_{n}$ takes values in 
	$\Big\{r \in \mathcal{S} : \sum_{m=1}^M r^{(m)} \leq n \Big\}$. 
\end{remark}

In order to prove C-optimality, we then consider the target function as briefly stated in \eqref{target for C-optimality}.
\begin{definition}
	\label{Gittins' target function}
	For each $m \in \mathcal{M}$, let 
	$\big(h^{(m)}(t)\big)_{t \geq 1}$ be the uniformly bounded non-negative cost process at the $t$th trial of the $m$th bandit with prevailing reward process $\big(\Gamma^{(m)}(t)\big)_{t \geq 1}$ (Definition \ref{prevailing process}). For an allocation strategy $(\tau,p)$, (Definition \ref{admissible control}), we define the \emph{Gittins' target function} by
	$$V({\tau}, p) := 
	\mathfrak{E}\bigg(\sum_{n=1}^{\infty} 
	\beta^{n} \big(h^{({\rho}_{n-1})}(t^{\rho}_n) - 
	\Gamma^{({\rho}_{n-1})}(t^{\rho}_n)\big) 
	\bigg)$$
	where $\rho$ is the simple form of $({\tau}, p)$ with corresponding counting processes $t^{\rho}_n := \sum_{k=0}^{n-1} \mathbb{I}(\rho_k = 
	\rho_{n-1})$, and $\mathfrak{E}$ is a partially consistent orthant nonlinear expectation, as in Definition \ref{orthant expectation}.
\end{definition}

\begin{remark}
	We can also write $V({\tau}, p)$ in terms of ${\tau}$ and $p$ directly without identifying the simple form ${\rho}$. This is done in the proof of Theorem \ref{negativity} in Step B.2. This definition, however, makes it clear that $V$ depends on $(\tau,p)$ only through its simple form.
\end{remark}

\subsection{Step B.1: Fubini theorem and Suboptimality}
In this subsection, we will show that considering generic stopping times and choice of bandits yields a non-negative expected loss. This can be shown using the robust representation result.

First, we recall the following corollary of Fubini's theorem.
\begin{corollary}
	\label{conditional expectation}
	Let  $(G,\mathcal{G}, \mathbb{P})$ and $(H, \mathcal{H}, \mathbb{Q})$ be 
	probability spaces. Let $\mathcal{G}'$ and $\mathcal{H}'$ be sub $\sigma$-algebras of 
	$\mathcal{G}$ and  $\mathcal{H}$ respectively, with $\mathcal{H}':= \{\emptyset, 
	H\}$. Then, for any integrable 
	random variable $X$ on $(G 
	\times H, \mathcal{G} \otimes \mathcal{H}, \mathbb{P} \otimes \mathbb{Q})$, 
	we have
	$$\mathbb{E}^{\mathbb{P} \otimes \mathbb{Q}} \Big(X \; \Big| \; 
	\mathcal{G}' 
	\otimes \mathcal{H}'\Big) = \mathbb{E}^{\mathbb{P}}\Big(\int_H X(\; 
	\cdot \;, h) 
	d\mathbb{Q}(h) \; \Big| \; 
	\mathcal{G}' 
	\Big) \;\;\; \mathbb{P} \otimes \mathbb{Q}\text{-a.s.}$$	
\end{corollary}
\begin{theorem}
	\label{positivity}
	For any 
	allocation strategy $(\tau, p)$, we have
	$$V(\tau, p) \geq 0.$$
	
	\begin{proof}
		Let $\rho$ be the simple form of $(\tau, p)$. Write $R^{(m)}_t$ for the total number of trials on other bandits before making the $t$th trial on the $m$th bandit, i.e. 
		$$R^{(m)}_t := \sum_{k \neq m} \sum_{n=0}^{N^{(m)}_t} \mathbb{I}(\rho_n = k) \quad\text{where} \quad N^{(m)}_t := \inf\bigg\{N \geq 0 : \sum_{n=0}^N\mathbb{I}(\rho_n = m) = t \bigg\}.$$
		Since $R^{(m)}_t$ does not depend on future realizations of the $m$th bandit, $R^{(m)}_t$ is 
		$\mathcal{F}^{(m)}_{t-1} \otimes \big(\bigotimes_{k=1, k \neq m}^M 
		\mathcal{F}^{(k)}_{\infty}\big)$-measurable (taking the product in an appropriate order). 
		Moreover, as $N^{(m)}_t$ is increasing in $t$, it follows that $R^{(m)}_t$ is increasing in $t$.
		
		Now, fix $\epsilon > 0$. By Theorem \ref{delay and prevailing}, for each $m \in \mathcal{M}$, we can find 
		a probability measure $\mathbb{Q}^{(m)} \in \mathcal{Q}^{(m)}$ such 
		that, for every adapted decreasing process $(\alpha^{(m)}(t))_{t \geq 
			0}$ taking 
		values in $[0,1]$, we have
		\begin{equation}
		\label{alpha inequality}
		\mathbb{E}^{\mathbb{Q}^{(m)}}\bigg(\sum_{t=1}^{\infty} 
		\alpha^{(m)}(t)\beta^t \big(h^{(m)}(t) - \Gamma^{(m)}(t) \big) \bigg)  \geq -\epsilon.
		\end{equation}
		
		Define 
		$$\tilde{\alpha}^{(m)}(t) := \int_{\prod_{k \neq m}\Omega^{(k)}} 
		\; \beta^{R^{(m)}_t} \; d \Big(\bigotimes_{k \neq m}\mathbb{Q}^{(k)}\Big). $$
		By Fubini's theorem, as $R^{(m)}_t$ is 
		$\mathcal{F}^{(m)}_{t-1} \otimes \big(\bigotimes_{k=1, k \neq m}^M 
		\mathcal{F}^{(k)}_{\infty}\big)$-measurable and $\beta \in (0,1]$, the process $\big(\tilde{\alpha}^{(m)}(t)\big)$ is an 
		$\big(\mathcal{F}^{(m)}_t\big)$-predictable process taking values in 
		$[0,1]$.
		
		Moreover, $R^{(m)}_t$ is also
		$\mathcal{F}^{(m)}_{\infty} \otimes \big(\bigotimes_{k=1, k \neq m}^M 
		\mathcal{F}^{(k)}_{\infty} \big)$-measurable, so by Corollary \ref{conditional expectation} we can write
		\begin{eqnarray*}
			\tilde{\alpha}^{(m)}(t) & = & \mathbb{E}^{\bigotimes_{k=1}^M 
				\mathbb{Q}^{(k)}}\Big( \beta^{R^{(m)}_t} \; 
			\Big| 
			\tilde{\mathcal{F}}^{(m)}_{\infty} \Big) \quad \text{where} \;\;\; \tilde{\mathcal{F}}^{(m)}_{\infty}  :=  \mathcal{F}^{(m)}_{\infty} 
			\otimes \bigotimes_{k \neq m}^M 
			\mathcal{F}^{(k)}_{0} .
		\end{eqnarray*} 
		As $t \mapsto R^{(m)}_t$ is 
		increasing, it then follows that $\tilde{\alpha}^{(m)}(t) $ is decreasing in $t$. Hence, by Theorem \ref{delay and prevailing} we obtain \eqref{alpha inequality} with $\alpha^{(m)}$ replaced with $\tilde{\alpha}^{(m)}$.
		
		By the definition of $\mathfrak{E}$ and Fubini's theorem, it follows that
		\begin{eqnarray*}
			\lefteqn{\mathfrak{E}\bigg( \sum_{n=1}^{\infty} 
				\beta^{n} \left(h^{({\rho}_{n-1})}(t^{\rho}_n) - 
				\Gamma^{({\rho}_{n-1})}(t^{\rho}_n)\right) \bigg)}
			\\ & \geq &
			\mathbb{E}^{\bigotimes_{k=1}^M \mathbb{Q}^{(m)}}\bigg( 
			\sum_{n=1}^{\infty} 
			\beta^{n} \left(h^{({\rho}_{n-1})}(t^{\rho}_n) - 
			\Gamma^{({\rho}_{n-1})}(t^{\rho}_n)\right)\bigg)
			\\		& = & \mathbb{E}^{\bigotimes_{k=1}^M 
				\mathbb{Q}^{(m)}}\bigg(\sum_{m=1}^{M} 
			\sum_{t=1}^{\infty} 
			\beta^{R^{(m)}_t}\beta^t \left(h^{(m)}(t) - 
			\Gamma^{(m)}(t)\right) \bigg)
			\\		& = & \sum_{m=1}^{M} \mathbb{E}^{\bigotimes_{k=1}^M 
				\mathbb{Q}^{(m)}}\bigg(
			\sum_{t=1}^{\infty} 
			\mathbb{E}^{\bigotimes_{k=1}^M 
				\mathbb{Q}^{(m)}}\Big(\beta^{R^{(m)}_t}\beta^t \left(h^{(m)}(t) - 
			\Gamma^{(m)}(t)\right)
			\Big| \tilde{\mathcal{F}}^{(m)}_{\infty} \Big) \bigg) \\
			& = & \sum_{m=1}^{M} \mathbb{E}^{\bigotimes_{k=1}^M 
				\mathbb{Q}^{(m)}}\bigg(
			\sum_{t=1}^{\infty} 
			\tilde{\alpha}^{(m)}(t) \beta^t\left(h^{(m)}(t) - 
			\Gamma^{(m)}(t)\right)\bigg) \\
			& = & \sum_{m=1}^{M} \mathbb{E}^{\mathbb{Q}^{(m)}}\bigg(
			\sum_{t=1}^{\infty} 
			\tilde{\alpha}^{(m)}(t) \beta^t\left(h^{(m)}(t) - 
			\Gamma^{(m)}(t)\right)\bigg) \;\;\; \geq \;\;\; -M\epsilon.
		\end{eqnarray*}
		As $\epsilon$ is arbitrary, the result follows.
	\end{proof}
\end{theorem}
\subsection{Step B.2: Optimality}
In this subsection, we will show that the strategy determined by a particular time allocation sequence yields a zero expected cost in the Gittins' target function. 
\begin{theorem}
	\label{negativity}
	For each  $m \in \mathcal{M}$, let $(\sigma^{(m)}_k)_{k \geq 0}$ be the 
	sequence of running maximum random times associated to the $m$th bandit, as 
	defined in Definition \ref{optimal sequence}, i.e. we define 
	$\sigma^{(m)}_k $ and $\hat{S}^{(m)}_k$ recursively by
	$S^{(m)}_{k} := \sum_{l=0}^{k-1} 
	\sigma^{(m)}_{l}$ and
	\begin{equation*}
	\sigma^{(m)}_{k} := 
	\inf\left\{\theta \geq 1 : \gamma^{(m)}(S^{(m)}_{k} + \theta) 
	>  
	\gamma^{(m)}(S^{(m)}_{k}) \right\}.
	\end{equation*} 
	Then for any allocation 
	strategy of the form $(\sigma, p)$, we have
	$$V(\sigma, p) \leq 0.$$
	\begin{proof}
		Recall the recording sequence $\eta_n$ associated with
		$(\sigma, p)$.  	
		We define the following notation, given an allocation sequence $\sigma$.
		\begin{enumerate}[$\bullet$]
			\item  $\tilde{\Theta}_n$ denotes the total number of plays of the 
			system 
			before making the $n$th decision. i.e.
			$$\tilde{\Theta}_n := \sum_{m=1}^{M} \sum_{i=0}^{\eta_n^{(m)}-1} \sigma^{(m)}_i.$$
			\item $\tilde{\sigma}_n$ denotes the duration we decide to play following the $n$th 
			decision time, i.e.
			$$\tilde{\sigma}_n := \sigma^{(m)}_k \qquad	\text{on the event } \{p_n = m, \eta^{(m)}_n = k\}. $$
			N.B. $p_n = m$ means that we decide to play the $m$th bandit at the $n$th decision. The event $\eta^{(m)}_n = k$ means that we have had $k$ runs of the $m$th bandit before the $n$th decision. Thus, we choose to make $\sigma^{(m)}_k$ more trials on this bandit before making another decision.
			\item $\tilde{\Psi}^{(m)}_n$ denotes the total number of trials on the $m$th bandit before making the $n$th decision, i.e. 
			$$\tilde{\Psi}^{(m)}_n := \sum_{i=0}^{k-1} \sigma^{(m)}_i  \qquad
			\text{on the event } \{\eta^{(m)}_n = k\}. $$
		\end{enumerate}
		Using this notation, we can define a variation on the Gittins' target fuction, with the restriction that we consider only the first $N$ plays of the system, that is,
		$$V(N,\sigma, p) := \mathfrak{E}\bigg(\sum_{n=0}^{N-1} 
		\beta^{\tilde{\Theta}_n} \bigg(\sum_{l=1}^{\tilde{\sigma}_n}\beta^l
		\left(h^{(p_n)}(\tilde{\Psi}^{(p_n)}_n + l) - 
		\Gamma^{(p_n)}(\tilde{\Psi}^{(p_n)}_n + l)\right)\bigg) 
		\bigg)$$
		with the convention $h^{(0)}(t) = \Gamma^{(0)}(t) = 0$ for all $t$.
		
		By considering the simple form $\rho$ of the strategy $(\sigma,p)$ and applying Lebesgue property of $\mathfrak{E}$, we can show that $V$ agrees with Definition \ref{Gittins' target function} as $N \to \infty$, that is,
		$$\lim_{N \to \infty} V(N,\sigma, p) = V(\sigma, p).$$
		Hence, it suffices to show that 
		$V(N,\sigma, p) \leq 0$ for all $N \in \mathbb{N}$. This will be proved by 
		induction.
		
		It is clear that $V(0, \sigma, p) = 0$. Fix $N \in \mathbb{N}$ and assume that $V(N,\sigma, p) \leq 0$. To show that $V(N+1, \sigma, p) \leq 0$, by subadditivity, it suffices to show that
		$$\mathfrak{E}\bigg(\beta^{\tilde{\Theta}_N} \bigg(\sum_{l=1}^{\tilde{\sigma}_N}\beta^l
		\left(h^{(p_N)}(\tilde{\Psi}^{(p_N)}_N + l) - 
		\Gamma^{(p_N)}(\tilde{\Psi}^{(p_N)}_N + l)\right)\bigg) 
		\bigg)	\leq 0.$$
		
		Define the following random variables as in \eqref{decision filtration}:
		\begin{equation}\label{eq:PsiThetadefn} 
		\Psi_r^{(m)} := \sum_{i=0}^{r^{(m)}  - 1} \sigma^{(m)}_i \qquad \text{and} 
		\qquad
		\Theta_r := \sum_{m=1}^{M}\Psi_r^{(m)}
		\end{equation}
		
		for
		\[r \in \underline{\mathcal{S}}_N := \left\{r \in \mathcal{S} : \sum_{m=1}^M r^{(m)} 
		\leq N,\; r^{(m)} \leq T^{(m)} \right\}.\] 
		Note that  $\Psi_r^{(m)} = \tilde{\Psi}_N^{(m)}$ and $\Theta_r = \tilde{\Theta}_N$ on the event $\{\eta_N = r\}$. From the definition of the recording sequence (Definition \ref{recording sequences}), it 
		follows that $\Theta_r$ and the event 
		$A_N^{(r,m)} := \{\eta_N = r, p_N = m\}$ are both $\mathcal{F}(\Psi_r)$-measurable.
		
		On an event $A_N^{(r,m)}$, $\tilde{\Psi}_N^{(m)} = \Psi_r^{(m)}$ and $\Psi_r^{(m)}$ is a stopping time with respect to the filtration $(\mathcal{F}^{(m)}_t)_{t \geq 0}$. By considering the optimality obtained in Theorem \ref{optimal 
			stopping}, we can show that
		\begin{align*}
		&\mathfrak{E}\bigg(\beta^{\Theta_N} 
		\bigg(\sum_{l=1}^{\tilde{\sigma}_N}\beta^l
		\left(h^{(p_N)}(\Psi^{(m)}_N + l) - 
		\Gamma^{(p_N)}(\Psi^{(m)}_N + l)\right)\bigg) \bigg)\\ 
		& =  \mathfrak{E}\bigg(\sum_{m=1}^{M} \sum_{r \in \underline{\mathcal{S}}_N} 
		\mathbb{I}_{A_N^{(r,m)}}\beta^{\Theta_r} 
		\bigg(\sum_{l=1}^{\sigma_r(m)}\beta^l
		\left(h^{(m)}(\Psi^{(m)}_r + l) - 
		\Gamma^{(m)}(\Psi^{(m)}_r + l)\right)\bigg) 
		\bigg) \\
		& =  \mathfrak{E}\bigg(\sum_{m=1}^{M} \sum_{r \in \underline{\mathcal{S}}_N} 
		\mathbb{I}_{A_N^{(r,m)}}\beta^{\Theta_r} 
		\bigg(\sum_{l=1}^{\sigma_r(m)}\beta^l
		\left(h^{(m)}(\Psi^{(m)}_r + l) - 
		\gamma^{(m)}(\Psi^{(m)}_r)\right)\bigg) 
		\bigg) \\
		& \leq  \sum_{m=1}^{M} \sum_{r \in \underline{\mathcal{S}}_N} \mathfrak{E}\bigg(
		\mathbb{I}_{A_N^{(r,m)}}\beta^{\Theta_r} 
		\bigg(\sum_{l=1}^{\sigma_r(m)}\beta^l
		\left(h^{(m)}(\Psi^{(m)}_r + l) - 
		\gamma^{(m)}(\Psi^{(m)}_r)\right)\bigg) 
		\bigg)
		\\
		& \leq   \sum_{m=1}^{M} \sum_{r \in \underline{\mathcal{S}}_N} 
		\mathfrak{E}\bigg(
		\mathbb{I}_{A_N^{(r,m)}}\beta^{\Theta_r} 
		\mathfrak{E}_{\Psi_r}\bigg(\sum_{l=1}^{\sigma_r(m)}\beta^l
		\left(h^{(m)}(\Psi^{(m)}_r + l) - 
		\gamma^{(m)}(\Psi^{(m)}_r)\right) 
		\bigg)\bigg)  \\
		& =   \sum_{m=1}^{M} \sum_{r \in \underline{\mathcal{S}}_N} 
		\mathfrak{E}\bigg(
		\mathbb{I}_{A_N^{(r,m)}}\beta^{\Theta_r} 
		\mathcal{E}^{(m)}\bigg(\sum_{l=1}^{\sigma_r(m)}\beta^l
		\left(h^{(m)}(\Psi^{(m)}_r + l) - 
		\gamma^{(m)}(\Psi^{(m)}_r) \right) \bigg| \mathcal{F}^{(m)}_{\Psi^{(m)}_r}
		\bigg)\bigg) \\
		& =   \sum_{m=1}^{M} \sum_{r \in \underline{\mathcal{S}}_N} 
		\mathfrak{E}\left(0\right) \;\;\; = \;\; 0.
		\end{align*}
		We see that $V(N+1, \sigma, p)\leq 0$, and the desired result follows by induction.
	\end{proof}
\end{theorem}
\subsection{Step B.3: C-optimality}
In the previous subsections, we introduced an allocation problem when the prevailing process is offered as compensation (Definition \ref{Gittins' target function}). We also proved that the optimal value can be achieved by choosing a proper family of allocation time sequences (i.e. $\sigma$ as in Theorem \ref{negativity}).

The prevailing reward process $\Gamma$ for each bandit is non-decreasing, and the optimal allocation sequences $\sigma$ require us to make a new decision whenever the process $\Gamma$ increases. By exploiting this fact, together with the discount effect, we will see that it is preferable to play the bandit with the lowest value of $\Gamma$ first. In particular, we can establish the Robust Gittins index theorem, which we repeat for convenience of the reader.
\begin{theorem}[Theorem \ref{robust Gittins}: Robust Gittins theorem]
\label{proof: robust Gittins}
	Suppose that for each $m \in \mathcal{M}$, $(\mathcal{F}^{(m)}_t)_{t \geq 0}$ is generated by some underlying process $(\xi^{(m)}_t)_{t \geq 1}$. Let $\psi^{(m)}_n$ be the total number of trials of the $m$th bandit before the $n$th play of the system. i.e.  $\psi^{(m)}_n := 
	\sum_{k=0}^{n-1}\mathbb{I}(\rho^*_k = m)$ (given an allocation strategy $\rho^*$ up to time $n-1$). 
		
	Then the allocation strategy 
	$\rho^*$ given (recursively) by $$\rho^*_n := \min \Big\{m \in \mathcal{M} \;\; : \;\; m \in \argmin_{k} 
	\gamma^{(k)}(\psi^{(k)}_n) \Big\}$$
	is C-optimal (Definition \ref{C-optimal}) under $\mathfrak{E}$ for the cost $$g^\rho(n) = \beta^n h^{(\rho_{n-1})}(t^{\rho}_n) \quad \text{where} \quad t^{\rho}_n = \sum_{k=0}^{n-1}\mathbb{I}(\rho_k = \rho_{n-1}).$$
\begin{proof}
	Recall the definition of $\Psi^{(m)}_r$ in \eqref{decision filtration}. We can see that $\Psi_r$ determines the orthant filtration when $\tilde{\eta}_{n} = r$ where $(\tilde{\eta}_{n})$ is a recording sequence constructed from the time allocation sequence $\sigma$, i.e. when the $m$th bandit was run for $r^{(m)}$ times under the (optimal) allocation sequence $\sigma$. In particular, $\Psi^{(m)}_r$ corresponds to the number of trials on the $m$th bandit. 
	
	To explicitly define our choice sequence, we set
	$$p^*_n := \min \left\{m \in \mathcal{M} \;\; : \;\; m \in \argmin_{k} 
	\gamma^{(k)}(\Psi^{(k)}_r) \right\}  \;\;\; \text{on the event} \;\;\; \{ \tilde{\eta}_n = r\}.$$		
	As $\Psi^{(m)}_r$ is an $(\mathcal{F}^{(m)}_t)$-stopping time, $\gamma^{(m)}(\Psi^{(m)}_r)$ is well-defined and is $\mathcal{F}(\Psi^{(m)}_r)$-measurable. It also follows that
	\begin{align*}
	\Big(\{ \tilde{\eta}_n = r\}\cap\{p^*_n = m\}\Big) & = \{ \tilde{\eta}_n = r\}\cap \bigcap_{k=1}^m\left\{ \gamma^{(m)}(\Psi^{(m)}_r) < \gamma^{(k)}(\Psi^{(k)}_r)  \right\}\\
	&\qquad \cap \bigcap_{k=1}^M\left\{\gamma^{(k)}(\gamma^{(m)}(\Psi^{(m)}_r) \leq \Psi^{(k)}_r) \right\}\\
	&\in \mathcal{F}(\Psi_r)
	\end{align*}
	
	Hence, $p^*$ is a choice sequence for the allocation sequence $\sigma$ (Definition \ref{recording sequences}). Therefore, $(\sigma, p^*)$ is an admissible allocation strategy. 
	Moreover, observe that $\rho^*$ given in the statement of 
	this theorem is the simple form of the allocation strategy $(\sigma,p^*)$. 
	
	By Theorem \ref{negativity} and Theorem \ref{positivity},
	$$\mathfrak{E}\bigg(\sum_{n=1}^{\infty} 
	\beta^{n} \left(h^{(\rho^{*}_{n-1})}(t^*_n) - 
	\Gamma^{(\rho^*_{n-1})}(t^*_n)\right) 
	\bigg) = 0  \;\;\; \text{for} \;\;\; t^*_n := \sum_{k=0}^{n-1} 
	\mathbb{I}(\rho^*_k = 
	\rho^*_{n-1}). $$
	Theorem \ref{positivity} also implies that for any allocation strategy 
	$(\tau,p)$ (and thus for any simple form $\rho$),
	$$\mathfrak{E}\bigg(\sum_{n=1}^{\infty} 
	\beta^{n} \left(h^{(\rho_{n-1})}(t^{\rho}_n) - 
	\Gamma^{(\rho_{n-1})}(t^{\rho}_n)\right) 
	\bigg) \geq 0. $$
	
	Next, we will show that $n \mapsto \beta^n\Gamma^{(\rho_{n-1})}(t^{\rho}_n)$ is predictable with respect to our observed filtration. We recall that  $$\Gamma^{(m)}(t) = \max_{0 \leq \theta \leq t-1} \gamma(\theta) \;\;\; : \;\;\; t= 1, 2, ....$$		
	Now, observe that $t^{\rho}_n = \sum_{k=0}^{n-1} \mathbb{I}(\rho_k = \rho_{n-1}) = 1 + \eta^{(\rho_{n-1})}_{n-1}$ where $(\eta_n)$ is a recording sequence corresponding to a strategy $(\mathfrak{1}, \rho)$ and hence,
	
	$$\Gamma^{(\rho_{n-1})}(t^{\rho}_n) = \sum_{r \in {\mathcal{S}}_{n-1}}\mathbb{I}(\eta_{n-1} = r)\mathbb{I}(\rho_{n-1} = m)\Gamma^{(m)}(1 + r^{(m)})$$
	where ${\mathcal{S}}_N := \Big\{r \in \mathcal{S} \; : \; \sum_{m=1}^{M}r^{(m)} = N \Big\}$.
	
	Since $\Gamma^{(m)}(1 + r^{(m)})$ is $\mathcal{F}^{(m)}_{r^{(m)}}$-measurable, by the Doob--Dynkin lemma, there exists a measurable function $f^{(m)}_r: \mathbb{R}^{r^{(m)}} \to \mathbb{R}$ such that
	\begin{equation} 
	\label{Doob-Dynkin for Gamma}
	\Gamma^{(m)}(1 + r^{(m)}) = f^{(m)}_r(\xi^{(m)}_1, ..., \xi^{(m)}_{r^{(m)}}).
	\end{equation}
	Note that, while the process $\big(\xi^{(m)}_t\big)$ is defined on $(\Omega^{(m)}, \mathcal{F}^{(m)})$, we can extend it to $(\bar{\Omega}, \bar{\mathcal{F}})$ by considering an appropriate embedding.
	
	By substituting into \eqref{Doob-Dynkin for Gamma} in a similar way to \eqref{arranging function for filtration}, we can write $\Gamma^{(\rho_{n-1})}(t^{\rho}_n)$ as a (measurable) function of $\left( \left( \eta_k\right)_{0 \leq k \leq n-1}, \left( \rho_k\right)_{0 \leq k \leq n-1},  \left( \xi^\rho_k\right)_{0 \leq k \leq n-1} \right)$. By Definition \ref{recording sequences} and Remark \ref{observed vs decision filtration}, $\left( \eta_k\right)$ and $\left( \rho_k\right)$ are adapted to the observed filtration $\left(\mathcal{H}^\rho_n\right)_{n\ge 0}$. It follows that $\Gamma^{(\rho_{n-1})}(t^{\rho}_n)$ is $\mathcal{H}^\rho_{n-1}$-measurable. 
	
	Therefore, for each $\rho$, $C_0^\rho(n) := \Gamma^{(\rho_{n-1})}(t^\rho_n)$ defines a subcompensator at time $N = 0$, and for $\rho = \rho^*$, the cost is fully compensated with $V_0 = 0$.
	
		To construct a compensator for a subsequent time $N$, we consider `restarting' our system at an orthant time $r= (r^{(1)},...,r^{(M)}) \in {\mathcal{S}}_N \subseteq \mathcal{S}$ (as in Theorem \ref{equal filtration}). As $\mathcal{F}(r)$ describes the information from all bandits, this needs to be done carefully. Each of our single-bandit filtrations $(\mathcal{F}^{(m)}_t)_{t\ge 0}$ is generated by a discrete-time real-valued process, and $\mathcal{F}(r) = \bigotimes_{m}\mathcal{F}^{(m)}_{r^{(m)}}$, so the Doob--Dynkin lemma states that any $\mathcal{F}(r)$-measurable random variable can be written as a Borel function of the first $r$ observations. For concreteness, we denote these observations $\omega_r$.

	We proceed by freezing the value of $\omega_r$ and $\omega_u$ with $u \leq r$. Let $(\rho^{*, \omega_r}_{n})_{n\geq N}$ denote the minimum-Gittins-index strategy given by $\rho^{*}$ defined in the theorem when we restart our analysis at $r$ from a given $\omega_r$. We do not change the Gittins indices $\gamma$ when we fix $\omega_r$, so the corresponding $\Gamma$ processes satisfy
	$$\Gamma_{\omega_r}^{(m)}(t) = \max_{r^{(m)} \leq \theta \leq t-1} \gamma^{(m)}(\theta)  \leq \max_{u^{(m)} \leq \theta \leq t-1} \gamma^{(m)}(\theta) = \Gamma_{\omega_u}^{(m)}(t)  \quad \text{for all} \;\; u \leq r$$
	and are measurable with respect to $\omega_r$. As discussed in Remark \ref{time-consistency remark}, the optimal strategy $(\rho^{*,\omega_r}_n)_{n \geq N}$ coincides with the strategy $\rho^*$ (and is therefore also measurable with respect to $\omega_r$). 
	
	By repeating our earlier analysis, we see that, for each $\omega_r$, $ \Gamma_{\omega_r}^{(\rho^{*,\omega_r}_{n-1})}\leq  \Gamma_{\omega_u}^{(\rho^{*,\omega_r}_{n-1})}$. We can now unfreeze $\omega_r$ and $\omega_u$ and, summing over all possible scenarios, show that
	$$C^\rho_N(n) := \beta^n \max_{{\eta}^{(\rho_{n-1})}_N \leq \theta \leq t^\rho_n - 1} \gamma^{(\rho_n)}(\theta)$$
	is decreasing in $N$. By applying the same argument as earlier, we also have $n \mapsto C^\rho_N(n)$ is $\mathcal{H}^\rho_n$-predictable. Therefore, $(C^\rho_N(n))$ defines a subcompensator for strategy $\rho$ and fully compensates for $\rho = \rho^*$. We set $V_N = 0$ and observe \eqref{eq: compensation} is satisfied.
	
	Finally, as $t \mapsto \Gamma_{\omega_r}^{(m)}(t)$ is increasing for all $m$ and all $\omega_r$, and $\rho^*$ is a strategy where the lowest $\Gamma$ is chosen first, it 
	follows that for all $1 \leq N \leq \infty$, if $\rho \sim_N \rho^*$, then
	\begin{equation}
	\label{dynamic optimal}
	\sum_{n=N+1}^{L} 
	C^{\rho^*}_N(n) \leq \sum_{n=N+1}^{L} 
	C^{\rho}_N(n)  \qquad \text{for all } \; L \geq N+1.
	\end{equation}		
In particular, \eqref{C-optimal equation} is satisfied and therefore $\rho^*$ is C-optimal.
\end{proof}
\end{theorem}
\section{Proof of Other Relevant Results}
\label{Append: additional result}

\begin{definition}[F\"ollmer and Schied \cite{stoc_fin}]
	Let $(\Omega, \mathcal{G}, \mathbb{P})$ be a probability space and let $\mathcal{Y}$ be a family of $\mathcal{G}$-measurable random 
	variable. We say $Z$ is a $\mathcal{G}$-\emph{essential infimum} of 
	$\mathcal{Y}$ denoted by $Z = \mathcal{G}\text{-}\essinf Y$ if 
	\begin{enumerate}[(i)]
		\item $Z$ is $\mathcal{G}$-measurable.
		\item $Z \leq Y$ $\; \mathbb{P}-a.s.$ for all $Y \in \mathcal{Y}$.
		\item For $Z'$ such that $Z' \leq Y$ $\; 
		\mathbb{P}-a.s.$ for all $Y \in \mathcal{Y}$, we must have $Z' \leq Z$ 
		$\mathbb{P}-a.s.$.
	\end{enumerate}
	We also define a similar notion for $\mathcal{G}$-\emph{essential supremum}. We may omit $\mathcal{G}$ in front of $\essinf$ if the measurability of the family is obvious.
\end{definition}
\begin{theorem}[Existence of Essential infimum]
	\label{existence of essinf}
	The $\mathcal{G}$-essential infimum exists. 
	
	Suppose in addition that $\mathcal{Y}$ is directed downwards, that is for $Y,Y' \in \mathcal{Y}$, there exists $\tilde{Y} \in \mathcal{Y}$ such that ${\tilde{Y}} \leq \min(Y, Y')$. Then there exists a decreasing sequence $\left(Y_n \right)_{n \in \mathbb{N}} \subseteq \mathcal{Y} $  such that $Y_n \searrow \essinf Y$ $\; \mathbb{P}$-a.s. 
	
	The similar result also holds for $\mathcal{G}$-essential supremum.
	\begin{proof}
		See Theorem 1.3.40. in \cite{Stoc_Cal} or Theorem A.37 in \cite{stoc_fin}.
	\end{proof}
\end{theorem}
\begin{theorem}[Proof of Theorem \ref{Thm:Explicit_index}]
	\label{Thm:Explicit_index: proof}
	Let $\gamma(s)$ be the robust Gittins index (Definition \ref{Gittins index}) (with superscript $(m)$ omitted). Then
	$$\gamma(s) = \essinf_{\tau \in \mathcal{T}(s)} \; \esssup_{\mathbb{Q} \in \mathcal{Q}} \frac{\mathbb{E}^{\mathbb{Q}}\big( \sum_{t=1}^{\tau} \beta^{t} h(s+t) \big| \mathcal{F}_s\big)}{\mathbb{E}^{\mathbb{Q}}\big( \sum_{t=1}^{\tau} \beta^{t} \big| \mathcal{F}_s\big)} $$
	where $\mathcal{Q}$ is a family of probability measure induced in Theorem \ref{Thm:dynamic robust rep}.
\begin{proof}
	By Corollary \ref{cost_vs_fair charge} and Theorem \ref{Thm:dynamic robust 
	rep} (together with positive homogeneity), we have that, for any $\tau \in 
	\mathcal{T}(s)$,
	\begin{equation}
	\label{positive_supremum}
	\esssup_{\mathbb{Q} \in \mathcal{Q}}\mathbb{E}^\mathbb{Q}\bigg( \sum_{t=1}^{\tau} \beta^{t} \big(h(s+t) - \gamma(s)\big) \bigg| \mathcal{F}_s\bigg) \geq 0.
	\end{equation}
	
	By Lemma 11.19 (together with the construction of Theorem 11.22) in F\"ollmer and Schied \cite{stoc_fin}, the family	
	$\big\{\mathbb{E}^\mathbb{Q}\big( \sum_{t=1}^{\tau} \beta^{t} \big(h(s+t) - \gamma(s)\big) \big| \mathcal{F}_s\big) : \mathbb{Q} \in \mathcal{Q} \big\}$
	must be directed upwards. Hence, by Theorem \ref{existence of essinf}, we can find a family $(\mathbb{Q}_n) \subseteq \mathcal{Q}$ such that
	\begin{equation}
	\label{convergence_to_esssup}
	\mathbb{E}^{\mathbb{Q}_n}\bigg( \sum_{t=1}^{\tau} \beta^{t} \big(h(s+t) - \gamma(s)\big)\bigg) \xrightarrow[n \to \infty]{} \esssup_{\mathbb{Q} \in \mathcal{Q}}\mathbb{E}^\mathbb{Q}\bigg( \sum_{t=1}^{\tau} \beta^{t} \big(h(s+t) - \gamma(s)\big) \bigg| \mathcal{F}_s\bigg). 
	\end{equation}
	Let $\tilde{\Omega}$ be an event with probability one such that the followings hold.
	\begin{enumerate}
		\item \eqref{positive_supremum} and  \eqref{convergence_to_esssup} holds.
		\item For all $n \in \mathbb{N}$,
		\begin{equation*}
		\frac{\mathbb{E}^{\mathbb{Q}_n}\big( \sum_{t=1}^{\tau} \beta^{t} h(s+t) \big| \mathcal{F}_s\big)}{\mathbb{E}^{\mathbb{Q}_n}\big( \sum_{t=1}^{\tau} \beta^{t} \big| \mathcal{F}_s\big)} \leq  \esssup_{\mathbb{Q} \in \mathcal{Q}}\frac{\mathbb{E}^{\mathbb{Q}}\big( \sum_{t=1}^{\tau} \beta^{t} h(s+t) \big| \mathcal{F}_s\big)}{\mathbb{E}^{\mathbb{Q}}\big( \sum_{t=1}^{\tau} \beta^{t} \big| \mathcal{F}_s\big)},
		\end{equation*}
		
		$\mathbb{E}^{\mathbb{Q}_n}\big( \sum_{t=1}^{\tau} \beta^{t} \big| \mathcal{F}_s\big) \geq \beta,$ and
		\begin{align}
		\label{property of subsequence measure}
		& \mathbb{E}^{\mathbb{Q}_n}\bigg( \sum_{t=1}^{\tau} \beta^{t} \big(h(s+t) - \gamma(s)\big) \bigg| \mathcal{F}_s\bigg) \nonumber \\ & = \mathbb{E}^{\mathbb{Q}_n}\bigg( \sum_{t=1}^{\tau} \beta^{t} h(s+t) \bigg| \mathcal{F}_s\bigg) - \gamma(s) \mathbb{E}^{\mathbb{Q}_n}\bigg( \sum_{t=1}^{\tau} \beta^{t} \bigg| \mathcal{F}_s\bigg).
		\end{align}
	\end{enumerate}
	Fix $\omega \in \tilde{\Omega}$ and $\epsilon > 0$. By \eqref{positive_supremum} and  \eqref{convergence_to_esssup}, there exists $n \in \mathbb{N}$ such that
	$$ \mathbb{E}^{\mathbb{Q}_n}\bigg( \sum_{t=1}^{\tau} \beta^{t} \big(h(s+t) - \gamma(s)\big) \bigg| \mathcal{F}_s\bigg)(\omega) \geq -\epsilon.$$
	By \eqref{property of subsequence measure}, we can rearrange the inequality above and obtain
	\begin{align*}
	\gamma(s)(\omega)& \leq \frac{\mathbb{E}^{\mathbb{Q}_n}\big( \sum_{t=1}^{\tau} \beta^{t} h(s+t) \big| \mathcal{F}_s\big)}{\mathbb{E}^{\mathbb{Q}_n}\big( \sum_{t=1}^{\tau} \beta^{t} \big| \mathcal{F}_s\big)}(\omega) + \frac{\epsilon}{\mathbb{E}^{\mathbb{Q}_n}\big( \sum_{t=1}^{\tau} \beta^{t} \big| \mathcal{F}_s\big)(\omega)}\\
	& \leq \esssup_{\mathbb{Q} \in \mathcal{Q}}\frac{\mathbb{E}^{\mathbb{Q}}\big( \sum_{t=1}^{\tau} \beta^{t} h(s+t) \big| \mathcal{F}_s\big)}{\mathbb{E}^{\mathbb{Q}}\big( \sum_{t=1}^{\tau} \beta^{t} \big| \mathcal{F}_s\big)}(\omega) + \frac{\epsilon}{\beta}.
	\end{align*}
	As $\epsilon$ is arbitrary, it follows that on $\tilde{\Omega}$,
	$$ \gamma(s) \leq \esssup_{\mathbb{Q} \in \mathcal{Q}}\frac{\mathbb{E}^{\mathbb{Q}}\big( \sum_{t=1}^{\tau} \beta^{t} h(s+t) \big| \mathcal{F}_s\big)}{\mathbb{E}^{\mathbb{Q}}\big( \sum_{t=1}^{\tau} \beta^{t} \big| \mathcal{F}_s\big)},$$
	hence,
	$$ \gamma(s) \leq \essinf_{\tau \in \mathcal{T}(s)} \; \esssup_{\mathbb{Q} \in \mathcal{Q}}\frac{\mathbb{E}^{\mathbb{Q}}\big( \sum_{t=1}^{\tau} \beta^{t} h(s+t) \big| \mathcal{F}_s\big)}{\mathbb{E}^{\mathbb{Q}}\big( \sum_{t=1}^{\tau} \beta^{t} \big| \mathcal{F}_s\big)}.$$
	By Theorem \ref{optimal stopping} and Theorem \ref{Thm:dynamic robust rep}, we can find $\sigma := \sigma(s,\gamma(s)) \in \mathcal{T}(s)$ such that,
	\begin{equation*}
	\esssup_{\mathbb{Q} \in \mathcal{Q}}\mathbb{E}^\mathbb{Q}\bigg( \sum_{t=1}^{\sigma} \beta^{t} \big(h(s+t) - \gamma(s)\big) \bigg| \mathcal{F}_s\bigg) = 0,
	\end{equation*}
	hence, for all $\mathbb{Q} \in \mathcal{Q}$,
	$$\mathbb{E}^\mathbb{Q}\bigg( \sum_{t=1}^{\sigma} \beta^{t} \big(h(s+t) - \gamma(s)\big) \bigg| \mathcal{F}_s\bigg) \leq 0.$$
	Therefore,
	$$ \gamma(s) \geq \frac{\mathbb{E}^{\mathbb{Q}}\big( \sum_{t=1}^{\sigma} \beta^{t} h(s+t) \big| \mathcal{F}_s\big)}{\mathbb{E}^{\mathbb{Q}}\big( \sum_{t=1}^{\sigma} \beta^{t} \big| \mathcal{F}_s\big)},$$
	and we conclude
	\begin{align*}
	\gamma(s) & \geq \esssup_{\mathbb{Q} \in \mathcal{Q}} \frac{\mathbb{E}^{\mathbb{Q}}\big( \sum_{t=1}^{\sigma} \beta^{t} h(s+t) \big| \mathcal{F}_s\big)}{\mathbb{E}^{\mathbb{Q}}\big( \sum_{t=1}^{\sigma} \beta^{t} \big| \mathcal{F}_s\big)} \\
	& \geq \essinf_{\tau \in \mathcal{T}(s)} \; \esssup_{\mathbb{Q} \in \mathcal{Q}} \frac{\mathbb{E}^{\mathbb{Q}}\big( \sum_{t=1}^{\tau} \beta^{t} h(s+t) \big| \mathcal{F}_s\big)}{\mathbb{E}^{\mathbb{Q}}\big( \sum_{t=1}^{\tau} \beta^{t} \big| \mathcal{F}_s\big)}.
	\end{align*}
	This completes the proof.
\end{proof}
\end{theorem}
\begin{proposition}
	\label{property of joining expectation: proof}
	All properties described in Proposition \ref{property of joining expectation} hold.
\begin{proof}

	$(i)$ is straightforward to prove using the 
	definition directly, as in the case of $\mathcal{E}$. 
	
	By Theorem \ref{Thm:dynamic robust rep} and the tower property,
	\begin{eqnarray*}
		\esssup_{\mathbb{Q} \in 
			\mathcal{Q}}\mathbb{E}^{\mathbb{Q}}\big(X \big| 
		\mathcal{F}(S)\big) & = & \esssup_{\mathbb{Q} \in 
			\mathcal{Q}}\mathbb{E}^{\mathbb{Q}}\Big(\mathbb{E}^{\mathbb{Q}}\big(X
			 \big| 
		\mathcal{F}(S')\big) \Big| 
		\mathcal{F}(S)\Big) \\
		& \leq & \esssup_{\mathbb{Q} \in 
			\mathcal{Q}}\mathbb{E}^{\mathbb{Q}}\Big(\esssup_{\mathbb{Q}' \in 
			\mathcal{Q}}\mathbb{E}^{\mathbb{Q}'}\big(X \big| 
		\mathcal{F}(S')\big) \Big| 
		\mathcal{F}(S)\Big),
	\end{eqnarray*}	
	hence $(ii)$ follows. 
	
	For $\hat{Y}(\omega^{(1)},...,\omega^{(M)}) = 
	X^{(1)}(\omega_1) \times \cdots \times X^{(M)}(\omega_M) $, by Fubini's theorem,
	\begin{align*}
	&\esssup_{\bigotimes_{m=1}^M\mathbb{Q}^{(m)} \in 
	\mathcal{Q}}\mathbb{E}^{\bigotimes_{m=1}^M\mathbb{Q}^{(m)}}\bigg(\prod_{m=1}^M
	 X^{(m)}\bigg| 
	\mathcal{F}(S)\bigg) \\
& \qquad  =  \esssup_{\bigotimes_{m=1}^M\mathbb{Q}^{(m)} \in 
	\mathcal{Q}} \;\; \prod_{m=1}^M \mathbb{E}^{\mathbb{Q}^{(m)}}\big( 
	X^{(m)}\big| 
	\mathcal{F}^{(m)}_{S^{(m)}}\big).
	\end{align*}
	Then $(iii)$ and $(iv)$ follow by considering different choices of $X^{(m)}$.
\end{proof}
\end{proposition}
\begin{lemma}
	\label{existence of optimal stopping seq}
	Let $V_s : \mathcal{T}(s) \to L^\infty(\mathcal{F}_s)$ be a function such that for every $\tau, \sigma \in \mathcal{T}(s)$ and $A \in \mathcal{F}_s$, we have
	$$V_s(\tau\; \mathbb{I}_A + \sigma \;\mathbb{I}_{A^c}) =  V_s(\tau) \mathbb{I}_A + V_s(\sigma) \mathbb{I}_{A^c}. $$
	Then there exists a sequence $\tau_n \in \mathcal{T}(s)$ such that $V_s(\tau_n) \searrow \essinf_{\tau \in \mathcal{T}(s)}V_s(\tau)$ $\mathbb{P}$-a.s.. 
	\begin{proof}
		For $\tau, \sigma \in \mathcal{T}(s)$, we define $A := \left\{V_s(\tau) > V_s(\sigma)\right\} \in \mathcal{F}_s$. Then 
		$$\tilde{\tau} := \tau\; \mathbb{I}_A + \sigma \;\mathbb{I}_{A^c} \in \mathcal{T}(s).$$ 
		By assumption, $V_s(\tilde{\tau}) \geq V_s(\tau) \wedge V_s(\sigma)$; the result follows from Theorem \ref{existence of essinf}.
	\end{proof}
\end{lemma}
\begin{lemma}
	\label{double inf relation}
	Let $f: \mathcal{T}(s) \times L^\infty(\mathcal{F}_s) \to	
	L^\infty(\mathcal{F}_s) $ satisfy:
	\begin{enumerate}[(i)]
		\item For all $\tau \in \mathcal{T}(s)$, $f(\tau,0) \geq 0$ 
		$\mathbb{P}$-a.s. and $f(\tau,X) < 0$ for some $X \in 
		L^\infty(\mathcal{F}_s).$	
		\item There exists $L \in [0,\infty)$ such that, for every $X,Y \in 
		L^\infty(\mathcal{F}_s)$ with $X \geq Y $ 
		$\mathbb{P}$-a.s and $\tau \in \mathcal{T}(s)$, we have 
		$$0 \leq f(\tau,Y) - f(\tau,X) \leq L(X-Y) \;\;\; \mathbb{P}-a.s.$$
		\item  For all $A \in \mathcal{F}_s$, all $\tau, \sigma \in 
		\mathcal{T}(s)$ and all $X,Y 
		\in L^\infty(\mathcal{F}_s) $, we have
		$$f(\tau\mathbb{I}_A + \sigma \mathbb{I}_{A^c}, X \mathbb{I}_A + Y 
		\mathbb{I}_{A^c}) = f(\tau,X) \mathbb{I}_A + f(\sigma,Y) 
		\mathbb{I}_{A^c}.$$
	\end{enumerate}
	Define
	$$X^{*} := \essinf\Big\{X \in L^\infty(\mathcal{F}_s)  \;  : 
	\;\; 
	\essinf_{\tau \in \mathcal{T}(s)} f(\tau,X) \leq 0 \;\;\; 
	\mathbb{P}\text{-a.s.} \Big\}.$$
	Then $X^* \in L^\infty(\mathcal{F}_s) $ and
	$$	\essinf_{\tau \in \mathcal{T}(s)} f(\tau,X^*) = 0 \;\;\; 
	\mathbb{P}\text{-a.s.}$$
	Note: All essential infima in this lemma are taken among the
	$\mathcal{F}_s$-measurable functions.
	\begin{proof}
		Denote $$\mathcal{X} := \Big\{X \in L^\infty(\mathcal{F}_s)  \;  : 
		\;\; 
		\essinf_{\tau \in \mathcal{T}(s)} f(\tau,X) \leq 0 \;\;\; 
		\mathbb{P}\text{-a.s.} \Big\}. $$ 
		By (i), $\mathcal{X} \neq \phi$. For a fixed $X \in \mathcal{X}$, by Lemma \ref{existence of optimal stopping seq} there exists a sequence 
		$\tau_k \in 
		\mathcal{T}(s)$, such that $f(\tau_k, X) \searrow 
		\essinf_{\tau \in \mathcal{T}(s)}f(\tau, X) $ 
		$\mathbb{P}\text{-a.s.}$ Similarly, we can find a sequence $(\tau'_k)$ for $X' \in \mathcal{X}$.		
		
		Define a sequence 
		$\sigma_k := \tau_k \mathbb{I}_A + \tau'_k \mathbb{I}_{A^c}$ where $A = \{X 
		\leq X'\}$. Then
		\begin{align*}
			\essinf_{\tau \in \mathcal{T}(s)}f\left(\tau, \min(X,X')\right) & \leq f(\sigma_n, \min(X,X'))  \\
			&= f(\tau_k\mathbb{I}_A + \tau'_k \mathbb{I}_{A^c}, X \mathbb{I}_A + X' \mathbb{I}_{A^c}) \\
			& =  f(\tau_k, X) \mathbb{I}_A + f(\tau'_k, X') \mathbb{I}_{A^c} \\ 
			& \searrow  \essinf_{\tau \in \mathcal{T}(s)}f(\tau, X)\mathbb{I}_A + \essinf_{\tau \in \mathcal{T}(s)}f(\tau, X')\mathbb{I}_{A^c} \;\; \leq \;\; 0.
		\end{align*}
		Hence, $\mathcal{X}$ is downward directed. Therefore, by Theorem \ref{existence of essinf}, there 
		exists a sequence $\left(X_n\right)_{n \geq 0} \subseteq \mathcal{X}$ such that $X_n \searrow X^*$ $\mathbb{P}$-a.s. This implies that $X^*$ is almost surely bounded from above. 
		By monotonicity, as $f(\tau, 0) \geq 0$, it follows from strict 
		monotonicity that $f(\tau, -1) > 0$. Therefore, $-1$ is an 
		essential lower bound of $\mathcal{X}$, so $X^*$ is bounded below 
		by $-1$ and $X^* \in L^\infty(\mathcal{F}_s)$.
		
		For the final assertion, we will first show that $\essinf_{\tau \in \mathcal{T}(s)}f(\tau,X^*) \leq 0$ 
		$\mathbb{P}$-a.s. For each $n \in \mathbb{N}$, we can again find a 
		sequence $\tau^n_k$ such that
		$$f(\tau^n_k, X_n) \searrow \essinf_{\tau \in 
			\mathcal{T}(s)}f(\tau,X_n) \leq 0 \;\;\; \text{as} \;\;\; k \to \infty \;\;\; 
		\mathbb{P} \text{-a.s.}. $$
		By 
		condition 
		(ii), it follows that 
		\begin{eqnarray*}
			L(X_n - X^*) & \geq & f(\tau^n_k, X^*) - f(\tau^n_k, X_n) \\
			& \geq & \essinf_{\tau \in \mathcal{T}(s)} f(\tau, X^*) - f(\tau^n_k, X_n) \\
			& \nearrow & \essinf_{\tau \in \mathcal{T}(s)} f(\tau, X^*) - \essinf_{\tau \in 
				\mathcal{T}(s)} f(\tau, X_n) \;\;\;\;\;\; \text{as} \;\;\; {k \to \infty}\\
			& \geq & \essinf_{\tau \in \mathcal{T}(s)} f(\tau, X^*) .
		\end{eqnarray*}
		By taking $n \to \infty$, it follows that $\essinf_{\tau \in \mathcal{T}(s)} 
		f(\tau, X^*) \leq 0$.
		
		To finish the proof, it suffices to show that for all $\sigma \in 
		\mathcal{T}(s)$, we have $f(\sigma,X^*) \geq 0$. Fix $\sigma \in 
		\mathcal{T}(s)$. Define $F := -f(\sigma,X^*)$ and the event $B:= \left\{ F >  
		0\right\}$.	Let $\tau_k \in \mathcal{T}(s)$ be a sequence such that 
		$f(\tau_k, X^*) \searrow \essinf_{\tau \in \mathcal{T}(s)} 
		f(\tau, X^*) \leq 0$. We define $Y := X^* \mathbb{I}_{B^c} + \left(X^* - 
		\frac{F}{2L}\right) 
		\mathbb{I}_B \leq X^*$ and a sequence $\sigma_k = \tau_k \mathbb{I}_{B^c} + 
		\sigma \mathbb{I}_B$. Then
		\begin{align*}
			\essinf_{\tau \in \mathcal{T}(s)}f(\tau,Y) & \leq  f(\sigma_k,Y) \\
			&= f(\tau_k,X^*)\mathbb{I}_{B^c} + 
			f\Big(\sigma,X^*-\frac{F}{2L}\Big)\mathbb{I}_{B}\\
			& =  f(\tau_k,X^*)\mathbb{I}_{B^c} + 
			\bigg(\Big(f\Big(\sigma,X^*-\frac{F}{2L}\Big) - 
			f\left(\sigma,X^*\right)\Big) - F\bigg)\mathbb{I}_{B} \\
			& \leq  f(\tau_k,X^*)\mathbb{I}_{B^c} + \left(L\left(X^* - \left(X^* 
			-\frac{F}{2L} \right)\right) - F\right)\mathbb{I}_B \\
			& =  f(\tau_k,X^*)\mathbb{I}_{B^c} - \frac{F}{2}\mathbb{I}_B \\
			&\leq f(\tau_k,X^*)\mathbb{I}_{B^c} \;\;\searrow  \;\; \essinf_{\tau \in \mathcal{T}(s)} 
			f(\tau, X^*)\mathbb{I}_{B^c} \quad \leq \quad  0.
		\end{align*}
		Hence, $Y \in \mathcal{X}$. By minimality of $X^*$, it must follow that $B$ is 
		a $\mathbb{P}$-null set.
	\end{proof}
\end{lemma}

\begin{theorem}
	\label{existence of optimal stopping time}
	There exists $\tau^* \in \mathcal{T}(s)$ such that 
	$V_s(\tau^*,\lambda) = \essinf_{\tau \in \mathcal{T}(s)} V_s(\tau,\lambda)$.
	\begin{proof}
				 We write $\mathcal{F}^s_n := \mathcal{F}_{s+n}$ and define the 
				 processes
		$$Y_0 := \frac{(C+1)-\lambda}{1-\beta}, \quad Y_n := 
		\sum_{t=s+1}^{s+n}\beta^t \left( h(t) - \lambda \right) \;\;\; 
		\text{and} \;\;\;  Z_n := \essinf_{\substack{\tau \geq n}} 
		\mathcal{E}\big(Y_\tau \big| \mathcal{F}^{s}_n \big) $$
		where $\tau$ is considered over the space of all stopping times and $C$ 
		is an upper bound given in Assumption \ref{assumption for cost process}.
		
		Define $\tau^* := \inf\{n \geq 0: Y_n = Z_n\}$.
		
		By robust representation theorem (Theorem \ref{Thm:dynamic robust rep}), we can represent $\mathcal{E}$ as an essential supremum over a familiy of probability measures which satisfy the law of iteration \cite[Equation 4]{nonlinear_snell_envelope}.  It then follows Riedel \cite[Theorem 3]{nonlinear_snell_envelope} that if  $\tau^*<\infty$ $\mathbb{P}$-a.s., then $\tau^*$ is an optimal solution.
		
		Hence, to prove the required result, it suffices to prove that $\tau^* 
		\in \mathcal{T}(s)$, $Z_0 = \essinf_{\tau \in \mathcal{T}(s)} 
		V_s(\tau,\lambda)$ and  $\tau^*<\infty$ $\mathbb{P}$-a.s. 
		
		It is clear that we never stop at time $0$. Therefore, $\tau^* \in 
		\mathcal{T}(s)$ and $Z_0 = \essinf_{\tau \in \mathcal{T}(s)} 
		V_s(\tau,\lambda)$.
		
		On an event $\omega$ such that $\lambda(\omega) < C$ and $h(t)(\omega) \to C$, we can find $N(\omega)$ sufficiently large such that $\sum_{t=N+1}^{n}\beta^t \left( h(s+t) - \lambda \right)(\omega) > 0$ for all $n$. In particular, we have $\tau^*(\omega) \leq N(\omega) < \infty$. Therefore, we have $\tau^* < \infty$ $\mathbb{P}$-a.s. and thus it must be optimal.
	\end{proof}
\end{theorem}

\begin{lemma}
	\label{construction simple form}
	Let $(\rho_n)_{0 \leq n \leq L-1}$ be a simple form of $(\tau, p)$. Then the sequence $(\rho_n)_{0 \leq n \leq L-1}$ can be expressed recursively by the following relation. Set $\rho_0 = p_0$ and define
		\begin{equation}
		\label{simple form eqn}
		\rho_n = \begin{cases}
		\rho_{n-1} & \text{if} \; \sum_{m=1}^{M}\pi^{(m)}_n
		\geq n, \\
		p_{\psi_n} & \text{if} \;
		\sum_{m=1}^{M}\pi^{(m)}_n = n-1,
		\end{cases} \;\;\; \text{where} \;\;\; \pi^{(m)}_n :=\sum_{k=0}^{\hat{F}^{(m)}_n-1} \tau^{(m)}_k,
		\end{equation} 
		and where
		$$ \psi_n =  \sum_{m=1}^{M}\hat{F}^{(m)}_n  \;\;\; \text{and} \;\;\; \hat{F}^{(m)}_n := \min\left\{f \geq 0 \; : \; \sum_{k=0}^{n-1}\mathbb{I}(\rho_k = m) \leq \sum_{k=0}^{f-1}\tau^{(m)}_k\right\}.$$
		\end{lemma}
		\begin{proof}
		To see this, we view $\hat{F}^{(m)}_n$ as the number of runs (under $(\tau,p)$) of the $m$th bandit before making the $n$th play. We then consider $\pi^{(m)}_n$ as the total number of trials in the $m$th bandit required to complete the $\hat{F}^{(m)}_n$th run.
		
		If the run is not yet completed before the $n$th play, we continue to play the same machine. (i.e. we define $\rho_n = \rho_{n-1}$). If the run is completed in the $(n-1)$th play, we make a new decision in the $n$th play based on the choice sequence $p$. In that case, we have already made $\psi_n =  \sum_{m=1}^{M}\hat{F}^{(m)}_n$ decisions before the $n$th play. By our convention to start at $\rho_0$, the decision of the $n$th play is given by $p_{\psi_n}$.
		\end{proof}
\begin{definition} 
	\label{decision filtration G}
	Given a simple form choice sequence $\rho$, we define the \emph{decision filtration} induced by  $\rho$ by
\begin{equation}
	\mathcal{G}^\rho_n := \left\{A \in \mathcal{F}(T) \;\; : \;\; A \cap \{\eta_{n} = r\} \in \mathcal{F}(r) \;\right\}
\end{equation}
where $\eta$ is the recording sequence corresponding to $(\mathfrak{1}, \rho)$.
	
\end{definition}
\begin{lemma}
	\label{chain of decision filtration}
	The sequence of $\sigma$-algebras $\left(\mathcal{G}^\rho_n  \right)$ given in Definition \ref{decision filtration G} forms a filtration, i.e.
	$\mathcal{G}^{\rho}_{n} \subseteq \mathcal{G}^{\rho}_{n+1}$.
	\begin{proof}
		Suppose that $A \in \mathcal{G}^{\rho}_{n}$, then for $r \in \mathcal{S}$ and 
		$m \in \{0\} \cup \mathcal{M}$,
		$$A \cap \{\eta_{n} = r-e^{(m)}\} \in \mathcal{F}(r-e^{(m)}) \subseteq 
		\mathcal{F}(r).$$
		Moreover, by definition, $\{\rho_n = m \} \in \mathcal{G}^{\rho}_n$, it 
		follows that
		\begin{align*}A \cap \{\eta_{n+1} = r\} &= \bigcup_{m=0}^M \big(A \cap 
		\{\eta_{n} 
		= r-e^{(m)}\} \big) \cap \big( \{\rho_n = m \} \cap \{\eta_{n} 
		= r-e^{(m)}\} \big) \\&\in \mathcal{F}(r). 
		\end{align*}
		Hence, $A \in \mathcal{G}^{\rho}_{n+1}$.
	\end{proof}
\end{lemma}
\begin{lemma}
	\label{subset_1}
	With $(\xi^\rho_n)$ as in Remark \ref{observed vs decision filtration} and $(\mathcal{G}^\rho_n)$ in Definition \ref{decision filtration G}, $\xi^\rho_n$ is $\mathcal{G}^{\rho}_n$-measurable. In particular, $\mathcal{H}^{\rho}_n 
	:= \sigma\left(\xi^\rho_1, ..., \xi^\rho_n\right) \subseteq \mathcal{G}^{\rho}_n.$
	\begin{proof}
		By Lemma \ref{chain of decision filtration}, $\rho_{n-1}$  is 
		$\mathcal{G}^{\rho}_n$-measurable. Hence, for a fixed $B \in \mathcal{B}$, 
		we have
		$$
		\left\{\xi^\rho_n \in B\right\} \cap \left\{ \eta_n = r \right\} 
		=  \big\{\xi^{(m)}_{r^{(m)}} \in B\big\} \cap \underbrace{\left\{ 
			\eta_n = r 
			\right\}  \cap \left\{ \rho_{n-1} = m 
			\right\}}_{\in \mathcal{F}(r) \;\; \text{as} \;\; \left\{ 
			\rho_{n-1} = m 
			\right\} \; \in \; \mathcal{G}^{\rho}_n }  \in 
		\mathcal{F}(r). $$                
		Therefore,         $\xi^\rho_n$ is $\mathcal{G}^{\rho}_n$-measurable and the last 
		assertion follows from Lemma  \ref{chain of decision filtration}.
	\end{proof}
\end{lemma}
\begin{theorem}
	\label{equal filtration} 	With $(\mathcal{H}^\rho_n)_{n\ge0}$ as in Remark \ref{observed vs decision filtration} and $(\mathcal{G}^\rho_n  )_{n\ge0}$ in Definition \ref{decision filtration G},
	$\mathcal{G}^{\rho}_n = \mathcal{H}^{\rho}_n$.
	\begin{proof}
		We will prove this result by induction. It is clear that $\mathcal{G}^{\rho}_0 
		= \mathcal{H}^{\rho}_0$. We now assume that $\mathcal{G}^{\rho}_{n-1} = 
		\mathcal{H}^{\rho}_{n-1}$. 
		
		By Lemma \ref{subset_1}, it suffices to show that $\mathcal{G}^{\rho}_n 
		\subseteq \mathcal{H}^{\rho}_n$. Recall from 
        Definition \ref{recording sequences} that for $n \leq L$, the recording sequence $\eta_{n}$ takes values in 
		$${\mathcal{S}}_n := \Big\{r \in 
		\mathcal{S} \; : \; \sum_{m=1}^M r^{(m)} 
		= n \;\;\; \Big\}.$$
		
		For a fixed $r \in {\mathcal{S}}_N$, we define $\mathcal{P}_r$ to be the space of 
		sequences 
		of length $\sum_{m=1}^M r^{(m)} $ with values in $\mathcal{M}$ with exactly $r^{(m)}$ 
		replications of $m \in \mathcal{M}$. (e.g. for $\mathcal{M} =\{1,2\}$,  
		$\mathcal{P}_{(3,1)} = 
		\{(1,1,1,2), (1,1,2,1), (1,2,1,1), (2,1,1,1)\}$).
		
		For $A \in \mathcal{G}^{\rho}_n$ and $\pi \in \mathcal{P}_r$, 
		$$A \cap \big\{(\rho_k)_{0 \leq k \leq n-1} = \pi \big\} 
		= A \cap \big\{(\rho_k)_{0 \leq k \leq n-1} = \pi \big\} 
		\cap \{\eta_n = r \} \in \mathcal{F}(r).$$
		By the Doob--Dynkin lemma, there exists a measurable function $f: \mathbb{R}^n \to \mathbb{R}$ such that
		\begin{eqnarray}
		\label{arranging function for filtration}
			\mathbb{I}_{A \cap \left\{\left(\rho_k\right)_{0 \leq k \leq n-1} = 
				\pi \right\} }& = & f \Big(\big(\xi^{(m)}_t\big)_{1 \leq 
				t \leq 
				r^{(m)}, m \in \mathcal{M} }\Big) \nonumber
			\\ & = &
			f \Big(\big(\xi^{(m)}_t\big)_{1 \leq 
				t \leq 
				r^{(m)}, m \in \mathcal{M} }\Big)        \mathbb{I}_{ 
				\left\{\left(\rho_k\right)_{0 \leq k \leq n-1} = 
				\pi \right\} } \nonumber \\
			& = & {f}_\pi\Big(\left(\xi^\rho_k\right)_{1 \leq k \leq n} 
			\Big)         
			\mathbb{I}_{ 
				\left\{\left(\rho_k\right)_{0 \leq k \leq n-1} = 
				\pi \right\} } 
		\end{eqnarray}
		for some measurable function ${f}_{\pi}$ defined by reordering the 
		input of $f$ by $\pi$. 
		
		By the inductive hypothesis, $\mathbb{I}_{ 
			\left\{\left(\rho_k\right)_{0 \leq k \leq n-1} = 
			\pi \right\} }$ is $\mathcal{G}^{\rho}_{n-1}$-measurable and thus 
		$\mathcal{H}^{\rho}_{n-1}$-measurable. Hence, the RHS of \eqref{arranging function for filtration} is 
		$\mathcal{H}^{\rho}_{n}$-measurable. Therefore,
		$\mathbb{I}_A = \sum_{r \in {\mathcal{S}}_N} \sum_{\pi \in 
			\mathcal{P}_r}\mathbb{I}_{A \cap \left\{\left(\rho_k\right)_{0 \leq k 
				\leq n-1} =\pi\right\} }$ is $\mathcal{H}^{\rho}_{n}$-measurable, i.e. $A \in 
		\mathcal{H}^{\rho}_n$, which completes the proof.                
	\end{proof}
\end{theorem}
\section{Numerical Algorithm to Estimate Robust Gittins index}\label{app:numerical}
\label{Surface Algorithm}
To approximate the value of $\gamma_{k,\beta,T}\left(p,1/\sqrt{n}\right)$ in the setting of Section \ref{sec:numerics}, we proceed following the rough recipe below. (The code used is available on request.)

Fix a grid of values for $\gamma \in G \subseteq [0,1]$.
 \begin{enumerate}
	\item Set $V^\gamma_T\Big(p,\frac{1}{\sqrt{n+T}}\Big) = 0$.
	\item Assume that we know $V^\gamma_{t+1}$. We evaluate the backward recursion
	\begin{align*}
	&V^\gamma_{t}\Big(p,\frac{1}{\sqrt{n+t}}\Big) \\ \;\;\;&= \min\Big\{0,
	\mathcal{E}_{(t)}\Big( (h(t) - \gamma) + \beta 
	V^\gamma_{t+1}\Big(p_{t+1},\frac{1}{\sqrt{n+(t+1)}}\Big) 
	\Big)\bigg|_{(p_t, n_t) = (p, n+t)}\Big\}
	\end{align*}

	This is done by considering the discrete values of $p$ and using linear interpolation over $[0,1]$.
	\item Using these iterates, determine the initial value function
	$$U\Big(\gamma, p,\frac{1}{\sqrt{n}}, T\Big) = \mathcal{E}_{(0)}\Big( (h(1) - 
	\gamma) + \beta V^\gamma_{1}\Big(p_{1},\frac{1}{\sqrt{n+1}}\Big) 
	\Big)\bigg|_{(p_0, n_0) = (p, n)}.$$
By Snell's envelope argument,
 $$U\Big(\gamma, p,\frac{1}{\sqrt{n}}, T\Big) = 
 \essinf_{\tau \in \mathcal{T}(s)} \mathcal{E}\bigg(\sum_{t=1}^{\tau}\beta^t 
 \left( h(t) - \gamma \right) \bigg| \mathcal{F}_0 \bigg)\bigg|_{(p_0, n_0) = 
 (p, 
 n)}. $$
\item Repeat step 1-3 for all $\gamma \in G$.
\item Calculate $\gamma_{k,\beta,T}(p,1/\sqrt{n})$ for a fixed $(p,1/\sqrt{n})$ by
$$\gamma_{k,\beta,T}\Big(p,\frac{1}{\sqrt{n}}\Big) = \min\Big\{\gamma \in G : U\Big(\gamma, p,\frac{1}{\sqrt{n}}\Big) \leq 0 \Big\}$$
\item Repeat the previous step to compute $\gamma_{k,\beta,T}(p,1/\sqrt{n})$ for other values of $(p,1/\sqrt{n})$ and obtain the surface by linear interpolation.
\end{enumerate} 
\end{document}